\documentclass[11pt]{amsart} 
\usepackage{packages}

\title[Confluence in quantum $K$-theory and $q$-oscillatory integrals]
{Confluence in quantum $K$-theory of weak Fano manifolds and
  $q$-oscillatory integrals for toric manifolds}

\author{Todor Milanov and Alexis Roquefeuil}

\address{Kavli IPMU (WPI), UTIAS, The University of Tokyo, Kashiwa, Chiba 277-8583, Japan}
\email{todor.milanov@ipmu.jp}
\address{Kavli IPMU (WPI), UTIAS, The University of Tokyo, Kashiwa, Chiba 277-8583, Japan}
\email{alexis.roquefeuil@ipmu.jp}
\date{\today}
\thanks{{\em 2020 Math. Subj. Class.} 14N35, 14J33, 39A45 (Primary), 33D05, 35Q53 (Secondary)}
\thanks{
{\em Key words and phrases:} Frobenius structures,
K-theoretic Gromov--Witten invariants, q-difference equations, q-gamma function}
\thanks{
	Correspondence to be sent to: alexis.roquefeuil@ipmu
}
\thanks{
	Declarations of interest: none.
}

\begin{document}

\begin{abstract}
For a smooth projective variety whose anti-canonical bundle is nef, we prove confluence of the small $K$-theoretic $J$-function, i.e., after rescaling appropriately the Novikov variables, the small $K$-theoretic $J$-function has a limit when $q\to 1$, which coincides with the small cohomological $J$-function. 
Furthermore, in the case of a Fano toric manifold of Picard rank 2, we prove the $K$-theoretic version of an identity due to Iritani that compares the $I$-function of the toric manifold and the oscillatory integral of the toric mirror. In particular, our confluence result yields a new proof of Iritani's identity in the case of a toric manifold of Picard rank 2. 
\end{abstract}

\maketitle

\setcounter{tocdepth}{2}
\tableofcontents

\section{Introduction}

\subsection{Foreword}

Let $X$ be a smooth complex projective variety and let
$K(X)=K^0(X;\mathbb{C})$ be the Grothendieck group of topological 
vector bundles on $X$. For simplicity, we will assume that $H^{\rm
  odd}(X;\CC)=0$. In particular, the Chern character map
$\operatorname{ch}: K^0(X)\to H(X;\CC)$ is a ring isomorphism. Let us
denote by $X_{0,n,d}$ the proper moduli stack of genus-0 stable maps 
of degree $d\in H_2(X;\mathbb{Z})$ with $n$ marked points. The
operation that assigns to a point in the moduli space the 
cotangent line at the $i$-th marked point is functorial and it gives
rise to a line bundle $L_i \to X_{0,n,d}$, while evaluation at the $i$-th marked
point gives rise to a map of Deligne--Mumford stacks
$\operatorname{ev}_i:X_{0,n,d}\to X$ known as the evaluation map. 
Let $E_1,\dots,E_n\in K(X)$, then following Givental and Y. P. Lee (see
\cite{Gi1, Lee1}) we introduce the $K$-theoretic Gromov--Witten (GW)
invariants.
\begin{definition}
The $K$-theoretic Gromov--Witten invariants of $X$ are given by
\ben
\langle E_1 L_1^{k_1},\dots,E_n L_n^{k_n}\rangle_{g,n,d} = 
\chi\Big(\mathcal{O}_{\rm virt}\otimes  
\operatorname{ev}^*_1(E_1) L_1^{k_1} \cdots \operatorname{ev}^*_n(E_n) L_n^{k_n} \Big) \in \mathbb{Z},
\een
where $\chi(\mathcal{F})$ denotes the holomorphic Euler characteristic
of $\mathcal{F}$ and $\mathcal{O}_{\rm virt}$ is the so called {\em virtual structure
  sheaf} (see \cite{Lee1}). 
\end{definition}

Let us fix a set $P_1,\dots,P_r$ of ample line bundles, s.t.,
$p_i=c_1(P_i)$ form a $\ZZ$-basis of
$H^2(X;\mathbb{Z})\cap H^{1,1}(X;\CC)\cong NS^1(X)$. If
$d\in H_2(X;\mathbb{Z})$ then we define
\ben
Q^d:=Q_1^{\langle p_1,d\rangle}\cdots Q_r^{\langle p_r,d\rangle},
\een
where $Q_1,\dots,Q_r$ are formal variables known as the {\em Novikov
  variables}. Let us assign to each $Q_i$ degree $m_i\in \ZZ$ defined
by the identity $c_1(T_X) = \sum_{i=1}^r m_i p_i$. For $t\in K(X)$ put
\beq\label{KGW-series}
\langle E_1 L_1^{k_1},\dots,E_n L_n^{k_n}\rangle_{g,n}(t) =
\sum_{d} \sum_{\ell=0}^\infty 
\frac{Q^d}{\ell !} \langle E_1 L_1^{k_1},\dots,E_n
L_n^{k_n},t,\dots,t\rangle_{g,n+\ell,d}.
\eeq
Let us fix a basis $\{\Phi_i\}_{i=1}^N \subset K(X)$ and denote by
$\{\Phi^i\}_{i=1}^N$ the dual basis with respect to the Euler pairing
\ben
g_{ij}:=\chi(\Phi_i\otimes \Phi_j) =  
\int_X \operatorname{ch}(\Phi_i) \operatorname{ch}(\Phi_j) \operatorname{Td}(X).  
\een
Following Givental, we introduce
\begin{definition}[\cite{IMT}, Definition 2.4]\label{def:intro_k_theoretical_j_function}
  The small $K$-theoretic $J$-function of the variety $X$ is the formal power series
  \ben
  J(q,Q)=1-q + \sum_{d\in \operatorname{Eff}(X)} \sum_{i=1}^N
  \Big\langle \frac{\Phi_i}{1-qL_1} \Big\rangle_{0,1,d} \,
  Q^d \Phi^i
  \in
  K(X)(q)[\![Q]\!].
  \een
\end{definition}
Let us recall also the definition of the cohomological GW invariants. Let us fix bases
$\{\phi_i\}_{i=1}^N$ and $\{\phi^i\}_{i=1}^N$ of $H(X;\CC)$ dual to each other with
respect to the Poincare pairing, that is,
\ben
(\phi_i, \phi^j):= \int_{[X]} \phi_i \cup \phi^j = \delta_{i,j},\quad
1\leq i,j\leq N.
\een
In fact, we choose $\phi_i:=\operatorname{ch}(\Phi_i)$ and 
$\phi^i := \operatorname{ch}(\Phi^i)\operatorname{td}(X)$. 
The GW invariants are defined by
\ben
\langle
\phi_{i_1}\psi_1^{k_1},\dots,
\phi_{i_n}\psi_n^{k_n}
\rangle_{g,n,d} =
\int_{[X_{g,n,d}]^{\rm virt}}
\prod_{s=1}^N \operatorname{ev}_s^*(\phi_{i_s}) \, \psi_s^{k_s},
\een
where $\operatorname{ev}_s: X_{g,n,d}\to X$ is the evaluation map at
the $s$-th marked point, $\psi_s=c_1(L_s)$, and $[X_{g,n,d}]^{\rm
  virt}$ is the virtual fundamental cycle constructed in \cite{BF_intnormalcone}. 
The cohomological $J$-function is defined by
\ben
J^{\rm coh}(z,Q)=-z+
\sum_{d\in \operatorname{Eff}(X)} \sum_{i=1}^N
\Big\langle \frac{\phi_i}{-z-\psi} \Big\rangle_{0,1,d} \,
Q^d \phi^i.
\een
The parameter $z$ in $J^{\rm coh}(z,Q) $ is in fact redundant due to
the homogeneity properties of the $J$-function. Namely, let us define
the degree operator
\ben
\xymatrix{\operatorname{deg}: H(X;\CC)\ar[r] &  H(X;\CC)},
\een
which for a homogeneous element $\phi\in H^{2k}(X;\CC)$ is defined by
$\operatorname{deg}(\phi)=k \phi$. Using the formula for the dimension
of the virtual fundamental cycle we get
\begin{equation}\label{eqn:intro_cohom_J_recover_z}
J^{\rm coh}(z,Q_1,\dots,Q_r) = z^{1-\operatorname{deg}}
J^{\rm coh}(1,z^{-m_1}Q_1,\dots, z^{-m_r} Q_r).
\end{equation}

One way to compare cohomological and $K$-theoretical Gromov--Witten
invariants is to compare their $J$-functions as solutions of their
respective functional equations. In more details, one can use
confluence of $q$-difference equations to obtain the cohomological
$J$-function as a limit of the $K$-theoretical one. This has been
studied for projective spaces in \cite{roquefeuil_thesis} and for the
quintic threefold in \cite{Wen_QK_quintic_threefold}. Recall that a
line bundle $L$ is said to be {\em nef} if it has a non-negative degree on
all complex curves in $X$, that is, if $f:C\to X$ is a holomorphic map
from a complex curve $C$ to $X$, then $\int_{[C]} c_1(f^*L)\geq 0$. 
Our first main goal will be to prove the following confluence result.
\begin{theorem*}[Theorem \ref{t1}]
  If $X$ is a smooth projective variety, such that, the anti-canonical
  bundle $K_X^\vee$ is nef, then the limit
  \ben
  \lim_{q\to 1}
  (q-1)^{\operatorname{deg}-1}\, 
  \operatorname{ch}\Big(
  J(q, (q-1)^{m_1} Q_1,\dots, (q-1)^{m_r} Q_r)
  \Big)
  \een
  exists and it coincides with $J^{\rm coh}(1,Q)$. 
\end{theorem*}
Let us point out that the title of our paper is a bit misleading,
because weak Fano manifold $X$ means that $K_X^\vee$ is nef and {\em big}, while
our theorem does not require the condition that $K_X^\vee$ is big.
In particular, our theorem applies to all Fano and Calabi--Yau manifolds.
The proof of Theorem \ref{t1} is based on the reconstruction
result of Givental--Tonita for the $K$-theoretic $J$-function in terms of
cohomological GW invariants. More precisely, in the case of the 
small $J$-function, the recursion procedure outlined in the proof of
Proposition 4 in \cite{GT} takes a very elegant form. Using the Fano
condition and the formula for the dimension of the virtual fundamental
cycle, the statement of Theorem \ref{t1} follows quite easily.

Our next goal is motivated by the problem of comparing the $K$-theoretic $I$-function of a toric
manifold $X$ with the corresponding oscillatory integral defined through
the toric mirror of $X$. The $K$-theoretic $I$-function was introduced by
Givental in
\cite{Givental:QK_fixed_point_localization,Givental:perm_toric_q_hypergeometric}. More
precisely, using fixed-point localization techniques, Givental was
able to prove that a certain $q$-hypergeometric series
$I_X^{K-\textnormal{th}}$, called the {\em small $K$-theoretic $I$-function}
(see Definition \ref{def:gamma_k_th_I_function}), belongs to the
permutation-equivariant $K$-theoretic Lagrangian cone of $X$. 
Let us point out that in the case of a Fano toric manifold, up to some
simple normalization factor, $I_X^{K-\textnormal{th}}$  coincides with the small $J$-function of
$X$. Moreover, the small $J$-function in permutation equivariant
$K$-theoretical GW theory coincides with the non-equivariant
one, i.e., with the small $J$-function used in the current paper.  
On the other hand, from the toric data of $X$ one can construct a
$K$-theoretic version of the toric mirror model (see
\cite{Givental:perm_mirror}). Namely, following Givental, let us consider the following
family of functions:
\[
    \begin{tikzcd}
      Y := \left( \mathbb{C}^* \right)^n
        \arrow[r,"W^{K-\textnormal{th}}"]
        \arrow[d,"\pi"]
      &
      \mathbb{C}
      \\
      B := \left( \mathbb{C}^* \right)^r
    \end{tikzcd},
  \]
where we denote by $x_1, \dots, x_n$ the standard coordinates on $Y$,
$Q_1, \dots, Q_r$ the standard coordinates on $B$, and the maps
$W^{K-\textnormal{th}}$ and $\pi$ are given by 
  \begin{align*}
    W^{K-\textnormal{th}}(x_1,\dots,x_n)
    &:=
    \sum_{j=1}^n \sum_{l > 0}
    \frac{x_j^l}{l(1-q^l)}
    \in \mathbb{C},
    \\
    \pi(x_1, \dots, x_n)
    &:=
    \left(
      \cdots,
        \prod_{j=1}^n x_i^{m_{ij}}
      ,\cdots
    \right)
    \in B.
  \end{align*}
Let us recall first the case of quantum cohomology, which was
investigated in details in \cite{Iritani:gamma_structure}. Put 
\[
\mathcal{I}^{\textnormal{coh}}(z,Q)
    :=
    \int_{\Gamma_\mathbb{R}}
    \exp \left(
      W^{\textnormal{coh}}_{|\pi^{-1}(Q)}(x_1,\dots,x_n)
    \right)
    \omega_{\pi^{-1}(Q)},
\]
where 
\[
\omega_{\pi^{-1}(Q)}:=
\frac{
d\log x_1 \wedge\cdots \wedge d\log x_n}{ 
d\log Q_1\wedge \cdots \wedge d\log Q_r},
\] 
is interpreted naturally as a holomorphic volume form on the fiber
$\pi^{-1}(Q)$ and $\Gamma_\mathbb{R}$
is the real Lefschetz thimble (see Remark
\ref{rmk:gamma_cohomology_real_cycle}).
The comparison result goes as follows.
\begin{theorem*}[\cite{Iritani:gamma_structure}, see Theorem
  \ref{thm:gamma_cohomology_oscillatory_i_function}]
If $X$ is a Fano toric manifold,
then the cohomological oscillatory integral
$\mathcal{I}^{\textnormal{coh}}$ and the cohomological $I$-function
(see Definition \ref{def:gamma_cohomology_I_function}) are related by the identity 
  \[
    \mathcal{I}^{\textnormal{coh}}(z,Q)
    =
    \int_{\left[X\right]}
    \widehat{\Gamma}(TX)\cup z^\rho\, z^{\operatorname{deg}}\, 
    I^{\textnormal{coh}}(z,Q),
  \]
where $\rho=c_1(TX)\cup$ is the operator of cup product multiplication
by $c_1(TX)$, $\int_{\left[X\right]}$ denotes the intersection product
by $[X] \in H_*(X;\mathbb{C})$, and $\widehat{\Gamma}(TX)$ is a
cohomological characteristic class given by 
  \[
    \widehat{\Gamma}(TX)
    =
    \prod_{\delta_j : \textnormal{Chern root of }TX} \Gamma(1 + \delta_j) \in H^*(X;\mathbb{C}).
  \]
\end{theorem*}
Let us go back to the $K$-theoretical setting. The difficulty of the
problem depends on the Picard rank of $X$. The Picard rank 1 case does
not contain non-Fano manifolds. Therefore, since we would like to see
the role of the Fano condition, we will focus on the case when $X$ is a
Fano toric manifold of Picard rank 2. Let us point out that the Picard
rank 2 case is the 1st one to consider if one is interested in extending
Iritani's result to non-Fano manifolds. 
Using the $K$-theoretic mirror family of functions, we construct a
$q$-analogue of the oscillatory integral, by using the Jackson integral 
\[
  \left[
    \int_0^\infty
    \right]_q
    f(x) d_qx
    :=
    \sum_{d \in \mathbb{Z}}
    q^d
    f\left(q^d \right).
\] 
In Definition \ref{def:qosc_q_oscillatory_integral_real}, we introduce the $q$-oscillatory integral
\[
  \mathcal{I}^{K-\textnormal{th}}(q,Q_1,Q_2)
    :=
    \left[
    \int_{\Gamma }
    \right]_q
    \exp \left(
      W^{K-\textnormal{th}}_{|\pi^{-1}(Q)}(x_1,\dots,x_n)
    \right)
    \omega_{\pi^{-1}(Q),q}.
  \]
  \begin{remark}\label{rem:Iritani}
    The $q$-difference equations in K-theoretic GW theory were first
    studied by Iritani. Namely, using graph spaces, he was able to
    construct a system of $q$-difference equations, which can be
    solved in terms of the K-theoretic J-function. His argument is
    available in \cite{IMT}, Section 2.6. The existence of the
    q-difference equations was established by different methods also by Givental and
    Tonita in \cite{GT}. Furthermore, in the case of toric manifolds,
    Iritani proposed a $q$-oscillatory integral defined via the Jackson
    integral, which are very similar to the integrals that we would
    like to use, except that we work with $|q|>1$, while Iritani
    considered $|q|<1$ (see \cite{Iritani:qdifference_toric}, Section 5).  
  \end{remark}
In Proposition \ref{prop:qosc_mirrors_have_same_qde}, we show that the
functions $I_X^{K-\textnormal{th}}$ and
$\mathcal{I}^{K-\textnormal{th}}$ are solutions to the same set of
$q$-difference equations. 
In order to compare these two functions, we introduce a multiplicative
characteristic class, which can be viewed as a $q$-analogue of Iritani's gamma
class $\widehat{\Gamma}(TX)$. 
Denote by $(z;q)_\infty := \prod_{l \geq 0}(1-q^lz)$ the $q$-Pochhammer symbol (for $|q| < 1$), and let $\Gamma_q(x) := (1-q)^{1-x} \frac{(q;q)_\infty}{(q^x;q)_\infty}$ be the $q$-gamma function.
\begin{definition*}[Definition \ref{def:qosc_q_gamma_class}]
  The $q$-gamma class of a symplectic toric manifold is defined to be, for $|q| > 1$,
  \[
    \widehat{\gamma_q}(TX)
    :=
    \prod_{\delta_j : \textnormal{Chern root of }TX} 
\delta_j (1-q^{-1})^{\delta_j-1} \Gamma_{q^{-1}}(\delta_j) \in H^*(X;\mathbb{C}).
  \]
\end{definition*}
The 2nd main goal of this paper is to prove the following
$K$-theoretical analogue of Iritani's theorem: 
\begin{theorem*}[Theorem
  \ref{thm:gamma_comparison_theorem_in_quantum_k_theory}]
Suppose that $q>1$ is a real number and that the Novikov variables satisfy
$Q_1, Q_2 \in q^\mathbb{Z}$. Then the $K$-theoretic $I$-function
  $I_X^{K-\textnormal{th}}$ and the oscillatory integral
  $\mathcal{I}^{K-\textnormal{th}}$ are related by the following
  identity: 
  \[
    \mathcal{I}^{K-\textnormal{th}}(q,Q_1,Q_2)
    =
    \int_{
      \left[
        X
      \right]
    }
\widehat{\gamma}_q (TX) \cup
    \textnormal{ch}_q\left( I_{X}^{K-\textnormal{th}}(q,Q_1,Q_2)
\right),
  \]
where $\int_{[X]}$ denotes the cap product with the fundamental class
$[X] \in H_*(X_{\mu,K};\mathbb{C})$, $\widehat{\gamma}_q(TX)$
is the $q$-gamma class of Definition \ref{def:qosc_q_gamma_class} and 
  \[
    \textnormal{ch}_q(E)
    :=
    \sum_{\delta_j : \textnormal{Chern root of }E}
    q^{\delta_j}.
  \]
\end{theorem*}
Note that in quantum cohomology, the scope of \cite{Iritani:gamma_structure} goes much beyond the comparison theorem as we stated it here - in particular, the gamma class $\widehat{\Gamma}(TX)$ is used to define the A-side integral structure and the quantum cohomology central charge.
The search for $K$-theoretical analogues of these constructions should motivate further investigations related to the $q$-gamma class $\widehat{\gamma}_q (TX)$.

\subsection{Plan of the paper}

This paper is essentially structured in two independant parts.
The first part consists of Sections \ref{sec:ABC} and \ref{sec:confluence}, whose goal is to prove the confluence of the small $K$-theoretical $J$-function to the small cohomological $J$-function.
Section \ref{sec:ABC} deals with prerequisites to understand the reconstruction result of \cite{GT}, which we will explain in Subsections \ref{sec:DM-isom} and \ref{ssec:reconstruction_of_j_function}.
Finally, in Subsection \ref{ssec:confluence_of_j_function}, we state and prove the 1st main result of this paper - confluence of the small $K$-theoretic $J$-function (Theorem \ref{t1}).

The second part consists of Section \ref{sec:oscillatory_integral_and_gamma}, in which we study two mirrors to the $J$-function: the $I$-function and the ($q$-)oscillatory integral.
Subsection \ref{ssec:toric_manifolds} fixes the notations for symplectic toric manifolds.
Subsection \ref{ssec:oscillatory_in_qh} deals with the results already known in quantum cohomology. We will also give a new strategy to prove the comparison result between the small cohomological $I$-function and the oscillatory integral (see Theorem \ref{thm:gamma_cohomology_oscillatory_i_function}).
In subsection \ref{ssection:q_gamma_structure_in_qk}, motivated by
Givental's proposal for mirror symmetry for the small $K$-theoretical
$J$-function, we define a $q$-oscillatory integral (defined through the Jackson
integral, see Definition
\ref{def:qosc_q_oscillatory_integral_real}). There is a subtlety in
the case when the Novikov variables do not belong to the $q$-spiral
$q^\ZZ:=\{q^m\ |\ m\in \ZZ\}$. Following De Sole and Kac  
\cite{DeSole_Kac:integral_representation_of_qgamma} we were able to
find a natural construction of the $q$-oscillatory integral for
arbitrary values of the Novikov variables modulo a certain conjecture
about the regularity of the $K$-theoretic quantum $q$-difference
equations. Furthermore, in Section \ref{sec:comp_thm} we prove the 2nd
main result of this paper, that is, the comparison between the $K$-theoretical $I$-function and the $q$-oscillatory integral (Theorem 
\ref{thm:gamma_comparison_theorem_in_quantum_k_theory}).
Finally, in Section \ref{sec:confl_comp_thm} we give a new proof of Iritani's
theorem by taking the limit $q\to 1$ in Theorem 
\ref{thm:gamma_comparison_theorem_in_quantum_k_theory} and using the
confluence result in Theorem \ref{t1}. 

\begin{remark}
  When this paper was uploaded as a preprint, we have decided to add three appendices.
  The first appendix contains a sketch of the proof of Theorem \ref{thm:confluence_qHHR_recursion_formula}, following \cite{GT}, Propositions 3 and 4.
  In the second appendix, we show that our proof of Theorem \ref{thm:gamma_comparison_theorem_in_quantum_k_theory} can also be used in the setting of quantum cohomology to obtain a new proof of Theorem \ref{thm:gamma_cohomology_oscillatory_i_function}.
  The last appendix deals with a different definition of the $q$-oscillatory integral (see Definition \ref{def:qosc_q_oscillatory_integral_real}), for which we tried to do similar computations but could only manage partial results.
  Since these appendices only contain partial results, or no new result, we have decided to remove them from the present submission while keeping them available on the preprint version.
\end{remark}
\section{The ABC of twisted Gromov--Witten theory}
\label{sec:ABC}

Our goal here is to give the definitions of Gromov--Witten invariants required for the reconstruction theorem of quantum $K$-theory from quantum cohomology of \cite{GT}.
In Subsection \ref{sec:def} we define ABC-twisted (cohomological) GW invariants.
In Subsection \ref{sec:tw-rem} we explain how to reconstruct ABC-twisted GW invariants in terms of the usual cohomological GW invariants.
In the remaining subsections \ref{sec:fake_kgw} and \ref{sec:stem_inv}, we define the two particular kinds of ABC-twisted GW invariants that appear in the reconstruction theorem.

Suppose that $Y$ is an orbifold whose coarse moduli space $|Y|$ is a
projective variety. Let $IY=\bigsqcup_{v=1}^m Y_v$ and
$\overline{IY}=\bigsqcup_{v=1}^m \overline{Y}_v$ 
be respectively the {\em inertia} and the {\em rigidified
  inertia} orbifolds of $Y$, where the index $v$ ($1\leq v\leq m$) enumerates the connected
components of the coarse moduli space $|IY|=|\overline{IY}|$. Let us assume that
$Y_1=Y$ and so $Y_v$ (resp. $\overline{Y}_v$) with $v\neq 1$ are the so-called {\em twisted
} (resp. {\em rigidified twisted}) {\em sectors}. We would like to recall Givental--Tonita's twisted
orbifold Gromov--Witten theory of $Y$, which plays a key role in the
study of $K$-theoretic GW theory (see \cite{GT,T}).

In fact, the only orbifolds that we will be interested in will be the global
quotients $Y=[X/{\mathbf{\mu}}_m]$, where ${\mathbf{\mu}}_m$ is the
multiplicative group of order $m$, $X$ is a smooth projective variety,
and the action of ${\mathbf{\mu}}_m $ is the trivial one. In this
case the twisted sectors are parametrized by the elements $g$ of the
group ${\mathbf{\mu}}_m$ and we have
\ben
Y_g = [X/{\mathbf{\mu}}_m],\quad
\overline{Y}_g=[X/({\mathbf{\mu}}_m/\langle g\rangle )],
\een
where $\langle g\rangle$ denotes the cyclic subgroup generated by
$g$. 

\subsection{Definition of ABC twists}\label{sec:def}
Some standard references for orbifold GW theory is \cite{AGV,CR} (see
also \cite{T} for an overview). Let $Y_{g,n,d}$ be the moduli space of
orbifold stable maps $f:(C,s_1,\dots,s_n)\to Y$. The moduli space has
connected components $Y_{g,n,d}^{v_1,\dots,v_n}$ parametrized by $n$-tuples $(v_1,\dots,v_n)$,
$v_i\in [1,m]:=\{1,2,\dots,m\}$ and consisting of stable maps, such
that, $f|_{s_i}\in \overline{Y}_{v_i}$. Let $\C_{g,n,d}^{(v_1,\dots,v_n)}:=Y_{g,n+1,d}^{(v_1,\dots,v_n,1)}$ be the
connected components of the universal curve $\C_{g,n,d}$ and
$\pi: \C_{g,n,d}\to Y_{g,n,d}$ and $\operatorname{ev}:\C_{g,n,d}\to Y$
be respectively the map forgetting the last marked point and the
evaluation map at the last marked point. Finally, let $Z_v$ ($v\in [1,m]$) be
the closed substack of $\C_{g,n,d}$ parametrizing stable maps
$(C,s_1,\dots,s_{n+1},f)$, such that, if $C'$ denotes the
irreducible component of $C$ that carries the $(n+1)$-st marked point
$s_{n+1}$, then
\begin{itemize}
\item[(i)]
  $C'$ carries exactly two nodes of $C$, say $z_+$ and $z_-$, and no
  other marked points.
\item[(ii)]
 The map $f$ maps $C'$ to $Y$ with degree 0, that is, $C'$ is
 contracted to a point.
\item[(iii)]
  The evaluation map at $z_+$ or $z_-$ lands in $\overline{Y}_v.$
\end{itemize}
Let us recall that the forgetful map $\pi$ is characterized by the
property that it forgets the $(n+1)$st marked point and it contracts
the resulting unstable components. Therefore, the fibers of
$\pi|_Z:Z:=\bigsqcup_v Z_v\to Y_{g,n,d}$ are non-empty only if the domain
curve $C$ of the corresponding stable map in $Y_{g,n,d}$ is singular.
In this case, the points in the fiber of $\pi|_Z$ correspond to the
singular points of $C$. The above moduli spaces and maps between them form
the following diagram:
\ben
\xymatrix{Z_v\ar[r]^-{\iota_v} &
  C_{g,n,d}^{(v_1,\dots,v_n)}=Y_{g,n+1,d}^{(v_1,\dots,v_n,1)}
  \ar[d]_\pi
  \ar[r]^-{\operatorname{ev}_{n+1}} &
  Y\\
  & Y_{g,n,d}^{(v_1,\dots,v_n)}
  \ar[r]^-{(\operatorname{ev}_{1},\dots, \operatorname{ev}_{n})} &
  \overline{Y}_{v_1}\times \cdots \times \overline{Y}_{v_n},}
\een
where $\operatorname{ev}_i$ is the evaluation map at the $i$-th marked
point and $\iota_v$ is the natural inclusion map. 

The twisted GW invariants depend on the choice of the following 3
types of data:
\begin{enumerate}
\item[(A)]
  A finite number of orbifold vector bundles
  $E_\alpha\to Y$ (where $1\leq \alpha\leq k_A$) and identically indexed
  multiplicative characteristic classes
  \ben
  A_\alpha (\ell):= \exp\left(\sum_{i=0}^\infty
    s^A_{\alpha,i} \operatorname{ch}_i(\ell)
  \right),\quad s^A_{\alpha,i}\in \CC.
  \een
\item[(B)]
  A finite number of polynomials $f_\beta\in K^0(X)[\ell]$ where ($1\leq
  \beta\leq k_B$) and identically indexed multiplicative characteristic
  classes 
  \ben
  B_\beta (\ell):= \exp\left(\sum_{i=0}^\infty
    s^B_{\beta,i} \operatorname{ch}_i(\ell)
  \right),\quad s^B_{\beta,i}\in \CC.
  \een
\item[(C)]
  A finite number of orbifold vector bundles $E_{v,\gamma}\to Y$
  (where $v\in [1,m], 1\leq \gamma\leq k_v$) and identically indexed
  multiplicative characteristic classes
  \ben
  C_{v,\gamma} (\ell):= \exp\left(\sum_{i=0}^\infty
    s^C_{v,\gamma,i} \operatorname{ch}_i(\ell)
  \right),\quad s^C_{v,\gamma,i}\in \CC.
  \een
\end{enumerate}
Using the above data we define the following three types of cohomology
classes:
\begin{definition}
The $ABC$-twists are the three cohomological classes on $Y_{g,n,d}$ defined by:
\ben
\Theta^A_{g,n,d}:=
\prod_{\alpha=1}^{k_A}
A_\alpha(\pi_* \operatorname{ev}^*_{n+1}(E_\alpha)),
\een
\ben
\Theta^B_{g,n,d}:=
\prod_{\beta=1}^{k_B}
B_\beta(\pi_* (
\operatorname{ev}^*_{n+1}f_\beta(L^{-1}_{n+1})-
\operatorname{ev}^*_{n+1}f_\beta(1) )),
\een
where $\operatorname{ev}^*_{n+1}f_\beta(\ell)\in
\operatorname{ev}^*_{n+1}K^0(X)[\ell]$ and
\ben
\Theta^C_{g,n,d}:=
\prod_{v=1}^m
\prod_{\gamma=1}^{k_v}
C_{v,\gamma}(\pi_*(
\operatorname{ev}^*_{n+1}(E_{v,\gamma})\otimes
\iota_* \O_{Z_v}),
\een
where $\pi_*=R^0\pi_*-R^1\pi_*$ is the $K$-theoretic pushforward and
$L_{n+1}\to C_{g,n,d}$ is the orbifold line bundle corresponding to
the cotangent line at the $(n+1)$st marked point.
\end{definition}
Put
$\Theta^{ABC}_{g,n,d} =\Theta^{A}_{g,n,d} \Theta^{B}_{g,n,d}
\Theta^{C}_{g,n,d} $. Then the ABC-twisted GW-invariants are
defined by the following formula:
\begin{definition}\label{ABC-inv}
Let $k_1, \dots, k_n \in \mathbb{Z}_{\geq 0}$. The ABC-twisted Gromov--Witten invariants of $Y$ are defined by
\ben
\langle
\phi_{i_1}\psi^{k_1},\dots,\phi_{i_n}\psi^{k_n}
\rangle^{ABC}_{g,n,d} =
\int_{[Y_{g,n,d}]^{\rm virt}} \Theta^{ABC}_{g,n,d} \prod_{s=1}^n
\operatorname{ev}_s^*(\phi_{i_s}) \psi_s^{k_s},
\een
where $\{\phi_i\}_{i=1}^N$ is a basis of $H^*(\overline{IY};\CC)=H^*(IY;\CC)$,
$\psi_s$ is the 1st Chern class of the orbifold line bundle $L_s$
corresponding to the cotangent line at the $s$-th marked point, and 
$[Y_{g,n,d}]^{\rm virt}$ is the virtual fundamental cycle of \cite{BF_intnormalcone}.
\end{definition}

Let us also recall some properties of the grading in orbifold GW
theory, which we will need later on to prove confluence of the $K$-theoretical $J$-function. Let
$\mathcal{Y}=(\mathcal{Y}_1\rightrightarrows \mathcal{Y}_0)$ be an
orbifold groupoid representing the Morita equivalence class of $Y$.
The inertia orbifold $IY$ is represented by the orbifold groupoid
whose objects are pairs $(y,g)$, $y\in \mathcal{Y}_0$ and $g\in
\operatorname{Aut}(y):=\{\xi\in \mathcal{Y}_1| s(\xi)=t(\xi)=y\}$ and
whose morphisms $\operatorname{Mor}_{I\mathcal{Y}}((y,g),(y',g'))$ are morphisms
$f\in \operatorname{Mor}_{\mathcal{Y}}(y,y')$ such that
$g'\circ f=f\circ g$.  Furthermore, the orbifold tangent bundle $TY$
is represented by the orbifold groupoid $T\mathcal{Y}$, whose objects
are pairs $(y,\xi)$, $y\in \mathcal{Y}_0$ and $\xi\in
T_y\mathcal{Y}_0$ and whose morphisms
$\operatorname{Mor}_{T\mathcal{Y}}((y,\xi),(y',\xi'))$ consists of
morphisms $f\in \operatorname{Mor}_{\mathcal{Y}}(y,y')$, such that,
$df(\xi)=\xi'$. If $(y,g)\in I\mathcal{Y}$, then $g$ acts naturally on
$T_y\mathcal{Y}_0$ and since $g$ has finite order, it is
diagonalizable. Let $\lambda_i(y,g)=e^{2\pi\mathbf{i} \alpha_i(y,g)}$,
$0\leq \alpha_i(y,g)<1$, $1\leq i\leq \operatorname{dim}(Y)$ be the
eigenvalues of $g$. Then $\iota(y,g):=\sum_i \alpha_i(y,g)$ is a
rational number depending only on the connected component $Y_v$ to
which the point $(y,g)$ belongs to, so we put
$\iota(v):=\iota(y,g)$. If $\phi\in H^{2k}(Y_v,\CC)$, then the
Chen--Ruan degree of $\phi$ is defined to  be
$\operatorname{deg}_{CR}(\phi):=k+\iota(v)$.

\begin{proposition}
The complex dimension of the virtual fundamental cycle of
$Y_{g,n,d}^{(v_1,\dots,v_n)}$ is
\ben
\textnormal{dim} \left[ Y_{g,n,d}^{(v_1,\dots,v_n)} \right]^{\textnormal{vir}} =
3g-3+n + D(1-g) + \int_d c_1(TY) -\sum_{i=1}^n \iota(v_i),
\een
where $D$ is the complex dimension of $Y$.
\end{proposition}

\begin{corollary}
If the twisted GW invariant of Definition \eqref{ABC-inv} is non-zero,
then the following inequality holds:
\ben
\sum_{s=1}^n \operatorname{deg}_{CR}(\phi_{i_s}) + k_{s} \leq
3g-3+n + D(1-g) + \int_d c_1(TY).
\een
\end{corollary}
\subsection{Removing the twists}\label{sec:tw-rem}
The results of this section are not really needed in this paper. We
include them just for completeness of the reconstruction theorem. To begin with let us organize the
twisted GW invariants into a generating function. Let
$\mathbf{t}=(t_{k,v,a})$ be a sequence of formal variables where
$k\geq 0$, $v\in [1,m]$, and $1\leq a\leq
N_v:=\operatorname{dim}H(Y_v;\CC)$. Let $\{\phi_{v,a}\}$ ($v\in [1,m],
1\leq a\leq N_v$) be a basis of $H(Y_v;\CC)$. Put
\ben
\mathbf{t}(z):=\sum_{k=0}^\infty \sum_{v=1}^m \sum_{a=1}^{N_v}
t_{k,v,a} \phi_{v,a} z^k. 
\een
The total descendent potential of the ABC-twisted GW invariants is
defined by
\ben
\D^{\rm ABC}(\hbar,\mathbf{t}) :=\exp\Big(
\sum_{n,g,d} \frac{\hbar^{g-1} Q^d}{n!} \langle
\mathbf{t}(\psi_1),\dots,\mathbf{t}(\psi_n)\rangle^{\rm ABC}_{g,n,d}
\Big).
\een
Let us define also AB-twisted, A-twisted, and non-twisted (i.e. cohomological) GW
invariants by formula \eqref{ABC-inv} except that we replace
$\Theta^{ABC}$ with respectively $\Theta^{AB}:=\Theta^A\Theta^B$,
$\Theta^A$, and $1$. 
The corresponding total descendent potentials $\D^{\rm AB}$, $\D^{\rm A}$, and $\D$
are defined in the same fashion.
Our goal is to express $\D^{\rm ABC}$ in terms of $\D$.

Let us begin with the C-twist. Let $I:IY\to IY$ be the involution
induced by the map $(y,g)\mapsto (y,g^{-1})$. The involution maps a
connected component $Y_v$ isomorphically to a connected component which
we denote by $Y_{I(v)}$. The C-twist is removed by the following
formula:
\begin{proposition}[\cite{T}]
The C-twist is removed from the total descendant potential $\D^{\rm ABC}$ according to the following identity:
\ben
\D^{\rm ABC}(\hbar,\mathbf{t})=\exp\Big(\frac{\hbar}{2} 
\sum_{v=1}^m \sum_{k,l=0}^\infty \sum_{a,b=1}^{N_v}
A^v_{ka,lb} \frac{\partial^2}{\partial t_{v,k,a} \partial
  t_{I(v),l,b}}
\Big)\, \D^{\rm AB}(\hbar,\mathbf{t}),
\een
where the coefficients $A^v_{ka,lb}\in \CC$ are defined in terms of
the twisting data of type $C$.
\end{proposition}
Let us recall the definition of the coefficients $A^v_{ka,lb}$. First, we need to define a
map
\ben
\Delta_v: 
H(Y_v;\CC)[z]\to H(Y_v;\CC)\otimes H(Y_{I(v)};\CC)[z_1,z_2],
\een
which is a $H(Y_v;\CC)[z]$-modules morphism, where the ring structure
on $H(Y_v;\CC)[z]$ is the obvious one induced by the topological cup product, $z$ acts on the
tensor product via multiplication by $z_1 + z_2$ and
$\phi \in H(Y_v;\CC)$ acts by cup-product multiplication on the first
tensor factor, i.e., via the operator $\phi\cup\otimes 1\cup$. In order to define
$\Delta_v$ then we need only to specify the image of $1$: we put
\ben
\Delta_v(1) := \sum_{a,b=1}^{N_v} \eta^{ab} \phi_{v,a}\otimes
\phi_{I(v),b}\quad , 
\een
where
\ben
\eta_{ab}:= \frac{1}{r(v)} \int_{[Y_v]} \phi_{v,a} \cup I^*(\phi_{I(v),b}) 
\een
are the entries of the matrix of the orbifold Poincare pairing and
$\eta^{ab}$ are the entries of the corresponding inverse matrix.
Here $r(v)$ denotes the order of the local automorphism $g$ from a
point $(y,g)\in Y_v$ and $[Y_v]$ is the fundamental cycle of the
coarse moduli space $|Y_v|$.  Let us identify
$H(Y_v;\CC)[z]=H(Y_v\times \CC P^\infty;\CC)$, so that $z$ is the 1st
Chern class of the universal line bundle $\ell:=\O(1)\to \CC
P^\infty$.  Then the coefficients of the above differential operator
are defined by
\ben
\sum_{k,l=0}^\infty \sum_{a,b=1}^{N_v}
A^v_{ka,lb} \phi_{v,a} \otimes \phi_{I(v),b} z_1^k z_2^l :=
\frac{
  \Delta_v\Big(\prod_{\gamma=1}^{k_v} C_{v,\gamma}\Big(
  q^*(E_{v,\gamma})|_{Y_v}\otimes (1-\ell) \Big) -1 \Big)}{
  -z_1-z_2},
\een
where $q:IY\to Y$ is the forgetful map $(y,g)\mapsto y$. 

Let us continue with the B-twist. The relation in this case is easier.
\begin{proposition}[\cite{T}]
The $B$-twist is removed from the total descendant potential $\D^{\rm AB}$ according to the following identity:
\beq\label{rem:B}
\D^{\rm AB}(\hbar c_B^2,\mathbf{t} c_B) = 
\operatorname{const}_B\, \D^{\rm A}(\hbar,\mathbf{t}(z)+ z -\delta_B(z)z),
\eeq
where the constants $c_B,\operatorname{const}_B\in \CC$ and the vector
$\delta_B(z)\in H(IY;\CC)[z]$ depend only on the twisting data of type $B.$
\end{proposition}
Let us recall the definition of $c_B$ and $\delta_B(z)\in
H(IY;\CC)[z]$. Let us identify
again $H(IY;\CC)[z]=H(IY\times\CC P^\infty;\CC)$ so that
$z=c_1(\ell)$. Then
\ben
\prod_{\beta=1}^{k_B} 
B_\beta\Big(
-\frac{f_\beta(\ell^{-1})-f_\beta(1) }{\ell-1}
\Big)=: c_B \delta_B(z),
\een
where $c_B\in \CC$ and $\delta_B(z)=1+\cdots$, where the dots stand
for cohomology classes in $H(IY;\CC)[z]$ of degree $>0$. 
\begin{remark}
Formula \eqref{rem:B} is slightly different from \cite{T}, Theorem
1.2. Tonita did not say explicitly, but from the proof it becomes
clear that he works with multiplicative characteristic classes
$B_\beta$ for which $s^B_{\beta,0}=0$. If $s^B_{\beta,0}=0$, then
$c_B=1$ and formula \eqref{rem:B} coincides with Tonita's. 
\end{remark}

Finally, let us recall how to remove the A-twist. Let
$E_{\alpha,v}:=q^*E_\alpha|_{Y_v}$ and 
$E_{\alpha,v,f}$ $0\leq f<1$ be the orbibundle on $Y_v$ whose fiber at a point $(y,g)\in
Y_v$ consists of all $\xi\in (E_{\alpha,v})_{(y,g)}$, such that, $g
(\xi)=e^{2\pi\ii f} \xi$. We have $E_{\alpha,v}=\bigoplus_{f}
E_{\alpha,v,f}$. Let us recall Tseng's operator: 
\begin{definition}[\cite{Ts}]\label{def:ABC_tseng_operator}
Tseng's operator $\Delta_A(z)$ is the cohomological operator defined by
\ben
\Delta_A(z):= 
\prod_{\alpha=1}^{k_A} \sqrt{A_\alpha(E^{\rm inv}_{\alpha})}
\exp\Big(
\sum_{i=0}^\infty \sum_{j=-1}^\infty
\sum_{v,f} 
s^A_{\alpha,i+j}\, \operatorname{ch}_i (E_{\alpha,v,f})\,
\frac{B_{j+1}(f) z^j }{(j+1)!}
\Big),
\een
where $E_\alpha^{\rm inv}=\sum_{v=1}^m E_{\alpha,v,0}\in K^0(IY)$ and
$B_j(t)$ are the Bernouli polynomials defined by the following identify: 
\ben
\frac{x e^{tx}}{e^x-1} = \sum_{j=0}^\infty B_j(t) \frac{x^j}{j!}. 
\een
\end{definition}
It turns out that $\Delta_A(z)$ is a symplectic transformation with
respect to Givental's symplectic loop space formalism (more precisely
its orbifold version). The quantization $\widehat{\Delta}_A$ yields a
differential operator acting on the Fock space 
\ben
\CC(\!(\hbar)\!)[\![s^A, q_0, q_1+\mathbf{1}, q_2,\dots]\!],
\een
where $s^A=(s^A_{\alpha,i})$ and $q_k= (q_{k,v,a})$ ($v\in [1,m],
1\leq a\leq N_v$) are formal 
vector variable. The components of $q_k$ should be thought off secretly as
linear coordinates on $H(IY;\CC)$ with respect to the basis
$\{\phi_{v,a}\}$. The shift in $q_1+\mathbf{1}$, also known as the
{\em dilaton shift}, is only in the variable $q_{1,1,1}$ corresponding
to the unit $\phi_{1,1}=1\in H(|Y|;\CC)$.   Both
$\D^A(\hbar,\mathbf{t})$ and $\D(\hbar,\mathbf{t})$ are identified
with elements in the Fock space via respectively the substitutions
\ben
\mathbf{q}(z) =(\mathbf{t}(z)-z) 
\prod_{\alpha=1}^{k_A} 
\sqrt{A_\alpha(E^{\rm inv}_{\alpha})}
\een
and $\mathbf{q}(z)=\mathbf{t}(z)-z$. Then,
\begin{proposition}[\cite{T}]
The $A$-twist is removed from the total descendant potential $\D^{\rm A}$ according to the following identity:
\ben
\D^A(\hbar,\mathbf{q}) = \operatorname{const}_A\,
\widehat{\Delta}_A\, 
\D(\hbar,\mathbf{q}),
\een
where $\operatorname{const}_A$ is a constant depending only on the
twisting data of type A.
\end{proposition} 

\subsection{Fake $K$-theoretic GW invariants}\label{sec:fake_kgw}
Suppose now that $X$ is a smooth projective variety. By definition,
the tangent bundle $\T_{g,n,d}$ of the moduli space $X_{g,n,d}$ is a $K$-theoretic
vector bundle defined by 
\ben
\T_{g,n,d}:= \pi_* \operatorname{ev}^*_{n+1}(T_X-1) -
\pi_*(L_{n+1}^{-1}-1) -
(\pi_*\iota_* \O_Z)^\vee,
\een
where the notation is the same as in Section \ref{sec:def} for the
case $Y=X$. The fake $K$-theoretic GW invariants are defined as if applying Hirzebruch--Riemann--Roch (HRR) formula to $X_{g,n,d}$ (being a stack, the usual HRR formula does not hold).
\begin{definition}
The fake Gromov--Witten invariants of $X$ are defined by the formula
\ben
\langle \Phi_{i_1} L^{k_1},\dots, \Phi_{i_n} L^{k_n}
\rangle^{\rm fake}_{g,n,d}:=
\int_{[X_{g,n,d}]^{\rm virt}} 
\operatorname{td}(\T_{g,n,d}) 
\prod_{s=1}^n
\operatorname{ch}(\operatorname{ev}_s^*(\Phi_{i_s})\otimes L_s^{k_s})
\in \mathbb{Q},
\een
where the Todd class of a vector bundle is defined by 
\ben
\operatorname{td}(E) = \prod_{x:{\rm Chern \, root\,  of } E} \frac{x}{1-e^{-x}}.
\een
\end{definition}
Note that the fake $K$-theoretic invariants are ABC-twisted invariants
with the following twisting data:
\begin{enumerate}
\item[(A)] Vector bundle $E:=T_X-1$ and corresponding multiplicative
  characteristic class
\ben
A(\ell):=\operatorname{td}(\ell) = \frac{z}{1-e^{-z}}.
\een
\item[(B)]
Polynomial $f(\ell)=\ell$ and corresponding multiplicative
  characteristic class
\ben
B(\ell) := \operatorname{td}(-\ell) = 
\frac{1}{\operatorname{td}(\ell)} = \frac{1-e^{-z}}{z}. 
\een
\item[(C)]
Trivial vector bundle and corresponding characteristic class
\ben
C(\ell):= 
\operatorname{td}(-\ell)^\vee= 
\frac{1}{\operatorname{td}(\ell^\vee)}= \frac{e^z-1}{z}.
\een
\end{enumerate}
Note that in $H^*(X_{g,n,d})$, we can write twisting class as $\Theta^{ABC}_{g,n,d}=1+\cdots$, where
the dots contain only cohomology classes on $X_{g,n,d}$ of degree
$>0$. In particular, if we have an $ABC$-twisted correlator as in Definition
\ref{ABC-inv}, such that, the cohomological degree of the correlator insertions
add up to the dimension of the virtual fundamental cycle
$[X_{g,n,d}]^{\rm virt}$, then the higher degree terms of
$\Theta^{ABC}_{g,n,d}$ do not contribute and hence the ABC-twisted
invariant coincides with the usual GW invariant. 

\begin{definition}
The fake $K$-theoretic $J$-function is by definition 
\ben
J^{\rm fake}(q,Q,\tau)=1-q + \tau + \sum_{d,a,l} \frac{1}{l!} 
\left\langle \frac{\Phi_a}{1-qL}, \tau,\dots,\tau
\right\rangle^{\rm fake}_{0,1+l,d} Q^d \Phi^a
\in K(X)(q)[\![Q]\!],
\een
where $\tau\in K^0(X)$ and the first insertion in the correlator
should be expanded using $\frac{1}{1-qL}=\sum_{l \geq 0} q^k L^k$.
\end{definition}

\subsection{Stem invariants}\label{sec:stem_inv}
Let $X$ be a smooth projective variety and $m>1$ an integer. Let us
fix a primitive $m$-th root of unity $\zeta$ and denote by
$\eta=e^{2\pi\ii/m}$ the standard generator of the multiplicative
group $\mu_m$. Following Givental and
Tonita (see \cite{GT}) we will refer to the moduli space 
$[X/\mu_m]_{0,n+2,d}^{(\eta,1,\dots,1,\eta^{-1})}$ as the moduli
space of {\em stems}. The inertia orbifold of $Y:=[X/\mu_m] $ consists of
$m$ copies of $Y$, which correspond naturally to the elements $g$ of
the group $\mu_m$. Let $\one_g$ be the unit in $H(|Y_g|;\CC)$. Let us
fix a basis $\{\phi_i\}$ ($1\leq i\leq N$) of $H(X;\CC)$ and denote by
$\phi_i \, \one_g$ the cohomology class $\phi_i$ but viewed as an
element in $H(|Y_v|;\CC)=H(X;\CC)$. Clearly,
$\{\phi_i\, \one_g\}$ ($1\leq i\leq N$, $g\in \mu_m$) is a basis of
the orbifold cohomology $H(|IY|;\CC)$. Note that if $\phi_i\in H^{2d_i}(X;\CC)$
and $g=e^{2\pi \ii k/m}$, then the Chen--Ruan degree of $\phi_i\,
\one_g$ is $d_i$, because the action of $g$ on the tangent bundle is
trivial. We will be interested in the ABC-twisted 
GW invariants of $[X/\mu_m]$ with the following twisting data:

{\bf Type A:}
Vector bundles $E_\alpha=T_X \otimes \CC_{\zeta^{\alpha-1}}$ $(1\leq
\alpha\leq m$),  where $\CC_{\zeta^k}$ is the orbibundle $[(X\times
\CC)/\mu_m]$, where the action of $\mu_m$ on $\CC$ is defined by
requiring that the standard generator $\eta=e^{2\pi \ii/m}$ of $\mu_m$ acts
by $\zeta^k$.    The corresponding multiplicative classes are 
\begin{align}
\nonumber
A_1(\ell) & :=  \operatorname{td}(\ell)=\frac{z}{1-e^{-z}},\\
\nonumber
A_{k+1}(\ell) & := \operatorname{td}_{\zeta^k}(\ell) =
\frac{1}{1-\zeta^k e^{-z}}\quad (1\leq k\leq m-1).
\end{align}
 
{\bf Type B:}
Polynomials $f_\beta(\ell)=\CC_{\zeta^{\beta-1}}\, \ell \in
K^0(Y)[\ell]$ ($1\leq \beta\leq m$), where $\CC_{\zeta^k}$ are the
same as in the type A data above. The corresponding characteristic
classes are 
\begin{align}
\nonumber
B_1(\ell) & :=  \operatorname{td}(-\ell)=\frac{1-e^{-z}}{z},\\
\nonumber
B_{k+1}(\ell) & := \operatorname{td}_{\zeta^k}(-\ell) =
{1-\zeta^k e^{-z}}\quad (1\leq k\leq m-1).
\end{align}

{\bf Type C:} If $g\neq 1$, then we have only one orbibundle, that is,
$k_g:=1$ and $E_{g,1}$ is the trivial line bundle on $Y$. If $g=1$,
then we have $m$ orbifold line bundles, that is, $k_1=m$ and 
$E_{1,\gamma}:= \CC_{\zeta^{\gamma-1}}$ ($1\leq \gamma\leq m$). The
corresponding characteristic classes are
\begin{align}
\nonumber
C_{g,1}(\ell) & :=
                               \operatorname{td}(-\ell)^\vee=\frac{e^{z}-1}{z},
\quad g\neq 1\\
\nonumber
C_{1,1}(\ell) & :=
                               \operatorname{td}(-\ell)^\vee=\frac{e^{z}-1}{z},\\
\nonumber
C_{1,k+1}(\ell) & := \operatorname{td}_{\zeta^k}(-\ell)^\vee =
{1-\zeta^k e^{z}}\quad (1\leq k\leq m-1).
\end{align}
In other words,
\begin{align}
\label{thetaA}
\Theta^A & = \operatorname{td}(
\pi_*\operatorname{ev}^* (T_X-1)
)
\prod_{k=1}^{m-1}
\operatorname{td}_{\zeta^k}(
\pi_*\operatorname{ev}^* (
T_X\otimes \CC_{\zeta^k}
)
) ,\\
\label{thetaB}
\Theta^B & = \operatorname{td}(-\pi_*(L^{-1}-1))
\prod_{k=1}^{m-1} \operatorname{td}_{\zeta^k}(
-\pi_*((L^{-1}-1)\otimes 
\operatorname{ev}^*\CC_{\zeta^k}
) 
) ,\\
\label{thetaC}
\Theta^C & = 
\operatorname{td}(-\pi_*\iota_* \O_{Z_g})^\vee 
\operatorname{td}(-\pi_*\iota_* \O_{Z_1})^\vee 
\prod_{k=1}^{m-1}
\operatorname{td}_{\zeta^k}(
-\pi_*(
\iota_* \O_{Z_1}\otimes 
\operatorname{ev}^*\CC_{\zeta^k}
)
)^\vee .
\end{align}
Following Givental--Tonita we define the stem invariants of $X$ as an ABC-twisted GW invariant of $[X/\mu_m]$:
\begin{definition}[\cite{GT}]\label{def:stem_stem_invariants}
The stems invariants of $X$ are defined by the following formula:
\ben
\left[
\phi_{i_1} \, \psi_1^{k_1},\dots ,\phi_{i_{n+2}}\, \psi_{n+2}^{k_{n+2}}
\right]_{0,n+2,d}
:=
\left\langle
\phi_{i_1} \one_\eta \, \psi_1^{k_1}, 
\phi_{i_2} \, \psi_2^{k_2},\dots ,\phi_{i_{n+1}}\, \psi_{n+1}^{k_{n+1}},
\phi_{i_{n+2}}\one_{\eta^{-1}} \, \psi_{n+2}^{k_{n+2}}
\right\rangle_{0,n+2,d}^{ABC}.
\een
\end{definition}
\section{Confluence of the small $K$-theoretic $J$-function}\label{sec:confluence}
The main goal in this section is to prove confluence of the $K$-theoretic $J$-function to its cohomological analogue (Theorem \ref{t1}). As we
already explained in the introduction, we need to recall the
reconstruction of Givental--Tonita outlined in \cite{GT}, Proposition
4.
\subsection{The Kawasaki--Riemann--Roch formula}
To begin with, let us recall Kawasaki's formula generalising the
Hirzebruch--Riemann--Roch formula to orbifolds (see \cite{Ka}). Let
$Y$ be as in Section \ref{sec:ABC},
$\mathcal{Y}=(\mathcal{Y}_1\rightrightarrows \mathcal{Y}_0)$ be a
corresponding orbifold groupoid, and $q:IY\to Y$ be the forgetful
map $(y,g)\mapsto y$. If $E\to Y$ is an orbifold complex vector bundle, we use
 the decompositions $q^*E|_{Y_v}=: \bigoplus_{0\leq f<1}
E_{v,f}$ -- same as in the definition of Tseng's operator, c.f. Definition \ref{def:ABC_tseng_operator}.

\begin{definition}
  \begin{itemize}
    \item[(i)]
    Let $E \to Y$ be an orbifold complex vector bundle.
    The \textit{trace} of $E$ is the K-theoretic orbifold vector
    bundle on $IY$ defined by
    \ben
    \operatorname{Tr}(E):=\sum_{v=1}^m\sum_{0\leq f<1} e^{2\pi\ii f} E_{v,f}
    \een

    \item[(ii)]
    The {\em inertia tangent} and {\em
    inertia normal} vector bundles are orbifold $K$-theoretic vector
    bundles on $IY$ defined by respectively
    \ben
    T_{IY}:=\sum_{v=1}^m TY_{v,0}.
    \quad \mbox{ and }\quad
    N_{IY}:=\sum_{v=1}^m \sum_{0<f<1} TY_{v,f}.
    \een

    \item[(iii)] The {\em holomorphic Euler characteristic} of the bundle $E$ is given by
    \ben
    \chi(Y,E):=\sum_{i=0}^ D (-1)^i \operatorname{dim} H^i(Y,E)
    \een
  \end{itemize}
\end{definition}
Kawasaki's formula can be stated as follows:
\begin{theorem}\cite{Ka}\label{thm:KRR_Kawasaki_Riemann_Roch}
  Let $E$ be a holomorphic orbifold vector bundle on $Y$. Its
  holomorphic Euler characteristic can be computed by the following formula:
  \beq\label{KHRR}
  \chi(Y,E)=\int_{[IY]} \operatorname{td}(T_{IY}) \
    \frac{
      \operatorname{ch}\circ \operatorname{Tr}(E)}{
      \operatorname{ch}\circ \operatorname{Tr}(
      \wedge^\bullet(N_{IY}^\vee)
      )
    }\ ,
  \eeq
  where the denominator on the RHS is by definition
  \ben
  \operatorname{ch}\circ\operatorname{Tr}(\wedge^\bullet(N_{IY}^\vee)) =
  \sum_{v=1}^m\prod_{0<f<1}
  \prod_{\mbox{ Chern roots $x$ of } TY_{v,f}  }
  \Big(1-e^{-2\pi\ii f} e^{-x} \Big). 
  \een
\end{theorem}
The moduli spaces $X_{0,n,d}$ are usually not orbifolds. Nevertheless, 
it is known that Kawasaki's formula can be applied to $Y=X_{0,n,d}$
too -- see \cite{To2}. 
\subsection{Stems as Kawasaki strata}\label{sec:DM-isom}
Recall that in Definition \ref{def:intro_k_theoretical_j_function} of the $K$-theoretic $J$-function, we used $K$-theoretic GW invariants of the form
$\Big\langle \frac{\Phi_i}{1-qL_1} \Big\rangle_{0,1,d}$ which are understood as holomorphic Euler characteristic of bundles on the moduli space $X_{0,1,d}$.
Let us apply Kawasaki's formula \eqref{KHRR} to this moduli space. It is
convenient to think of the inertia orbifold $IX_{0,n,d}$ as the moduli
space of $(C,s_1,\dots,s_n,f,g)$, where $(C,s_1,\dots,s_n,f)$ is a
stable map in $X_{0,n,d}$ and
$g\in \operatorname{Aut}(C,s_1,\dots,s_n,f)$. Suppose that $\zeta$ is
a primitive $m$-th root of 1. Let us define
\ben
I_\zeta X_{0,n,d}:=\{ (C,s_1,\dots,s_n,f,g)\ |\ dg
\mbox{ acts on $T^*_{s_1}C$ via multipl. by } \zeta \}.
\een
We are going to construct an explicit isomorphism $\varphi$ of Deligne--Mumford stacks
\beq\label{stem_iso}
\xymatrix{
  \coprod_{\substack{d_0,\dots,d_{k+1}\\
  \eta_1,\dots,\eta_{k+1}}}
  \left(
    [X/\mu_m]^{\eta,1,\dots,1,\eta^{-1}}_{0,k+2,d_0}\times_{X^{k+1}}\Big(
    I_{\eta_1}X_{0,1,d_1}\times
    \cdots \times
    I_{\eta_{k+1}}X_{0,1,d_{k+1}}\Big)
    \right) \ar[r]^-{\cong} &
    I_\zeta X_{0,1,d} },
\eeq
where the disjoint union is over all sequences $d_0,\dots, d_{k+1} \in
\operatorname{Eff}(X)$ of effective curve classes and all sequences
$\eta_1,\dots,\eta_{k+1}$ of primitive roots of 1, such that,
$m(d_0+\cdots+d_k)+d_{k+1}=d$ and $\eta_i\neq 1$ for all $1\leq i\leq
k$ and $\eta_{k+1}\neq \zeta$. The fiber
product is defined in terms of the evaluation maps in such a way that
the $(i+1)$-st marked point in
$[X/\mu_m]^{\eta,1,\dots,1,\eta^{-1}}_{0,k+2,d_0}$ is identified with
the marked point of $I_{\eta_i}X_{0,1,d_i}$ for all $1\leq i\leq
k+1$.

The isomorphism $\varphi$ is defined as follows (see Figure \ref{fig:stem}). 
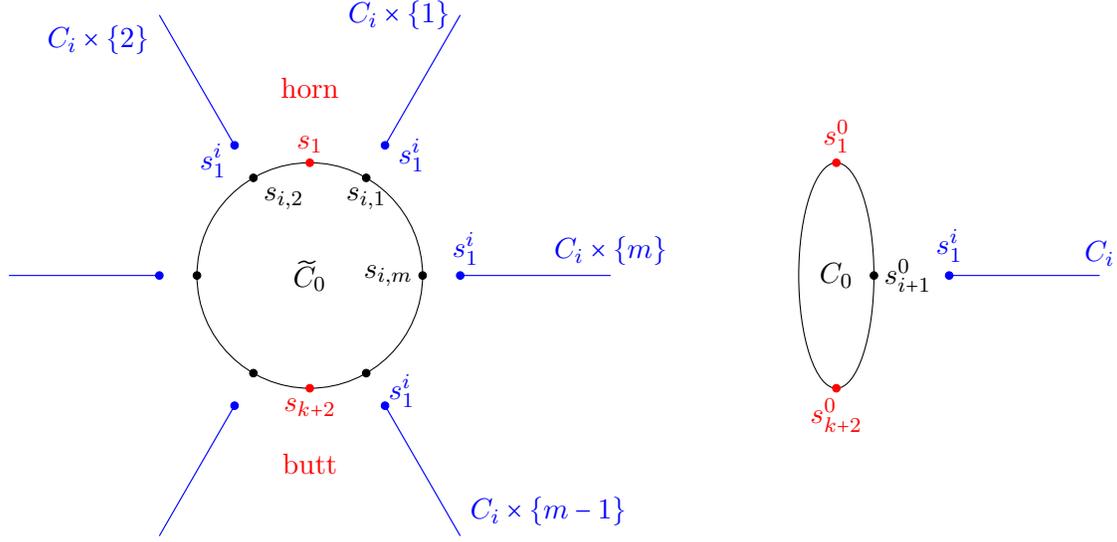
\begin{figure}
\centering
\begin{tikzpicture}
\draw (0,0) node{$\widetilde{C}_0$} circle[radius=1.5];
\filldraw (60:1.5) circle[radius=0.05] 
node[anchor=north]{$s_{i,1}$};
\filldraw (120:1.5) circle[radius=0.05]
node[anchor=north west]{$s_{i,2}$};
\filldraw (180:1.5) circle[radius=0.05]; 
\filldraw (240:1.5) circle[radius=0.05]; 
\filldraw (300:1.5) circle[radius=0.05];
\filldraw (0:1.5) circle[radius=0.05] 
node[anchor=east]{$s_{i,m}$};
\filldraw[red] (90:1.5) circle[radius=0.05]
node[above]{$s_1$};
\filldraw[red] (90:2.5) node{horn};
\filldraw[red] (270:1.5) circle[radius=0.05]
node[below]{$s_{k+2}$};
\filldraw[red] (270:0.8) node{butt};
\filldraw[red] (270:2.2)
circle[radius=0.05,fill=red]--
(270:3.6) node[right]{$C_{k+1}$};
\draw[red] (285:2.4) node{$s_{1}^{k+1}\ $ };
\filldraw[blue] (60:2) 
circle[radius=0.05, fill=blue]-- 
(60:4) node[left]{$C_i\times\{1\}$};
\filldraw[blue] (50:2.1) node{$s_1^i$};
\filldraw[blue] (120:2) 
circle[radius=0.05, fill=blue]-- 
(120:4) node[anchor=north east]{$C_i\times\{2\}$};
\draw[blue] (130:2) node{$s_1^i$} ;
\filldraw[blue] (180:2) 
circle[radius=0.05, fill=blue] --
(180:4); 
\filldraw[blue] (240:2) 
circle[radius=0.05, fill=blue] --
(240:4); 
\filldraw[blue] (300:2) 
circle[radius=0.05, fill=blue]-- 
(300:4) node[anchor=south west]{$C_i\times\{m-1\}$} ;
\draw[blue] (310:2) node{$s_{1}^i\ $ };
\filldraw[blue] (360:2) 
circle[radius=0.05, fill=blue]-- 
(360:4) node[anchor=south]{$C_i\times\{m\}$} ;
\draw[blue] (10:2.1) node{$s_1^i$};

\draw (7,0) node{$C_0$} circle[x radius=0.5, y radius=1.5];
\filldraw[blue] (8.5,0) 
circle[radius=0.05]-- 
(10.5,0) node[anchor=south]{$C_i$} ;
\draw[blue] (8.5,0.4) node{$s_1^i$};

\filldraw[red] (7,-2.2)
circle[radius=0.05,fill=red]--
(7,-3.6) node[right]{$C_{k+1}$};
\draw[red] (7.7,-2.2) node{$s_{1}^{k+1}\ $ };

\filldraw[red] (7,1.5) circle[radius=0.05] 
node[anchor=south]{$s^0_{1}$};
\filldraw (7.5,0) circle[radius=0.05] 
node[anchor=west]{$s^0_{i+1}$};
\filldraw[red] (7,-1.5) circle[radius=0.05]; 
\draw[red](7.7,-1.4) node {$s^0_{k+2}$};
\end{tikzpicture}
\caption{Kawasaki stratum $I_\zeta X_{0,1,d}$}\label{fig:stem}
\end{figure}
Suppose that 
$\C_0=(C_0,s^0_1,\dots,s^0_{k+2},f_0)$ is an orbifold stable map in 
$[X/\mu_m]^{\eta,1,\dots,1,\eta^{-1}}_{0,k+2,d_0}$ and that
$\C_i=(C_i,s^i_1,f_i,g_i)\in I_{\eta_i}X_{0,1,d_i}$ ($1\leq i\leq k+1$)
are such that $f_0(s^0_{i+1})=f_i(s^i_1)$ for all $1\leq i\leq k+1$. The
orbifold $C_0$ is a global quotient, that is,
$C_0=\left[\widetilde{C}_0/\mu_m\right]$ and an orbifold stable map $f_0:C_0\to
[X/\mu_m]$ is equivalent to the choice of a $\mu_m$-equivariant map
$\widetilde{f}_0:\widetilde{C}_0\to X$. The marked points $s^0_i$ can
be represented by the gerbes 
\begin{align}
\nonumber
s^0_1 & =[\{s_1\}/\mu_m],\\
\nonumber
s^0_{i+1} &=[\{s_{i,1},\dots, s_{i,m}\}/\mu_m] \quad (1\leq i\leq k),\\
\nonumber
s^0_{k+2} &=[\{s_{k+2} \}/\mu_m],
\end{align}
where $s_1$ and $s_{k+2}$ are the $\mu_m$-fixed non-singular points of
$\widetilde{C}_0$ and $\{s_{i,1},\dots,s_{i,m}\}$ is a regular $\mu_m$-orbit
of non-singular points. Note that there is a freedom in
choosing the $\mu_m$-action on $\widetilde{C}_0$: if we pick an
isomorphism $\phi:\mu_m\to \mu_m$ and define a new action by
$(g,z)\mapsto \phi(g)\cdot z$, then the quotient orbifold
$[\widetilde{C}_0/\mu_m]$ is equivalent to the original one. We
eliminate the freedom by requiring that the standard generator $\eta$ of
$\mu_m$ acts on the cotangent line $T^*_{s_1}\widetilde{C}_0$ by
multiplication by $\zeta$, that is, by the given primitive $m$-th root of 1. 
Let us define a nodal curve
\beq\label{sl-structure}
C=\widetilde{C}_0\coprod \left(
  \coprod_{i=1}^k C_i\times [1,m]
\right)
\coprod C_{k+1}/\sim,
\eeq
where the equivalence relation is such that the marked point $s_{i,a}\in
\widetilde{C}_0$ for $1\leq i\leq k$ is identified with $(s^i_1,a)\in C_i\times
\{a\}$ and $s_{k+2}$ is identified with $s^{k+1}_1$. Let us define an
automorphism $g$ of $C$. To begin with, let 
us denote by $g_0\in \operatorname{Aut}(\widetilde{C}_0)$ the
automorphism corresponding to the action of the generator $\eta\in
\mu_m$. Put
\begin{align}
\label{g0}
g(z) & := g_0(z)\quad \mbox{ for } z\in \widetilde{C}_0,\\
\label{gi}
g(z,a) & :=
\begin{cases}
(z,a+1) & \mbox{ if } z\in C_i \mbox{ and } 1\leq a\leq m-1,\\
(g_i(z),1) & \mbox{ if } z\in C_i \mbox{ and } a=m,
\end{cases} \quad \mbox{ for } 1\leq i\leq k, \\
  \label{gk+1}
  g(z) & := g_{k+1}(z) \quad \mbox{ for } z\in C_{k+1},
\end{align}
and 
\begin{align}
\nonumber
f(z) & := \widetilde{f}_0(z)\quad \mbox{ for } z\in \widetilde{C}_0,\\
\nonumber
f(z,a) & := f_i(z) \mbox{ if } z\in C_i (1\leq i\leq k) \mbox{ and }
         1\leq a\leq m,\\
\nonumber
f(z) & := f_{k+1}(z) \mbox{ if } z\in C_{k+1}.  
\end{align}
Clearly, $\C:=(C, s_1, f, g)$ is a stable map representing a point in
$I_\zeta X_{0,1,d}$.
\begin{proposition}[\cite{GT}]\label{prop:stems_isomorphism_components_of_inertia_stack}
  The map
  \[
    \varphi : \coprod_{\substack{d_0,\dots,d_{k+1}\\
    \eta_1,\dots,\eta_{k+1}}}
    \left(
    [X/\mu_m]^{\eta,1,\dots,1,\eta^{-1}}_{0,k+2,d_0}\times_{X^{k+1}}\Big(
    I_{\eta_1}X_{0,1,d_1}\times
    \cdots \times
    I_{\eta_{k+1}}X_{0,1,d_{k+1}}\Big)
    \right)
    \to
    I_\zeta X_{0,1,d}
  \]
constructed above (see Equation \eqref{stem_iso}) is an isomorphism of Deligne--Mumford stacks.
\end{proposition}

Constructing the map in the inverse direction is
in some sense the starting point in Givental and Tonita's work (see
\cite{GT}, Section 7). Let us recall their idea. 

\begin{proof}
Given a point
$\C:=(C, s_1, f, g)$ in $I_\zeta X_{0,1,d}$, the marked point $s_1$ on
$C$ is called the {\em horn}. Note that $g^m$ is an automorphism of
$\C$ acting trivially on the irreducible component of $C$ that carries the
horn. Let $\widetilde{C}_0$ be the maximal connected subcurve of $C$
that carries the horn, $g^m$ acts trivially on $\widetilde{C}_0$, and
the action of $g$ on $\widetilde{C}_0$ is balanced in the following sense. The
automorphism $g$ acts on $\widetilde{C}_0$ with exactly two fixed {\em
  non-singular points}: the horn $s_1$ and one more point called the
{\em butt} which we denote by $s_{k+2}$. Note that $\widetilde{C}_0$
is a chain of $\PP^1$s. Each $\PP^1$ in the chain is $g$-invariant
with two fixed points which can be uniquely given a name horn or butt
so that the butt of a given $\PP^1$ is identified with the horn of the
next $\PP^1$ in the chain. We require that the eigenvalue of $dg$ on a
tangent line at a horn (resp. butt) to be $\zeta^{-1}$
(resp. $\zeta$). The number $k$ here is defined as follows. Removing
$\widetilde{C}_0$ from $C$ we obtain a curve consisting of several
connected components. A connected component that does not contain the
butt $s_{k+2}$ is call a {\em leg}. A connected component containing
the butt is called the {\em tail}. The tail could be empty in which
case the butt $s_{k+2}$ is a regular point. Otherwise, if the tail is
non-empty, then the butt $s_{k+2}$ is a node of $C$.  The action of the
automorphism $g$ splits the set of points on $\widetilde{C}_0$ at which
the legs are attached, into $k$ groups each consisting of
$m$ points $s_{i,1},\dots,s_{i,m}$ forming an orbit of the cyclic
group $\langle g\rangle$. Clearly $g$ defines an isomorphism
between the legs attached to the points in the same orbit. Therefore,
we may assume that these legs are copies of the same curve $C_i$, that
is, we denote the leg attached to $s_{i,a}$ by $C_i\times \{a\}$. The
point on the leg identified with $s_{i,a}$ has the form $(s^i_1,a)$,
for some point $s^i_1\in C_i$. Furthermore, 
$g^m$ induces an automorphism $g_i$ of $C_i$ fixing the point
$s^i_1\in C_i$. The differential $dg_i$ acts on
$T_{s^i_1}C_i$ with some eigenvalue $\eta^{-1}_i\neq 1$, otherwise if
$\eta_i=1$, then $g_i$ will act trivially on the irreducible component
of $C_i$ that carries $s^i_1$, so the domain of $\widetilde{C}_0$ can be
extended. Clearly, $(C_i, s^i_1, f|_{C_i}, g_i)$ is a point in
$I_{\eta_i}X_{0,1,d_i}$ for some $d_i$ and $f|_{\widetilde{C}_0} $ induces an
orbifold stable map on the orbifold curve
$C_0:=\widetilde{C}_0/\langle g\rangle$. Finally, let $C_{k+1}$ be the
tail, $s^{k+1}_1\in C_{k+1}$ be the point at which the tail is
attached to $\widetilde{C}_0$, that is, $s^{k+1}_1=s_{k+2}$, and
$g_{k+1}$ be the restriction of $g$ to $C_{k+1}$. Then $(C_{k+1},
s^{k+1}_1, f|_{C_{k+1}}, g_{k+1})$ is a point in
$I_{\eta_{k+1}}X_{0,1,d_{k+1}}$ for some primitive root of unity
$\eta_{k+1}$ and a degree $d_{k+1}\in \operatorname{Eff}(X)$. Note
that $\eta_{k+1}\in \CC^*$ is such that $dg$ acts on
$T_{s^{k+1}_1}C_{k+1}$ by multiplication by $\eta_{k+1}^{-1}$ and that
$\eta_{k+1}\neq \zeta$ otherwise the chain $\widetilde{C}_0$ can
be extended. Therefore, this data defines
a point on the fiber product. We get a map in the inverse
direction of \eqref{stem_iso} which is the inverse that we were
looking for.  The above constructions are functorial and they can be done
in families, so \eqref{stem_iso} is an isomorphism of Delign--Mumford
stacks. Let us point out that the marked point $s_{k+2}\in
\widetilde{C}_0$ is either a regular point or a node of $C$. 
\end{proof}

\subsection{Reconstruction of the small $K$-theoretic $J$-function}\label{ssec:reconstruction_of_j_function}

Let us use Kawasaki's formula \eqref{KHRR} to express the small
$K$-theoretic $J$-function as an integral over the inertia stacks
$IX_{0,1,d}=\coprod_{m\geq 1}\coprod_\zeta I_\zeta X_{0,1,d}$, where
the 2nd disjoint union is over all primitive $m$-th roots of 1. We get 
\begin{align}
\nonumber
J(q,Q) = & 1-q + \tau(q,Q) + \\
\nonumber
 & + \sum_{d,a} Q^d \Phi_a
\int_{[I_1X_{0,1,d} ]   } 
\operatorname{td} (T_{IX_{0,1,d}} ) \, 
\frac{1}{
\operatorname{ch}\circ
\operatorname{Tr}(
\wedge^\bullet(N^\vee_{IX_{0,1,d}})
)} \, 
\frac{
\operatorname{ev}_1^*(
\operatorname{ch}(\Phi^a)  )}{
1-q e^{\psi_1} }  ,
\end{align}
where $\tau(q,Q)=\sum_{m>1} \sum_\zeta \tau_\zeta(q,Q)$ and  
\beq\label{zeta_stratum}
\tau_\zeta(q,Q):= 
\sum_{d,a} Q^d \Phi_a
\int_{[I_\zeta X_{0,1,d}]} 
\operatorname{td} (T_{IX_{0,1,d}} ) \, 
\frac{1}{
\operatorname{ch}\circ
\operatorname{Tr}(
\wedge^\bullet(N^\vee_{IX_{0,1,d}})
)} \, 
\operatorname{ch}\circ
\operatorname{Tr}
\left(
\frac{
\operatorname{ev}_1^* \Phi^a}{
1-q L_1}
\right) .
\eeq
According to Givental--Tonita (see \cite{GT}, Proposition 1), the
integral over $[I_1X_{0,1,d}]$ can be expressed in terms of fake
$K$-theoretic invariants of $X$, that is, 
\begin{theorem}[\cite{GT}, Proposition 1]
The $K$-theoretic small $J$-function can be expressed in terms of the
fake $K$-theoretic $J$-function as follows:
\beq\label{fake-J}
J(q,Q)=1-q + \tau+
\sum_{d,a} \sum_{l=0}^\infty 
\frac{Q^d \Phi_a }{l!}
\left\langle
\frac{
\Phi^a }{
1-q L_1} ,\tau,\dots,\tau
\right\rangle^{\rm fake}_{0,1+l,d}\quad,
\eeq
where $\tau$ is defined by equation \eqref{zeta_stratum}.
\end{theorem}
This follows
from the isomorphism that we have constructed in Proposition \ref{prop:stems_isomorphism_components_of_inertia_stack} for the case $m=1$ and $\zeta=1$ after analyzing how
the inertia tangent and the inertia normal bundles decompose under the
isomorphism.   

The remaining problem is to compute $\tau(q,Q)$. To begin with, let us
introduce the notation
\ben
\delta\tau_\zeta(q,Q):=
1-q + \sum_{\eta: \eta\neq \zeta} \operatorname{ch}(\tau_\eta(q,Q)),
\een
where the sum is over all primitive roots $\eta$ of $1$ different from
$\zeta$.  The answer is given by the following relation.
\begin{theorem}[\cite{GT}, Propositions 3 and 4]\label{thm:confluence_qHHR_recursion_formula}
The integrals $\tau_\zeta(q,Q)$, defined by \eqref{zeta_stratum},
satisfy the following recursive relations: 
\begin{align}
\label{GT-recursion}
&
\tau_\zeta(q,Q) = \sum_{d_0, a}\sum_{k=0}^\infty
\frac{Q^{md_0} \Phi_a}{k!} \times\\
\nonumber
&
\left[
\frac{
\operatorname{ch}(\Phi^a)}{
1-q \zeta e^{\psi_1/m}} ,
\tau^{(m)}(\psi_2,Q),\dots,
\tau^{(m)}(\psi_{k+1},Q), 
\delta\tau_\zeta(\zeta^{-1} e^{\psi_{k+2}/m},Q)
\right]_{0,k+2,d_0},
\end{align}
where $[ \, \cdots \, ]_{0,k+2,d_0}$ is the stem correlator of Definition \ref{def:stem_stem_invariants} and $\tau^{(m)}(z,Q)\in H(X;\CC)[\![z,Q]\!]$ is
obtained from $\tau(q,Q)$ via the Chern character map and the {\em
  Adam's operations}. To define the action Adam's operations on $\tau$, put
\beq\label{tau-series}
\tau(q,Q)=:\sum_{d,a} \tau_{d,a} (q) Q^d \Phi_a,
\eeq
where $\tau_{d,a}(q)$ is a rational function in $q$. Then 
\beq\label{leg_contr}
\tau^{(m)}(z,Q):= \sum_{d,a} 
\tau_{d,a}(e^{mz}) Q^{md} 
\operatorname{ch}(\Psi^m(\Phi_a)),
\eeq
where $\Psi^m:K^0(X)\to K^0(X)$ are the Adam's operation, that
is, ring homomorphisms which on line bundles $\ell$ are defined to be
$\Psi^m(\ell)=\ell^{\otimes m}$.
\end{theorem}
The reason why the above relation is
a recursion follows immediately from the observation that if we put a
lexicographical order on $d:=(d_1,\dots,d_r):=(\langle p_1,d\rangle,\dots, \langle
p_r,d\rangle)$ and compare the coefficients in front of $Q^d
=Q_1^{d_1}\cdots Q_r^{d_r}$, then the RHS will involve only the
components $\tau_{d',a'}(q)$ of $\tau(q,Q)$ for which $d'<d$. 

\subsection{Confluence of the small $K$-theoretic $J$-function}\label{ssec:confluence_of_j_function}

We now have the tools to prove the confluence of the small
$K$-theoretic $J$-function to its cohomological analogue. 
Our goal is to prove the following theorem.

\begin{theorem}\label{t1}
  If $X$ is a smooth projective variety, such that, the anti-canonical
  bundle $K_X^\vee$ is nef, then the limit
  \ben
  \lim_{q\to 1}
  (q-1)^{\operatorname{deg}-1}\, 
  \operatorname{ch}\Big(
  J(q, (q-1)^{m_1} Q_1,\dots, (q-1)^{m_r} Q_r)
  \Big)
  \een
  exists and it coincides with $J^{\rm coh}(1,Q)$. 
\end{theorem}
It is natural to ask whether the $q$-difference system satisfied by the $K$-theoretic
$J$-function has also a similar limit when $q \to 1$. We expect that
the techniques from our proof of Theorem \ref{t1} can be used to prove
that if we pullback the variables $Q_1, \dots, Q_r$ to $(q-1)^{m_1} Q_1,\dots,
(q-1)^{m_r} Q_r$, then the
formal limit as $q\to 1$ of the resulting system of
$q$-difference equations exists and it coincides with the system of
differential equations satisfied by the small cohomological
$J$-function, evaluated at $z=1$. 

Note that using Equation \eqref{eqn:intro_cohom_J_recover_z}, we can recover the cohomological $J$-function for all $z$.
We may assume that $\phi_i= \operatorname{ch}(\Phi_i)$ is a
homogeneous cohomology class. Let us assume also that $\Phi_1=1$. Let
us denote the complex degree (i.e. half of the 
standard cohomology degree) of $\phi_i$ by $|\phi_i|$. First, we will
prove the following lemma.

\begin{lemma}
The limit
\ben
\lim_{q\to 1}
(q-1)^{\operatorname{deg}-1}
\operatorname{ch}\Big(
\tau(q , (q-1)^{m_1}Q_1,\dots, (q-1)^{m_r} Q_r) 
\Big)
\een
exists and it is $0$.
In other words, using the decomposition $\tau(q,Q)=\sum_{d,a} \tau_{d,a} (q) Q^d \Phi_a$ of Equation \eqref{tau-series}, we claim that
\ben
\tau_{d,a} (q)=0\quad 
\forall 
\mbox{$d$ and $a$, s.t., }
|\phi_a| -1 + \int_d c_1(T_X) \leq 0.  
\een
\end{lemma}

\begin{proof}
We argue by induction on the lexicographical order of $d$. Note that
for $d=0$, $\tau_{0,a}=0$ for all $a$, because $\tau$ would be a sum of
integrals over the virtual fundamental cycle of $I_\zeta X_{0,1,0}$,
but due to the stability conditions the moduli spaces $X_{0,1,0}$ are
empty.  Let us compare the coefficients in front of $Q^d \phi_a$ in
\eqref{GT-recursion}. Let us expand the $i$-th insertion involving $\tau^{(m)}$
on the RHS as 
\ben
\sum_{d_i,a_i} 
\tau_{d_i,a_i}(e^{m\psi_{i+1}}) Q^{md_i} \phi_{a_i} m^{|\phi_i|}
\een
and the last insertion as
\ben
1-\zeta^{-1} e^{\psi_{k+2}/m} + 
\sum_{\eta_{k+1}: \eta_{k+1}\neq \zeta}
\sum_{d_{k+1}\neq 0, a_{k+1}}
\tau_{\eta_{k+1}, d_{k+1}, a_{k+1}}(\zeta^{-1} e^{\psi_{k+2}/m})
Q^{d_{k+1}} \phi_{a_{k+1}},
\een
where the sum over $d_{k+1}$ involves only $d_{k+1}\neq 0$ (so that
the tail is non-empty). 
We have $d=m(d_0+d_1+\cdots+d_k)+d_{k+1}$ and hence $d_i\leq d$ for all
$0\leq i\leq k+1$. We claim that $d_i<d$ for all $0\leq i\leq
k+1$. Indeed, suppose first that $d_{k+1}=d$, then $d_i=0$ for all
$0\leq i\leq k$. By stability the moduli spaces $X_{0,1,d_i}$ are
empty for all $1\leq i\leq k$. We get that there are no legs, that is,
$k=0$. Again by stability, since the stem moduli space has degree
$d_0=0$ and only two marked points, the stem invariant must be
$0$. The inequalities $d_i<d$ for all $0\leq i\leq k$ follow from
$m>1$. Our claim follows. Hence we may recall the
inductive assumption, that is, if $d_{k+1}\neq 0$, then
\beq\label{deg_ineq}
|\phi_{a_i}| -1 + \int_{d_i} c_1(T_X)  > 0, \quad 1\leq i\leq k+1.
\eeq
If $d_{k+1}=0$, then the last insertion becomes $1-\zeta^{-1}
e^{\psi_{k+2}/m}$. Therefore, in both cases $d_{k+1}\neq 0$ and
$d_{k+1}=0$, we have
\ben
|\phi_{a_{k+1}}| + \int_{d_{k+1}} c_1(T_X)  \geq 0.
\een
On the other hand, the stem correlator is defined through integration along the virtual
fundamental cycle of $[X/\mu_m]^{(\eta,1,\dots,1, \eta^{-1})}_{0,k+2,d_0}$ which has
complex dimension
\ben
3\cdot 0-3 +k+2 +D +\int_{d_0} c_1(X),
\een
where $D$ is the complex dimension of $X$. 
While the total degree of the cohomology classes inside the correlator
is at least 
\ben
|\phi^a| + \sum_{i=1}^{k+1} |\phi_{a_i}|\geq 
D-|\phi_a|+ k - \sum_{i=1}^{k+1} \int_{d_i} c_1(X).
\een
Comparing with the dimension of the virtual fundamental cycle we get 
\ben
|\phi_a|-1 + \sum_{i=0}^{k+1}\int_{d_i} c_1(X) \geq 0.
\een
Since $m>1$ and $K_X^\vee$ is nef we have
\ben
\int_d c_1(X) = m
\sum_{i=0}^k \int_{d_i} c_1(X) + \int_{d_{k+1}} c_1(X) \geq 
\sum_{i=0}^{k+1}\int_{d_i} c_1(X) .
\een
Therefore, the above inequality implies that if
$\tau_{d,a}\neq 0$, then $|\phi_a|-1 + \int_d c_1(X)\geq 0$. We need
only to check that the equality in the last inequality is not
possible.  We have to consider only contributions to the coefficient in
front of $Q^d\phi_a$ in the stem correlator \eqref{GT-recursion} for
which the set of inequalities \eqref{deg_ineq} is empty -- any of
these inequalities will destroy the equality. This happens only if
$k=0$ and $d_{k+1}=0$. The stem correlator takes the form
\ben
\left[
    \frac{\operatorname{ch}(\Phi^a)}{1-q\zeta e^{\psi_1/m}} ,
    1-\zeta^{-1} e^{\psi_2/m}
    \right]_{0,2,d_0},
\een
where $d=m d_0$. The above stem correlator is an integral against the
virtual fundamental cycle of $[X/\mu_m]^{\eta,\eta^{-1}}_{0,2,d_0}$ of
a cohomology class of the form
$\operatorname{const}\times \operatorname{ev}_1^*(\phi^a) + \cdots$,
where the dots stand for a cohomology class of complex degree $\geq
|\phi^a| + 1$. On the other hand, according to our assumption
$|\phi_a|=1-\int_d c_1(X)$. Since $|\phi_a|\geq 0$ and $m>1$, we get
that $\int_{d_0} c_1(X)=0$ -- if not then since $K_X^\vee $ is nef
$\int_{d_0} c_1(X)\geq 1$, so $1-\int_d c_1(X)\leq 1-m <0$. We get
that $|\phi_a|=1$ $\Rightarrow$ $|\phi^a|=D-1=$ the dimension of the
virtual fundamental cycle of
$[X/\mu_m]^{\eta,\eta^{-1}}_{0,2,d_0}$. Therefore, the dots of the
above cohomology class do not contribute to the stem
correlator. Finally, the stem correlator takes the form
\ben
\operatorname{const}
\left\langle \phi^a\, \mathbf{1}_\eta,
  \mathbf{1}_{\eta^{-1}}
\right\rangle^{\eta,\eta^{-1}}_{0,2,d_0} =
\operatorname{const}
\left\langle
  \phi^a, 1\right
\rangle_{0,2,d_0} = 0,
\een
where for the first equality we used the results of Jarvis--Kimura
(see \cite{JK}) to express the orbifold GW invariants of $[X/\mu_m]$
in terms of the GW invariants of $X$: the orbifold GW invariant
differs from the GW invariant by a factor equal to the degree of the
forgetfull map $[X/\mu_m]^{\eta,\eta^{-1}}_{0,2,d_0}\to X_{0,2,d_0}$.
For the second equality we used that $d_0\neq 0$ (due to stability)
and the string equation. This completes the proof.  
\end{proof}

\begin{proof}[Proof of Theorem \ref{t1}]
Let us first prove that the limit in Theorem \ref{t1} exists. In order
to do this, let us recall formula \eqref{fake-J}. The first three
terms, that is, $1-q+ \tau$ will contribute to the limit $-1$. 
Let us expand the $i$-th insertion
of $\tau$ as $\tau=\sum_{d_i,a_i} \tau_{d_i,a_i}(L_{i+1})
Q^{d_i}\Phi_{a_i}$ and express the fake $K$-theoretic correlator as a
twisted GW invariant. We get a sum of terms of the form
\ben
\frac{1}{l!}
\Phi_a Q^{d+d_1+\cdots+d_l} 
\left\langle
\frac{\operatorname{ev}_1^*(\phi^a/\operatorname{td}(X))}{
1-q e^{\psi_1}}, 
\tau_{d_1,a_1}(e^{\psi_2}) \phi_{a_1},\dots, 
\tau_{d_l,a_l}(e^{\psi_{l+1}}) \phi_{a_l}
\right\rangle_{0,1+l,d}^{ABC}.
\een
Note that in the definition of the K-theoretic J-function (see
Definition \ref{def:intro_k_theoretical_j_function}) we exchanged the
role of the bases $\{\Phi_a\}$ and $\{\Phi^a\}$, that is, here the
dual bases is inside the correlator. This transformation leaves the
J-function invariant because $\{\Phi_a\}$ and $\{\Phi^a\}$ are dual
bases. 
Note also that the first insertion in the above correlator has an
expansion at $q=1$ of the form:
\beq\label{1st-ins}
\sum_{k=0}^\infty
(1-q)^{-k-1} (\phi^a +\cdots)(  q^k \psi^k + \cdots),
\eeq
where the dots stand for cohomology classes of higher degrees.
Applying the operator $(q-1)^{\operatorname{deg}-1} \circ \operatorname{ch}$ we get that the
contribution to the correlator of the J-function corresponding to the
$k$-th term of the sum \eqref{1st-ins} is proportional to $(q-1)^{|\phi_a|-2-k}$.
Rescaling the Novikov variables by $Q_i\mapsto (q-1)^{m_i}Q_i$ will
rescale the monomial $Q^{d+d_1+\cdots + d_l}=\prod_{i=1}^r
Q_i^{\langle p_i, d+d_1,+\cdots + d_l\rangle}$ by $(q-1)^{\int_d
  c_1(X) +\int_{d_1} c_1(X)+\cdots + \int_{d_l} c_1(X)}$.  
Therefore, 
we have to prove that if the above correlator is not 0, then 
$|\phi_a|-2-k+\int_{d} c_1(X)+\sum_{i=1}^l \int_{d_i} c_1(X)\geq 0$.
According to the claim that we already proved, if $\tau_{d_i,a_i}\neq
0$, then $|\phi_{a_i}|-1 + \int_{d_i} c_1(X) > 0$, that is,
\beq\label{estimate}
2-\int_{d_i} c_1(X)\leq 
|\phi_{a_i}|. 
\eeq
Note that the
above twisted correlator is defined by integrating along the virtual
fundamental cycle of $X_{0,1+l,d}$ which has dimension $l-2+ D
+\int_{d} c_1(X)$. Therefore, by comparing the degree of the
cohomology classes in the correlator, we get the inequality
\ben
D-|\phi_a| +k +\sum_{i=1}^l |\phi_{a_i}| \leq 
l-2 + D +\int_{d} c_1(X). 
\een 
Recalling our estimate \eqref{estimate}, we get 
\ben
l\leq |\phi_a|-2-k + \int_{d} c_1(X) + \sum_{i=1}^l \int_{d_i} c_1(X). 
\een
This proves that the limit exists. In order to compute the limit, we
need to single out only the terms for which the RHS of the above
equality is $0$. But this would be the case only if $l=0$,
the dots (i.e. the higher degree terms) in \eqref{1st-ins} and the higher
degree terms of the ABC-twisting class $\Theta^{ABC}_{0,1,d}=1+\cdots$
are prepended by a factor of the form $(q-1)^p$, with $p>0$ thanks to our estimations.
When computing the limit when $q \to 1$, these higher degree terms vanish.
Thus, the limit becomes 
\ben
-1 + \sum_{d,a} \sum_{k=0}^\infty 
Q^d \phi_a
\left\langle
\phi^a\psi^k
\right\rangle_{0,1,d}\, (-1)^{k+1}= J^{\rm coh} (1,Q).
\een
\end{proof}
\section{Oscillatory integrals in Gromov--Witten theories}\label{sec:oscillatory_integral_and_gamma}

%
%
%
%
%

In this Section, we will revisit Givental's proposal for mirror
symmetry in toric geometry. 
More precisely, if $X$ is a Fano toric manifold of Picard rank 2, then we would like to
compare the following two solutions of the cohomological quantum
differential equations  and their $K$-theoretic analogues: 
\begin{itemize}
	\item[(i)]
	The $I$-function, which is a certain hypergeometric-type
        series associated to the toric manifold (see e.g.
        \cite{Givental:equivariant_GW_invariants} in quantum cohomology and
        \cite{Givental:QK_fixed_point_localization} in quantum
        $K$-theory) 
	\item[(ii)]
	Givental's oscillatory integral model in quantum cohomology
        \cite{Givental:toric_mirrors}  (see also Section 3 of \cite{Iritani:gamma_structure}) and its $q$-analogue in quantum $K$-theory (which will be introduced in Definition \ref{def:qosc_q_oscillatory_integral_real}).
\end{itemize}
The main goal in this Section is to provide a comparison of the
$K$-theoretic $I$-function and the $q$-oscillatory integral, which
will be done in Theorem
\ref{thm:gamma_comparison_theorem_in_quantum_k_theory}. We believe
that our results can be extended to all compact toric orbifolds. The
Picard rank 2 case is a natural case to investigate, because it includes
both Fano and non-Fano manifolds. However, the non-Fano
case will be pursued elsewhere. 

\subsection{Symplectic toric manifolds}\label{ssec:toric_manifolds}

%
%
%

In this Subsection, we recall some standard definitions from toric
geometry, explain some of their geometrical properties and give an
explicit construction of all symplectic toric Fano manifolds whose
Picard rank is 2. For more details on the subject, we refer to Chapter
7 of \cite{Audin:Symplectic_manifolds}. 

We consider the manifold $X$ to be a symplectic toric manifold, that is a smooth symplectic (or GIT) quotient $X = \mathbb{C}^n // \mathbb{T}^k$ of the $n$-dimensional symplectic vector space by a linear action of a $k$-dimensional torus.
We will describe such a toric manifold through its moment map.

\begin{definition}[Symplectic toric orbifolds]
	\begin{itemize}
	\item[(i)]
	Suppose that $\mu : \mathbb{Z}^n \to \mathbb{Z}^r$ is a linear
        map and let us denote by $\alpha_1, \dots, \alpha_n \in
        \mathbb{Z}^r$ the images of the canonical basis of
        $\mathbb{Z}^n$ by the moment map $\mu$. Let
        $\textnormal{Mat}(\mu) := \left( m_{ij} \right)$ be the matrix
        of $\mu$, that is, the entries in the $i$th column are the
        coordinates of $\alpha_i$. Slightly abusing the terminology,
        we will refer to $\mu$ as the \textit{moment map}. 
	The \textit{cone associated to the moment map} $\mu$ is
        defined to be the cone $\textnormal{Cone}(\mu)$ in
        $\mathbb{R}^r$ spanned by the elements $\alpha_1, \dots,
        \alpha_n \in \mathbb{R}^r \supset \mathbb{Z}^r$. 

	\item[(ii)]
	We denote by $\textnormal{BInd}(\mu)$ the set
	\[
		\textnormal{BInd}(\mu)
		:=
		\left\{
			\underline{\sigma}=(\sigma_1,\dots,\sigma_r) \subseteq \{1 , \dots, n \}
		\, \middle| \,
			\mathbb{Q}^r
			=
			\textnormal{Vect}_\mathbb{Q}
			\left(
				\alpha_{\sigma_1}, \dots, \alpha_{\sigma_r}
			\right)
		\right\},
              \]
              where $\textnormal{Vect}_\mathbb{Q}(v_1,\dots,v_r)$
              denotes the vector space spanned over $\mathbb{Q}$ by vectors
              $v_1,\dots,v_r$.  
	The \textit{singular cone associated to a moment map} $\textnormal{Cone}_\textnormal{sing}(\mu) \subseteq \mathbb{R}^r$ is the union of boundaries given by
	\[
		\textnormal{Cone}_\textnormal{sing}(\mu)
		=
		\bigcup_{\underline{\sigma} \in \textnormal{BInd}(\mu)}
		\partial \textnormal{Cone}
		\left(
			\alpha_{\sigma_1}, \dots, \alpha_{\sigma_r}
		\right),
              \]
              where
              $\textnormal{Cone}(v_1,\dots,v_r):=\mathbb{R}_{\geq
                0}v_1+\cdots +\mathbb{R}_{\geq 0}v_r$ denotes the real cone
              spanned by $v_1,\dots,v_r$.  
	A \textit{chamber} $K \subseteq \mathbb{R}^r$ is a connected component of $\textnormal{Cone}(\mu) - \textnormal{Cone}_\textnormal{sing}(\mu)$.

	\item[(iii)]
	Consider the composition of maps
	\[
		\begin{tikzcd}[
                cramped,row sep = 0cm, every text node part/.style={align=center}, node distance=0.1cm
            ]
			\mathbb{C}^n
				\arrow[r]
				\arrow[rr,bend left,"f"]
			&
			\mathbb{R}^n
				\arrow[r,"\mu"]
			&
			\mathbb{R}^r
			\\
			z_j
				\arrow[r,mapsto]
			&
			\mid z_j \mid^2
			&
		\end{tikzcd}
	\]
	Given a moment map $\mu$ and a chamber $K$, we define a symplectic manifold $X_{\mu, K}$ to be the quotient
	\[
		X_{\mu, K}
		:=
		f^{-1}(K)/\mathbb{T}^r,
	\]
	where the action of the $r$-dimensional torus $\mathbb{T}^r = \left( \mathbb{C}^* \right)^r$ on $\mathbb{C}^n \supset f^{-1}(K)$ is given by the matrix of the moment map $\mu$:
	$(t \cdot z)_j = t_1^{m_{1j}}\cdots t_r^{m_{rj}} z_j$ for
        $t=(t_1,\dots,t_r)\in \mathbb{T}^r$ and $z=(z_1,\dots,z_n)\in \mathbb{C}^n$.
	\end{itemize}
      \end{definition}
Note that $f^{-1}(K)$ is not $\mathbb{T}^r$-invariant. The
smallest $\mathbb{T}^r$-invariant subset of $\mathbb{C}^n$
containing $f^{-1}(K)$ is the following open subset of $\mathbb{C}^n$:
\[
  U_K = \bigcup_{I:K\subset \textnormal{Cone}(\alpha_I)}
  \mathbb{C}^{\overline{I}}\times (\mathbb{C}^*)^I,
\]
where the union is over all subsets $I=\{i_1,\dots,i_m\}\subset \{1,2,\dots,n\}$, such
that, the cone $\textnormal{Cone}(\alpha_I):=
\textnormal{Cone}(\alpha_{i_1},\dots, \alpha_{i_m})$ contains the
chamber $K$, $\overline{I}$ is the complement of $I$ in
$\{1,2,\dots,n\}$, and
\begin{align}
  \nonumber
  \mathbb{C}^I & := \{z\in \mathbb{C}^n\ |\ j\notin I \Rightarrow
                 z_j=0\},\\
  \nonumber
  (\mathbb{C}^*)^I & :=\{z\in \mathbb{C}^n\ |\ j\notin I \Leftrightarrow
                 z_j=0\}.
\end{align}
Strictly speaking we should define $X_{\mu,K}=U_K/\mathbb{T}^r$. By
definition,  $f^{-1}(\eta)\subset U_K$ for all $\eta\in K$. It is a
non-trivial result (see \cite{Audin:Symplectic_manifolds},
Theorem VII.2.1) that every $\mathbb{T}^r$-orbit in $U_K$ intersects
$f^{-1}(\eta)$ along a $(\mathbb{S}^1)^r$-orbit, that is,
$X_{\mu,K}=f^{-1}(\eta)/(\mathbb{S}^1)^r$ is a symplectic reduction,
where $\mathbb{S}^1\subset \mathbb{C}$ is the unit circle. 

  \begin{proposition}\label{prop:gamma_toric_manifold_properties}
	Consider a moment map $\mu$ and a chamber $K$, and denote the resulting symplectic toric variety by $X_{\mu,K}$.
	\begin{itemize}
		\item[(i)]
		The quotient $X_{\mu,K}$ is compact if and only if the cone associated to its moment map $\textnormal{Cone}(\mu)$ is contained in a half space of $\mathbb{R}^r$.

		\item[(ii)]
		The quotient $X_{\mu,K}$ is smooth if and only if for any $\underline{\sigma} \in \textnormal{BInd}(\mu)$ such that
		\[
		K \cap \textnormal{Cone}
		\left(
			\alpha_{\sigma_1}, \dots, \alpha_{\sigma_r}
		\right)
		\neq
		\varnothing,
		\]
		the linear map $\mu$ restricted to $\textnormal{Vect}\left(\alpha_{\sigma_1}, \dots, \alpha_{\sigma_r}\right)$ has determinant $\pm 1$.

		\item[(iii)]
		The quotient $X_{\mu,K}$ is Fano if and only if the vector $c_1(TX_{\mu,K}):=\alpha_1 + \cdots + \alpha_n \in \mathbb{R}^r$ is an element of the chamber $K$.
	\end{itemize}
\end{proposition}

\begin{remark}\label{rmk:gamma_toric_manifold_dimension_dual_cone}
	A cone $C$ of $\mathbb{R}^r$ is contained in a half space if and only if its dual cone has maximal dimension, i.e. $\textnormal{dim } C^\vee = r$
      \end{remark}
      Let $P_j=(U_K\times \mathbb{C})/\mathbb{T}^r$ ($1\leq j\leq r$) be the
line bundle on $X_{\mu,K}$, where the action of $\mathbb{T}^r$ on
$\mathbb{C}$ is given by the character $\mathbb{T}^r\to \mathbb{C}^*$,
$(t_1,\dots,t_r)\mapsto t_j^{-1}$. Let $p_j:=-c_1(P_j)$. 
\begin{proposition}\label{prop:gamma_cohomology_k_theory_of_toric_manifolds}
	Let $X_{\mu,K}$ be a symplectic toric manifold.
	Then, the cohomology ring and the topological $K$-ring of $X_{\mu,K}$ are given by
	\begin{align*}
		H^*(X_{\mu,K}; \mathbb{Q})
		&\simeq
		\mathbb{Q}[p_1, \dots, p_r]/\left(\prod_{j \in J_\nu} \alpha_j(p) \right)_\nu,
		\\
		K(X_{\mu,K})
		&\simeq
		\mathbb{Z}[P_1^{\pm 1}, \dots, P_r^{\pm
           1}]/\left(\prod_{j \in J_\nu} (1-U_j(P))
           \right)_\nu ,
	\end{align*}
	where
$\alpha_j(p) := m_{1j} p_1 + \cdots + m_{rj}p_r$,
$U_j(P) :=P_1^{m_{1j}}\cdots P_r^{m_{rj}}$,
	and $J_\nu=(j_{\nu,1},\dots,j_{\nu,l}) \subseteq \{1,
        \dots, N\}$ are the maximal subsets with respect to inclusion,
        such that, the cone spanned by $\alpha_{j_{\nu,1}}, \dots,
        \alpha_{j_{\nu,l}}$ does not intersect the chamber $K$.
\end{proposition}

\begin{example}
	Let us consider for the manifold $X = \textnormal{Bl}_{pt} \mathbb{P}^3 \simeq \mathbb{P}( \mathcal{O}_{\mathbb{P}^2} \oplus \mathcal{O}_{\mathbb{P}^2}(-1) )$ the following toric data: we use the moment map given by the matrix
	\[
		M =
		\begin{pmatrix}
			1	&	1	&	1	&	0	&	-1	\\
			0	&	0	&	0	&	1	&	1
		\end{pmatrix},
	\]
	and the chamber $K = \left( \mathbb{R}_{> 0} \right)^2$. We give below a figure of the toric data.
	
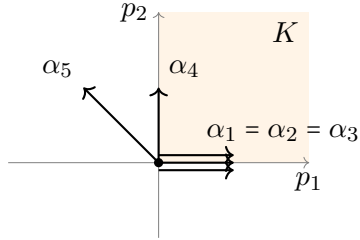
\begin{center}
		\begin{tikzpicture}
			\path[fill=LightO, fill opacity=.2] (0,0) -- (2,0) -- (2,2) -- (0,2) -- cycle;
			\draw[help lines,->] (-2,0) -- (2,0) coordinate (xaxis);
			\draw[help lines,->] (0,-1) -- (0,2) coordinate (yaxis);

			\draw
				(0,0)  node[scale=3](O){.}
			;
			\path[draw,line width=0.8pt,->] (0,0) -- (1,0);
			\path[draw,line width=0.8pt,->] (0,-0.1) -- (1,-0.1);
			\path[draw,line width=0.8pt,->] (0,0.1) -- (1,0.1);
			\path[draw,line width=0.8pt,->] (0,0) -- (0,1);
			\path[draw,line width=0.8pt,->] (0,0) -- (-1,1);

			\node[below] at (xaxis) {$p_1$};
			\node[left] at (yaxis) {$p_2$};
			\node[right] at (0.5,0.4) {$\alpha_1=\alpha_2=\alpha_3$};
			\node[above right] at (0,1) {$\alpha_4$};
			\node[above left] at (-1,1) {$\alpha_5$};
			\node[below left] at (2,2) {$K$};
		\end{tikzpicture}
	\end{center}
	Then, the sets $J_\nu$ of Proposition
        \ref{prop:gamma_cohomology_k_theory_of_toric_manifolds} are
        $\{1,2,3\}$ and $\{4,5\}$. The cohomology ring is given by 
	\[
		H^*(X_{\mu,K}; \mathbb{Q})	
		\simeq
		\mathbb{Q}[p_1, p_2]/\left(
			p_1^3 =0, p_2(p_2-p_1)=0
		\right),
	\]
and the topological $K$-ring by 
\[
K(X_{\mu,k}) \cong \mathbb{Z}[P^{\pm 1}_1,P^{\pm 1}_2]/\left( (1-P_1)^3=0, (1-P_2)(1-P_1^{-1}P_2)=0\right).
\]
\end{example}

\begin{proposition}\label{prop:gamma_model_toric_picard_rank_two}
Suppose that $X$ is a compact toric manifold with Picard rank
2. The isomorphism class of $X$ can be represented by a toric manifold
$X_{\mu,K}$, such that, the chamber 
$K=\left(\mathbb{R}_{>0}\right)^2$ 
and the matrix of the moment map $\mu$ has the following form:
	\[
\operatorname{Mat}(\mu)=
		\begin{pmatrix}
			1
			&
			\cdots
			&
			1
			&
			0
			&
			-a_1
			&
			\cdots
			&
			-a_k
			\\
			0
			&
			\cdots
			&
			0
			&
			1
			&
			1
			&
			\cdots
			&
			1
		\end{pmatrix},
	\]
where the first column $\binom{1}{0}$ is being repeated $N > 0$ times
and $a_j \in \mathbb{Z}_{\geq 0}$ ($1\leq j\leq k$). Furthermore,
$X_{\mu,K}$ is naturally isomorphic to a projectivised vector bundle,
that is, 
	\[
		X_{\mu,K}
		\simeq
		\mathbb{P}
		\left(
			\mathcal{O} \oplus \mathcal{O}(-a_1) \cdots \oplus \mathcal \mathcal{O}(-a_k) 
		\right)
		\to
		\mathbb{P}^{N-1}.
	\]
Finally, $X_{\mu,K}$ is Fano if and only if the inequality $N > a_1 + \cdots + a_k$ holds.
\end{proposition}

\begin{proof}
	Let $X_{\mu,K}=f^{-1}(K)/\mathbb{T}^r$ be a symplectic toric
        manifold of Picard rank $r$. Denote by
        $M:=\textnormal{Mat}(\mu) \in M_{r,n}(\mathbb{Z})$ the matrix
        of its moment map $\mu$, where $M_{r,n}(\mathbb{Z})$ denotes
        the space of matrices of size $r\times n$ with integer
        entries. 
	The action of $GL_{r}(\mathbb{Z})$ on $M_{r,n}(\mathbb{Z})$ by left multiplication $M \mapsto A \cdot M$ corresponds to changing the coordinates of the torus $\mathbb{T}^r$.
	Identifying the moment map $\mu$ with its matrix $M$, we have
        $X_{M,K}=X_{AM,A(K)}$. 
	Moreover, permuting the columns of the matrix $A$ amounts to a relabeling of the canonical basis of $\mathbb{C}^n \supset f^{-1}(K)$, therefore the quotient $X_{\mu,K}$ is still the same manifold.

Now, let us assume that the toric manifold $X_{\mu,K}$ has Picard rank
2. Note that the chamber $K$ is the interior of a cone
$\textnormal{Cone}(C_1,C_2)$, where $C_1$ and $C_2$ are columns of
$M$.  By permuting the columns of the matrix $M$ we may assume
that $C_1$ and $C_2$ are the 1st two columns of $M$.  By the
smoothness condition of Proposition
\ref{prop:gamma_toric_manifold_properties}, the matrix $A_K = ( C_1 |
C_2) \in M_{2,2}(\mathbb{Z})$  formed from the columns $C_1$ and $C_2$
has determinant $\pm 1$. Therefore, by multiplying $M$ from the left by a matrix $A \in GL_{2}(\mathbb{Z})$, we can assume without loss of generality that the matrix $A_K$ is the identity, and that the Kähler cone is the first quadrant: $K = \left( \mathbb{R}_{>0} \right)^2$.
Then, the matrix of the moment map will have the form
	\[
M=
		\begin{pmatrix}
			1	&	0	&	a 	&	\cdots  	\\
			0	&	1	&	b 	&	\cdots
		\end{pmatrix},
	\]
where $a,b \in \mathbb{Z}$. Let us investigate the sign of the
integers $a$ and $b$. Since the Kähler cone $K$ is the first quadrant 
$\left(\mathbb{R}_{>0} \right)^2$, the case where $a,b > 0$ would
contradict $K=\left( \mathbb{R}_{>0} \right)^2$, so it is
impossible. Furthermore, the case where $a,b < 0$ is also impossible, as the resulting toric manifold would fail the compactness condition of Proposition \ref{prop:gamma_toric_manifold_properties}: the cone associated to the moment map would contain the line spanned by the vector $\binom{a}{b}$, its negative being in the Kähler cone $K$.
	Therefore, the integers $a$ and $b$ must be of opposite signs.

	If $a>0$ and $b<0$, then applying the smoothness condition of Proposition \ref{prop:gamma_toric_manifold_properties}, we obtain that the matrix $\begin{pmatrix} 0 & a \\ 1 & b \end{pmatrix}$ must have determinant $\pm 1$, therefore $a=1$.
	Using a similar argument, if $a < 0$ and $b > 0$ then $b=1$.
	Therefore, the columns of $M$ following the first two columns are either of the type $\begin{pmatrix} -k \\ 1 \end{pmatrix}$ or of the type $\begin{pmatrix} 1 \\ -k \end{pmatrix}$, with $k \in \mathbb{Z}_{\geq 0}$.

	Finally, we show that $M$ can not contain simultaneously columns of these two types.
	Let us assume that
	\[
		M
		=
		\begin{pmatrix}
			1	&	0	&	1 	&	-b	&	\cdots 	\\
			0	&	1	&	-a 	&	1	&	\cdots
		\end{pmatrix}
	\]
	Then, the smoothness condition gives that
	\[
		\textnormal{det}
		\begin{pmatrix}
			1	&	-b 	\\
			-a 	&	1
		\end{pmatrix}
		=
		1 - ab
		= \pm 1
	\]
	Therefore $a=1, b=2$ or $a=2, b=1$. In both cases, the cone associated to the moment map $M$ contains the line $\textnormal{Vect}_{\mathbb{R}}(\binom{1}{-1})$, thus failing the compactness condition.
	Consequently, the remaining columns of $M$ are always of the same type $\begin{pmatrix} -k \\ 1 \end{pmatrix}$ or $\begin{pmatrix} 1 \\ -k \end{pmatrix}$.
	Finally, multiplying $M$ from the left by the matrix
        $\begin{pmatrix} 0 & 1 \\ 1 & 0 \end{pmatrix}$ and permuting
        the 1st two columns of $M$ if necessary, we can arrange that
        the matrix $M$ contains only vectors of the type
        $\begin{pmatrix} -k \\ 1 \end{pmatrix}$. 
\end{proof}

Finally, we give a result for computing intersection products for such symplectic toric manifolds.
Our formula will rely on the computation of Jeffrey--Kirwan residues for toric manifolds done in Section 2 of \cite{Szenes_Vergne:jeffrey_kirwan_residues}.
We will also refer to the same article for more details on these residues.

\begin{theorem}[\cite{Szenes_Vergne:jeffrey_kirwan_residues}, Theorem 2.6]\label{thm:gamma_computation_of_intersection_products}
	Let $X_{\mu, K}$ be a symplectic toric manifold, whose toric data is given as in the statement of Proposition \ref{prop:gamma_model_toric_picard_rank_two}: its chamber is $K=\left( \mathbb{R}_{>0} \right)^2$ and the matrix of its moment map is given by
	\[
		\begin{pmatrix}
			1
			&
			\cdots
			&
			1
			&
			0
			&
			-a_1
			&
			\cdots
			&
			-a_k
			\\
			0
			&
			\cdots
			&
			0
			&
			1
			&
			1
			&
			\cdots
			&
			1
		\end{pmatrix}
	\]
	Denote by $p_1, p_2$ the toric divisors of $X_{\mu,K}$ (see Proposition \ref{prop:gamma_cohomology_k_theory_of_toric_manifolds}), and let $f(x,y) \in \mathbb{C}[x,y]$ be some polynomial.
	Then, the intersection product given by the polynomial $f$ can be computed as an iterated residue as follows:
	\[
		\int_{\left[X_{\mu,K} \right]}
		f(p_1,p_2)
		=
		\operatorname{Res}_{y=0} \operatorname{Res}_{x=0}
		f(x,y)\frac{dx dy}{x^N y (y - a_1 x) \cdots (y - a_k x)}
	\]
\end{theorem}

\begin{proof}
	Recall that $\alpha_i \in \mathbb{Z}^2$ denotes the vector given by the $i$-th column of the moment map.
	Starting from the right hand side of the identity we want to prove, we use Theorem 2.6 of \cite{Szenes_Vergne:jeffrey_kirwan_residues} for the projective sequence $\mathcal{A} = (\alpha_1, \dots, \alpha_{N+k+1})$ and a sum-regular vector $\xi \in K=(\mathbb{R}_>0)^2$ located below the line $\textnormal{Vect}(c_1(TX_{\mu,K}))$. We obtain
	\[
		\operatorname{Res}_{y=0} \operatorname{Res}_{x=0}
		f(x,y)\frac{dp_1 dp_2}{x^N y (y - a_1 x) \cdots (y - a_k x)}
		=
		\textnormal{JK}_{\left( \mathbb{R}_{>0} \right)^2}
		\left(
			\frac{f(x,y)}{x^N y (y - a_1 x) \cdots (y - a_k x)}
		\right),
	\]
	where $\textnormal{JK}_{\left( \mathbb{R}_{>0} \right)^2}$ denotes the Jeffrey--Kirwan residue, see e.g. Equation (2.1) of \cite{Szenes_Vergne:jeffrey_kirwan_residues}.
	Using Proposition 2.3 of \cite{Szenes_Vergne:jeffrey_kirwan_residues}, we get
	\[
		\textnormal{JK}_{\left( \mathbb{R}_{>0} \right)^2}
		\left(
			\frac{f(x,y)}{x^N y (y - a_1 x) \cdots (y - a_k x)}
		\right)
		=
		\int_{\left[X_{\mu,K} \right]}
		f(p_1,p_2)
	\]
\end{proof}

\subsection{Oscillatory integral and gamma class in quantum cohomology}\label{ssec:oscillatory_in_qh}

%
%
%

In this Subsection, we recall the definition of the oscillatory
integral and the $I$-function of a torci manifold. Then we would like
to recall the results of Iritani from \cite{Iritani:gamma_structure,
  Iritani:QDM_of_toric_varieties_and_oscillatory}, which will be our
guiding principle for what we would like to do in the $K$-theoretic
settings.  


\subsubsection{Oscillatory integral in quantum cohomology}

From now on, when considering a symplectic toric manifold $X_{\mu,K}$, we will always assume that it satisfies the three conditions of Proposition \ref{prop:gamma_toric_manifold_properties}, i.e. it is compact, smooth, and Fano.

\begin{definition}[Landau--Ginzburg potential]\label{def:gamma_cohomological_landau_ginzburg_model}
	Let $X_{\mu,K}$ be a symplectic toric manifold and denote by $m_{ij} \in \mathbb{Z}$ the coefficients of the matrix of its moment map $\mu$.
	The \textit{Landau--Ginzburg potential} associated to the
        toric manifold $X_{\mu,K}$ is the following family of
        holomorphic functions:
	\[
		\begin{tikzcd}
			Y := \left( \mathbb{C}^* \right)^n
				\arrow[r,"W"]
				\arrow[d,"\pi"]
			&
			\mathbb{C}
			\\
			B := \left( \mathbb{C}^* \right)^r
		\end{tikzcd} ,
	\]
	where we denote by $x_1, \dots, x_n$ the standard global
        coordinates on the complex torus $Y$, $Q_1, \dots, Q_r$ the
        standard global coordinates on the complex torus $B$, the maps
        $W$ and $\pi$ are given by 
	\begin{align*}
		W(x_1, \dots, x_n)
		&
		:=
		x_1 + \cdots + x_n
		\in \mathbb{C}
		\\
		\pi(x_1, \dots, x_n),
		&
		:=
		\left(
			\cdots,
				\prod_{j=1}^n x_i^{m_{ij}}
			,\cdots
		\right)
		\in B.
	\end{align*}
	We will refer to the relations (after identification) $Q_i = \prod_{j=1}^n x_i^{m_{ij}}$ as the \textit{Batyrev constraints}.
\end{definition}

\begin{definition}[Oscillatory integral]\label{def:gamma_cohomology_oscillatory_integral}
	Consider the Landau--Ginzburg potential associated to a toric manifold $X_{\mu, K}$.
	Fix $(Q_1, \dots, Q_r) \in B$. The formula 
\[
\omega_{\pi^{-1}(Q)}:=\frac{
d\log x_1\wedge \cdots\wedge d\log x_n }{
d\log Q_1\wedge\cdots \wedge d\log Q_r }
\] 
defines a holomorphic form on $\pi^{-1}(Q_1,\dots,Q_r)$. The
oscillatory integral $\mathcal{I}^{\textnormal{coh}}(z,Q)$ is the
function defined by 
	\[
		\mathcal{I}^{\textnormal{coh}}(z,Q)
		:=
		\int_{\Gamma }
		e^{-W_Q/z}
		\omega_{\pi^{-1}(Q)},
	\]
where $W_Q := W|_{\pi^{-1}(Q_1,\dots,Q_r)}$ and $\Gamma \subset \pi^{-1}(Q_1,\dots,Q_r)$ is a
semi-infinite cycle representing a homology class in 
\begin{align}
\label{Lef_thimbles}
\varprojlim_{l \in \mathbb{Z}}
H_{n-r}\left( \pi^{-1}(Q_1,\dots,Q_r),
\operatorname{Re}(W_Q/z)>l;\mathbb{Z} \right) .
\end{align}
\end{definition}
The triple $(Y,W,\omega_{\pi^{-1}(Q)})$ was proposed by Givental (see
\cite{Givental:toric_mirrors}) as a mirror model of the toric manifold
$X_{\mu,K}$. On the other hand, there is a quantum field theory model
known as the Landau--Ginzburg model, whose partition function is
closely related to Givental's oscillatory integral. We will refer to
the triple $(Y,W,\omega_{\pi^{-1}(Q)})$ as the {\em Givental's
  mirror} or the {\em Landau--Ginzburg model} of $X_{\mu,K}$.  
\begin{remark}\label{rmk:gamma_cohomology_real_cycle}
	We will be interested mostly in a specific integration cycle $\Gamma_\mathbb{R}$,
        which we will refer to as the \textit{real Lefschetz thimble}. 
	Taking $(Q_1, \dots, Q_r) \in (\mathbb{R}_{>0})^r$, the
        corresponding real Lefschetz thimble $\Gamma_\mathbb{R}$ is
        defined by  
	\[
		\Gamma_\mathbb{R}
		:=
		\{
			(x_1,\dots,x_n) \in \pi^{-1}(Q_1,\dots,Q_r)
			\, | \,
			\forall j, x_j \in \mathbb{R}_{> 0}
		\}.
	\]
Note that if $Q \in (\mathbb{R}_{>0})^r$ and $z>0$, then
$\Gamma_{\mathbb{R}}$ represents a homology class in
(\ref{Lef_thimbles}). Therefore, we can define the corresponding
oscillatory integral. We refer to Section 3.3.1 in
\cite{Iritani:gamma_structure}, for further details on the group
(\ref{Lef_thimbles}) of semi-infinite cycles.
\end{remark}

\subsubsection{Small $I$-function}

We introduce a second function, called Givental's small $I$-function, that is related by mirror symmetry to cohomological Gromov--Witten invariants.
While we immediately define it as a formal power series, this formula should be understood through fixed point localisation in cohomology (see \cite{Givental:equivariant_GW_invariants}) or through the theory of GKZ $\mathcal{D}$-modules (see Lemma 4.6 of \cite{Iritani:gamma_structure}).

\begin{definition}[Cohomological small $I$-function]\label{def:gamma_cohomology_I_function}
	Let $X_{\mu,K}$ be a symplectic toric manifold. Write
        $\alpha_1, \dots, \alpha_n \in \mathbb{R}^r$ for the image of the canonical basis of $\mathbb{Z}^n$ by the moment map $\mu$.
	Denote by $p_1, \dots, p_r \in H^2(X_{\mu,K};\mathbb{Z})$ the toric divisors of Proposition \ref{prop:gamma_cohomology_k_theory_of_toric_manifolds}.
	The \textit{cohomological small $I$-function} $I^{\textnormal{coh}}(z,Q)$ of the toric manifold $X_{\mu,K}$ is the cohomologically valued power series defined by
	\[
		I^{\textnormal{coh}}(z,Q)
		:=
		e^
		{
			-\sum_{i=1}^r p_i \log(Q_i)/z
		}
		\sum_{\substack{
		d \in H_2(X_{\mu,K};\mathbb{Z})	\\
		d \textnormal{ effective}
		}}
		Q^d
		\prod_{j=1}^n
		\frac{
			\prod_{r=-\infty}^0
			\alpha_j(p)-rz
		}{
			\prod_{r=-\infty}^{\alpha_j(d)}
			\alpha_j(p)-rz
		}
		\in
		H^*(X_{\mu,K};\mathbb{Q})
		\otimes
		\mathbb{C}[z^{\pm 1}][ \! [Q] \! ],
	\]
	where $\alpha_j(p) := \langle \alpha_j, p \rangle = m_{1j} p_1 + \cdots + m_{rj} p_r \in H^*(X_{\mu,K};\mathbb{Q})$ and $Q^d := Q_1^{d_1} \cdots Q_r^{d_r}$, $d_i := \int_d p_i$.
\end{definition}

\begin{proposition}[\cite{Iritani:gamma_structure}, Lemma 4.6]\label{prop:oscillatory_i_function_differential_equation}
	The oscillatory integral $\mathcal{I}^{\textnormal{coh}}$ of
        Definition \ref{def:gamma_cohomology_oscillatory_integral} and
        the $I$-function $I^{\textnormal{coh}}(z,Q)$ of Definition
        \ref{def:gamma_cohomology_I_function} satisfy the same system
        of differential equations 
	\[
		\Delta \mathcal{I}^{\textnormal{coh}}(z,Q) = \Delta I^{\textnormal{coh}}(z,Q) = 0,	
	\]
	where $\Delta$ is the differential operator associated to the system of $r$ differential equations
	\[
		\prod_{j : m_{ij} > 0} \prod_{r = 0}^{m_{ij} - 1}
		\Big(
			r
			- \alpha_j(Q\partial_Q)
		\Big) I(z,Q)
		=
		z^{-m_i}Q_i
		\prod_{j : m_{ij} < 0} \prod_{r = 0}^{- m_{ij} - 1}
		\Big(
			r
			- \alpha_j(Q\partial_Q)
		\Big) I(z,Q),
	\]
where 
$
\alpha_j(Q\partial_Q):= \sum_{a=1}^r m_{aj} Q_{a} \partial_{Q_a}
$ 
and $m_i=\sum_{j=1}^n m_{ij}$.
\end{proposition}

\subsubsection{Comparison theorem}

Since we have two solutions $\mathcal{I}^{\textnormal{coh}}$ and
$I^{\textnormal{coh}}$ of the same differential system, we would like
to be able to compare these two solutions. Following Iritani, let us
introduce a multiplicative characteristic class which plays an
important role in quantum cohomology. 

\begin{definition}[Cohomological Gamma class]\label{def:gamma_cohomological_gamma_class}
	Let $E \to X$ be a vector bundle, and denote by
        $\delta_1, \dots, \delta_n $ 
its Chern roots.
	The \textit{cohomological Gamma class} $\widehat{\Gamma}(E) \in H^*(X;\mathbb{Q})$ is defined by
	\[
		\widehat{\Gamma}(E)
		:=
		\prod_{j=1}^n \Gamma(1+\delta_j)
		\in H^*(X_{\mu,K};\mathbb{Q}),
	\]
	where $\Gamma(1+\delta_j)$ is defined by
        substituting $x=\delta_j$ in the Taylor series expansion of
        $\Gamma(1+x)$ at $x=0$.  
\end{definition}

\begin{theorem}[\cite{Iritani:gamma_structure}, Theorem 4.14, Equation (70)]\label{thm:gamma_cohomology_oscillatory_i_function}
	Let $\Gamma_\mathbb{R}$ be the real Lefchetz thimble of Remark \ref{rmk:gamma_cohomology_real_cycle}.
	Then, the oscillatory integral
        $\mathcal{I}^{\textnormal{coh}}$ and the $I$-function
        $I^{\textnormal{coh}}$ are related by the identity 
	\[
		\mathcal{I}^{\textnormal{coh}}(z,Q)
		=
		\int_{\left[X_{\mu,K}\right]}
\widehat{\Gamma}(X_{\mu,K})\, \cup \, 
z^\rho\, z^{\operatorname{deg}}\, 
		I^{\textnormal{coh}}(z,Q),
	\]
	where $\int_{\left[X_{\mu,K}\right]}$ denotes the intersection
        product by $[X_{\mu,K}] \in H_*(X_{\mu,K};\mathbb{C})$,
        $\widehat{\Gamma}(X_{\mu,K})$ is the Gamma class of the
        holomorphic tangent bundle $TX_{\mu,K}$, and $\rho$ is the
        operator of cup product multiplication by $c_1(TX_{\mu,k})$. 
\end{theorem}

%
%
%
%

We refer to the paper of Iritani \cite{Iritani:gamma_structure} for the proof of this statement in general.

\subsection{$q$-oscillatory integral in quantum $K$-theory}\label{ssection:q_gamma_structure_in_qk}

The goal of this subsection is to introduce a $K$-theoretic analogue of Theorem \ref{thm:gamma_cohomology_oscillatory_i_function} comparing the oscillatory integral with the $I$-function in cohomology.
In quantum $K$-theory, we will consider Givental's permutation
equivariant $I$-function defined in Theorem p.8 of
\cite{Givental:perm_toric_q_hypergeometric}, and a $q$-analogue of the
oscillatory integral (see Definition \ref{def:qosc_q_oscillatory_integral_real}).

\subsubsection{$q$-oscillatory integral in quantum $K$-theory}

A $K$-theoretic analogue of the Landau--Ginzburg potential
defined in Definition \ref{def:gamma_cohomological_landau_ginzburg_model}
was proposed by Givental in \cite{Givental:perm_mirror} (see also \cite{Iritani:qdifference_toric}).

\begin{definition}[$K$-theoretic mirror family; \cite{Givental:perm_mirror}, Theorem 2]\label{def:gamma_k_theoretic_mirror_family}
	Let $X_{\mu,K}$ be a symplectic toric manifold and write
        $m_{ij} \in \mathbb{Z}$ for the entries of the matrix
        $\textnormal{Mat}(\mu)$ of the moment map $\mu$, and suppose
        that the length of $q \in \mathbb{C}^*$ is not 1. 
	The $K$-\textit{theoretic mirror family}
        associated to the toric manifold $X_{\mu,K}$ is the following
        family of holomorphic functions:
	\[
		\begin{tikzcd}
			Y := \left( \mathbb{C}^* \right)^n
				\arrow[r,"W_q"]
				\arrow[d,"\pi"]
			&
			\mathbb{C}
			\\
			B := \left( \mathbb{C}^* \right)^r
		\end{tikzcd}
	\]
	where we denote by $x_1, \dots, x_n$ the standard coordinates
        on $Y$, $Q_1, \dots, Q_r$ the standard coordinates on $B$, and $W_q$ and $\pi$ are defined by
	\begin{align*}
		W_q(x_1,\dots,x_n)
		&:=
		\sum_{j=1}^n \sum_{l > 0}
		\frac{{x_j}^l}{l(1-q^l)}
		\in \mathbb{C},
		\\
		\pi(x_1, \dots, x_n)
		&:=
		\left(
			\cdots,
				\prod_{j=1}^n x_i^{m_{ij}}
			,\cdots
		\right)
		\in B.
	\end{align*}
	We will refer to the relations (after identification in $Y$) $Q_i =
        \prod_{j=1}^n x_i^{m_{ij}}$ as the \textit{Batyrev
          constraints}. 
\end{definition}
There are two ways to define an oscillatory integral that solves the
system of $K$-theoretic quantum difference equations of $X_{\mu,K}$. One
of them, as proposed by Givental in \cite{Givental:perm_mirror}, is by
using a Riemann (or Lebesgue) integral. In fact, we made
an attempt to achieve our goals with such a definition, but we got into
a problem which is somewhat tricky to resolve.
The second way is
to use an appropriate multi-dimensional version 
of the Jackson integral, that is, to define a $q$-analogue of the
oscillatory integral for quantum $K$-theory. This is the approach
which we take in this paper. Such $q$-integrals
appeared first
in Section 5 of \cite{Iritani:qdifference_toric}
and later for Grassmannians in Section 8 of
\cite{Givental_Yan:QK_of_grassmanian_localization}.
The $q$-oscillatory integrals should resemble the formula

\[
	\left[
	\int_{\Gamma }
	\right]_q
	\exp \left(
		W_{|\pi^{-1}(Q)}(x_1,\dots,x_n)
	\right)
	\omega_{\pi^{-1}(Q),q},
\]
where the symbol
$
	\left[
	\int_{\Gamma}
	\right]_q
$
means we should consider a sum for which the inputs $x_j$ take values
in some lattice in the semi-infinite cycle $\Gamma\subset \pi^{-1}(Q)$
stable by multiplication by $q$. In this paper, we will focus only on
the q-oscillatory integral corresponding to the real cycle
$\Gamma_{\mathbb{R}}$ (see Remark
\ref{rmk:gamma_cohomology_real_cycle}). The case of an arbitrary
semi-infinite cycle, requires choosing a representative
in the homology class that admits an appropriate $q$-discrete
structure. Proving the existence of such a choice requires a
separate investigation, so we do not pursue it in this paper.

For the sake of simplicity, let us consider the case of Picard rank 2
symplectic toric manifolds. The case of an arbitrary Picard rank is
similar. To begin with, let us examine the analytic properties of the
function $W_q$.
	Assume that the length of $q \in \mathbb{C}^*$ is not 1.
	The power series (in the definition of $W_q$) $\sum_{l
          > 0} \frac{x^l_j}{l(1-q^l)}$, has finite convergence radius.
        More precisely, using the ratio test, we get that if $|q|<1$ (resp. $|q|>1$), then the series is
        convergent for $|x_j|<1$ (resp. $|x_j|<q^{-1}$). 
	The function defined by $f(x)=\exp \left( \sum_{l > 0} \frac{x^l}{l(1-q^l)} \right)$ has the following analytical continuations:
	\begin{equation}\label{eqn:gamma_oscillatory_integral_continuations}
		\exp \left( \sum_{l > 0} \frac{x^l}{l(1-q^l)} \right)
		=
		\left\{
			\begin{aligned}
			&
			\frac{1}{(x;q)_\infty} ,
			&&
			\textnormal{if } |q| < 1,
			\\
			&
			(q^{-1}x;q^{-1})_\infty ,
			&&
			\textnormal{if } |q| > 1,
			\end{aligned}
		\right. 
	\end{equation}
	where we denote by $(z;q)_\infty := \prod_{r \geq 0}(1 - q^r z)$ the $q$-Pochhammer symbol, defined for $|q| < 1$.
	To prove the first analytical continuation, we use the following Taylor series:
	\[
		\sum_{l > 0}
		\frac{x^l}{l(1-q^l)}
		=
		\sum_{l > 0} \sum_{k \geq 0}
		\frac{x^l}{l}q^{kl}
		=
		-\sum_{k \geq 0} \log(1-q^kx).
	\]
	To obtain the second one, we start from the same Taylor series and multiply both sides of the fraction by $q^{-l}$, and use the first analytical continuation.

Note that these two analytical continuations are closely related to
the $q$-exponential functions:
	\[
		e_q(x)
		:=
		\sum_{d=0}^\infty
		x^d
		\prod_{l=1}^d \frac{1-q}{1-q^l} = \frac{1}{((1-q) x;q)_\infty}
	\]
and 
\[
E_q(x):=\sum_{d=0}^\infty 
q^{d(d-1)/2} x^d
\prod_{l=1}^d \frac{1-q}{1-q^l} 
=((q-1) x;q)_\infty.
\]
Using the $q$-binomial theorem (see Equation (1.3.2) p.8 of
\cite{Gasper_Rahmann:q_hypergeometric_series}), one can show that if
$|q|<1$, then 
	\[
	  \frac{1}{(x;q)_\infty}
	  =
	  e_q\left(
	    \frac{x}{1-q}
	  \right)
	\]
and if $|q| > 1$, then 
	\begin{align*}
		\left(q^{-1} x ;q^{-1}\right)_\infty
		=E_{q^{-1}}\left(
	    \frac{x}{1-q}
	  \right). 
	\end{align*}
From now on we will assume that $|q|>1$.  Our
motivation for choosing $|q|>1$ comes from the formula for the
$K$-theoretic $I$-function of a Fano toric manifold, i.e., if $|q|>1$, then the $I$-function has an
infinite radius of convergence with respect to the Novikov
variables.  Furthermore, for simplicity, let us assume that $q$ is a
real number. 
\begin{definition}[Jackson integral; see e.g. Appendix A of \cite{Di_Vizio_Zhang:qsummation_and_confluence}]\label{def:qosc_jackson_integral}
	The Jackson integral is a $q$-analogue of the classical (Riemann) integral. We introduce the following formal definitions for a complex function:
	\begin{align*}
		\left[
			\int_0^\infty
		\right]_q
		f(x) d_qx
		&:=
		\sum_{d \in \mathbb{Z}}
		q^d
		f\left(q^d \right),
		\\
		\left[
			\int_0^{\infty/A}
		\right]_q
		f(x) d_qx
		&:=
		\sum_{d \in \mathbb{Z}}
		\frac{q^d}{A}
		f\left(
			\frac{q^d}{A}
		\right).
	\end{align*}
\end{definition}

Let us point out that sometimes the Jackson integral is defined to be
$(1-q)\left[\int_{0}^\infty\right]_q f(x) d_qx$  for $q<1$, so that in the limit
$q\to 1$ it coincides with the Riemann integral. For our purposes,
Definition \ref{def:qosc_jackson_integral} seems to be more convenient.

In order to write this $q$-oscillatory integral, we will fix
coordinates $x_1, \dots, x_{N+k+1}$ on the fibre $\pi^{-1}(Q_1,Q_2)$.
For fixed $Q_1, Q_2 \in \mathbb{C}^*$, the the fibre
$\pi^{-1}(Q_1,Q_2)$ is defined by the following equations: 
\begin{align*}
	Q_1
	&=
	x_1 \cdots x_N x_{N+1}^{-a_0} \cdots x_{N+k+1}^{-a_k}\ ,
	\\
	Q_2
	&=
	x_{N+1} \cdots x_{N+k+1}\ .
\end{align*}
We will consider an isomorphism $\varphi : \left( \mathbb{C}^*
\right)^{N+k-1} \simeq \pi^{-1}(Q)$ by assuming all coordinates
on the fibre except for $x_N$ and $x_{N+1}$ to be free, i.e., we define
\begin{equation}
\nonumber
	\begin{aligned}
	&\varphi(x_1, \dots, x_{N-1}, x_{N+2}, \dots x_{N+k+1})
	\\
	&:=
	\left(
		x_1, \dots, x_{N-1},
		Q_1 x_1^{-1} \cdots x_{N-1}^{-1} x_{N+2}^{a_1} \cdots x_{N+k+1}^{a_k},
		Q_2 x_{N+2}^{-1} \cdots x_{N+k+1}^{-1},
		x_{N+2}, \dots, x_{N+k+1}
	\right)
	\in \pi^{-1}(Q).
	\end{aligned}
\end{equation}
The isomorphism $\varphi$ gives a parametrization of the real cycle
$\Gamma_{\mathbb{R}}$, that is, we have an isomorphism
$\left(\mathbb{R}_{>0}\right)^{N+k-1}\cong \Gamma_{\mathbb{R}}$. Moreover, the
holomorphic volume form takes the form 
\[
\omega_{\pi^{-1}(Q)} = 
\frac{dx_1}{x_1}\wedge\cdots\wedge 
\frac{dx_{N-1}}{x_{N-1}}\wedge
\frac{dx_{N+2}}{x_{N+2}}\wedge\cdots\wedge
\frac{dx_{N+1+k}}{x_{N+1+k}}.
\]
Using the parametrization $\varphi$, it is natural to define the $q$-oscillatory
integral along the real Lefschetz thimble $\Gamma_{\mathbb{R}}$ by
using the multivariable version of Definition \ref{def:qosc_jackson_integral}.
\begin{definition}[$q$-oscillatory integral]
\label{def:qosc_q_oscillatory_integral_real} 
Suppose that $Q_1,Q_2\in q^\ZZ:=\{q^n\ |\ n\in \mathbb{Z} \}$. The $q$-oscillatory integral along
the real Lefchetz thimble is the function defined by the following 
Jackson integral along the coordinates 
$x_1, \dots, x_{N-1}, x_{N+2}, \dots, x_{N+k+1}$:
\begin{align*}
\mathcal{I}^{K-\textnormal{th}}(q,Q_1,Q_2)
		&:=
		\left[
		\int_{\Gamma_\mathbb{R} }
		\right]_q
\prod_{j=1}^n \left. E_{q^{-1}}\left( \frac{x_j}{1-q}\right)\right|_{\pi^{-1}(Q)}
		\omega_{\pi^{-1}(Q),q}
		\\
		&:=
		\sum_{
			\substack{d_1, \dots, d_{N-1} \in \mathbb{Z} \\ d_{N+2}, \dots, d_{N+k+1} \in \mathbb{Z}}
		}
		\prod_{j=1}^{n} E_{q^{-1}}\left( \frac{x_j(\underline{d},Q_1,Q_2)}{1-q}\right),
\end{align*}
where
\begin{align*}
		&\textnormal{for } j \neq N,N+1,
		\quad
		x_j(\underline{d},Q_1,Q_2)
		:=
		q^{d_j}\ ,
		\\
		&x_N(\underline{d},Q_1,Q_2)
		:=
		Q_1 q^{
			-\sum_{j'=1}^{N-1} d_{j'} + \sum_{j'=1}^k a_{j'} d_{N+1+j'}
		}\ ,
		\\
		&x_{N+1}(\underline{d},Q_1,Q_2)
		:=
		Q_2 q^{
			-\sum_{j'=1}^{k} d_{N+1+j'}
		}\ 
\end{align*}
and 
\[
\omega_{\pi^{-1}(Q),q} = 
\frac{d_qx_1}{x_1}\wedge\cdots\wedge 
\frac{d_qx_{N-1}}{x_{N-1}}\wedge
\frac{d_qx_{N+1}}{x_{N+1}}\wedge\cdots\wedge
\frac{d_qx_{N+1+k}}{x_{N+1+k}}.
\]
\end{definition}

\subsubsection{Non-discrete Novikov variables}
Definition \ref{def:qosc_q_oscillatory_integral_real} does not work
for $Q_1,Q_2\notin q^\ZZ$, because the corresponding Jackson integral
is divergent. In this section we would like to outline a construction
which should allow us to define the $q$-oscillatory integral for
arbitrary values of $Q_1$ and $Q_2$. Our definition relies on a
conjecture about the regularity of a certain system of $q$-difference
equations. For toric manifolds of Picard rank $1$, the regularity
is known, but for Picard rank $>1$ our conjecture seems to be a
separate project. The results of this subsection would not be used in
what follows. The reader not interested in our speculations could
skip it. 

We follow the construction of De Sole and Kac from
\cite{DeSole_Kac:integral_representation_of_qgamma}. Let us recall the
{\em Jacobi theta function} 
\ben
\theta(x):=\theta_{q^{-1}}(x)=\sum_{n\in \ZZ} q^{-n(n-1)/2} x^n.
\een
It is a holomorphic function in $x\in \CC^*$ with essential
singularities at $x=0$ and $\infty$. Moreover, we have the so-called
Jacobi triple product identity
\ben
\theta(x)= 
\prod_{j=0}^\infty (1-q^{-j-1}) \, 
\prod_{j=0}^\infty (1+q^{-j} x)\,
\prod_{j=0}^\infty (1+q^{-j-1} x^{-1}),
\een
which shows that $\theta$ has simple zeroes at the points $x=-q^m$
($m\in \ZZ$). Finally, let us recall also the following property: for all $m \in \mathbb{Z}$, we have
\beq\label{theta-inv}
\theta(q^m x) = q^{m(m+1)/2} x^m \, \theta(x),
\eeq
which follows from the definition of $\theta(x)$. 
Now, suppose
that $A>0$ is a real number. We introduce the function 
\beq\label{q-constant:k}
k(A,t):=\frac{\theta(A^{-1})}{\theta(q^{-t} A^{-1})}.
\eeq
Using \eqref{theta-inv}, we get $k(A,m) = A^{-m} q^{-m(m-1)/2}$ for
all $m\in \ZZ$. We will make use of
the infinite-order difference operator $k(A,y\partial_y)$ acting on
the space of formal power series $\CC[\![y]\!]$ via 
\ben
k(A,y\partial_y) 
\sum_{m=0}^\infty c_m y^m = 
\sum_{m=0}^\infty c_m k(A,m) y^m = 
\sum_{m=0}^\infty c_m q^{-m(m-1)/2} (y A^{-1})^m. 
\een
Using the ratio test, we get that if the series
$\sum_{m=0}^\infty c_m y^m$ has a non-zero radius of convergence, then
the operator $k(A,y\partial_y)$ will produce an entire
function. Finally, comparing the Taylor series expansions of the two
exponential functions $e_{q^{-1}}(x)$ and $E_{q^{-1}}(x)$ at $x=0$, we
get 
\ben
k(A,y\partial_y) \, e_{q^{-1}}\left( \frac{y A}{1-q}\right) = 
E_{q^{-1}}\left( \frac{y }{1-q}\right).
\een
Motivated by the above formula, we would like to modify the definition
of the $q$-oscillatory integral as follows. Suppose that
$A_1,\dots,A_n$ and $Q_1, Q_2$ are positive real numbers. Put  
\ben
k(A,Q\partial_Q):= \prod_{j=1}^n 
k(A_j,\alpha_j(Q\partial_Q)),\quad 
\alpha_j(Q\partial_Q):=m_{1j} Q_1\partial_{Q_1} +m_{2j} Q_2\partial_{Q_2},
\een
where $A=(A_1,\dots,A_n)$. We would like to define
\begin{align}
\label{qoi:non-discr}
\mathcal{I}^{K-\textnormal{th}}(q,Q_1,Q_2)&:=
k(A,Q\partial_Q)\  \left[
\int_{\Gamma_\mathbb{R} }
\right]_q
e_{q^{-1}}\left( \frac{x_1 A_1}{1-q} \right) \cdots\, 
e_{q^{-1}}\left( \frac{x_n A_n}{1-q} \right) 
\omega_{\pi^{-1}(Q),q}\\
\nonumber
&:= k(A,Q\partial_Q)\ 
\sum_{
\substack{
d_1, \dots, d_{N-1} \in \mathbb{Z} \\ 
d_{N+2}, \dots, d_{N+k+1} \in \mathbb{Z} } }
\prod_{j=1}^{N+k+1}
e_{q^{-1}}\left(
\frac{A_j x_j(\underline{d},Q_1,Q_2)}{1-q}
		\right),
\end{align}
where the action of the operator $k(A,Q\partial_Q)$ on the oscillatory
integral will be defined next. Let us first focus on the analytic
properties of the oscillatory integral
\begin{align}
\label{qoi:eq}
\widetilde{\mathcal{I}}^{K-\textnormal{th}}(q,Q_1,Q_2):=
\left[
\int_{\Gamma_\mathbb{R} }
\right]_q
e_{q^{-1}}\left( \frac{x_1 A_1}{1-q} \right) \cdots\, 
e_{q^{-1}}\left( \frac{x_n A_n}{1-q} \right) 
\omega_{\pi^{-1}(Q),q}.
\end{align}
Using the ratio test, it is straightforward to check that the Jackson
integral is convergent, that is,
$\widetilde{\mathcal{I}}^{K-\textnormal{th}}(q,Q_1,Q_2)$ is analytic
for $|q|>1$ and for all $Q_1,Q_2\in \CC$, such that,
$\operatorname{Re}(Q_1)>0$, 
$\operatorname{Re}(Q_2)>0$. On the other hand, in order to define the
action of the operator $k(A,Q\partial_Q)$ we need to expand
$\widetilde{\mathcal{I}}^{K-\textnormal{th}}(q,Q_1,Q_2)$ in a
neighborhood of $(Q_1,Q_2)=0$. This is exactly the place where we need to
make a conjecture about the structure of such an expansion. 
\begin{proposition}\label{prop:q-gkz}
The oscillatory integral
$\widetilde{\mathcal{I}}^{K-\textnormal{th}}(q,Q_1,Q_2)$ is a solution
to the following system of $q$-difference equations
\begin{equation}\label{eqn:gamma_system_modified_oscillatory_integral}
\left[
\prod_{j : m_{ij} > 0} \prod_{r=0}^{m_{ij}-1}
\left( q^{ -\alpha_j(Q\partial_Q)+r} -1 \right)-
a_i\,q^{-m_i}\, Q_i\, 
\prod_{j : m_{ij} < 0} \prod_{r=0}^{-m_{ij}-1}
\left(q^{-\alpha_j(Q\partial_Q)+r}-1 \right)
\right]
f_q(Q)=0,
\end{equation}
where $\alpha_j(Q\partial_Q)=m_{1j}
Q_1\partial_{Q_1}+m_{2j}Q_2\partial_{Q_2}$, 
$a_i:=\prod_{j=1}^n A_j^{m_{ij}}$, 
and $m_i=\sum_{j=1}^n m_{ij}$.
\end{proposition}
The proof of the above proposition is similar to the proof of
Proposition \ref{prop:qosc_mirrors_have_same_qde} below, so we omit
it. To simplify the notation in our discussion let us explain how to
define the action of $k(A,Q\partial_Q)$ on \eqref{qoi:eq} when all
constants $A_i=1$. We would like to conjecture that the above system of
$q$-difference equations \eqref{eqn:gamma_system_modified_oscillatory_integral} has a basis of solutions of the following
form
\beq\label{f-reg}
f(Q) =\sum_{a,b=0}^{L}  f_{a,b}(Q_1,Q_2) \ell(Q_1)^a \ell(Q_2)^b,
\eeq
where $L\in \ZZ_{\geq 0}$ and the coefficients $f_{a,b}(Q_1,Q_2)$ are
analytic functions at $(Q_1,Q_2)=0$, where 
\begin{align}
\label{q-log}
\ell(x) = \frac{1}{2}-x\partial_x \log \theta(x) = \frac{1}{2} - \frac{x\theta'(x)}{\theta(x)}
\end{align}
is the so-called $q$-logarithm. Here the constant
$\tfrac{1}{2}=\tfrac{\theta'(1)}{\theta(1)}$ is chosen so that
$\ell(1)=0$.
Let us mention some evidence for this conjecture:
in the case of a toric
manifold of Picard rank 1, the system of $q$-difference equations in
Proposition \ref{prop:q-gkz} has a non resonant regular singularity at $Q=0$ and
the existence of a basis of solutions of the form \eqref{f-reg} is
known (see \cite{Hardouin_Sauloy_Singer:q_difference_equation_book}, Theorem 3.1.7 p. 127).
Furthermore, in the Picard rank 2 case, we are able to prove the conjecture in the following cases:

\begin{proposition}
	We assume that the manifold $X_{\mu,K}$ satisfies the three following conditions:
	that all coefficients $m_{ij}$ of the moment map are odd;
	that the numbers $m_i:= \sum_j m_{ij}$ are all equal to some positive integer $M$ (i.e. $c_1(TX_{\mu,K}) = Mp_1 + Mp_2$);
	and that the coefficients $A_i$ are chosen so that the numbers $a_i:=\prod_{j=1}^n A_j^{m_{ij}}$ are all equal to some number $a$.
	Then, the $q$-difference system \eqref{eqn:gamma_system_modified_oscillatory_integral} has a basis of solutions of the form \eqref{f-reg}.
\end{proposition}

\begin{proof}
	Under these assumptions, the $q$-difference system we have to solve given by
	\[
		\left[
		\prod_{j : m_{ij} > 0} \prod_{r=0}^{m_{ij}-1}
		\left( 1- q^{ -\alpha_j(Q\partial_Q)+r} \right)-
		a\,q^{-M}\, Q_i\, 
		\prod_{j : m_{ij} < 0} \prod_{r=0}^{-m_{ij}-1}
		\left( 1- q^{-\alpha_j(Q\partial_Q)+r} \right)
		\right]
		f_q(Q)=0,
	\]
	Using Theorem p.8 of \cite{Givental:perm_toric_q_hypergeometric}, one obtains a $K$-theoretically-valued solution given by
	\[
		I_{X_{\mu,K}}^{K-\textnormal{th}} \left( q^{-1}, a \, q^{-M} \, Q \right),
	\]
	where $I_{X_{\mu,K}}^{K-\textnormal{th}}$ the small $K$-theoretic $I$-function
	\[
		I_{X_{\mu,K}}^{K-\textnormal{th}}(q,Q)
		=
		P^{-\ell(Q)}
		\sum_{d=(d_1,d_2)}
		Q^d
		\prod_{j=1}^n
		\frac{
			\prod_{r=-\infty}^0
			\left(
				1 - U_j(P)q^r
			\right)
		}
		{
			\prod_{r=-\infty}^{\alpha_j(d)}
			\left(
				1 - U_j(P)q^r
			\right)
		}
		\in
		K^0\left( X_{\mu,K} \right)(q)[\![Q]\!],
	\]
	where $\ell(Q)$ is the $q$-logarithm \eqref{q-log},
    $P^{-\ell(Q)} := \prod_{i=1}^2 P_i^{-\ell(Q_i)} \in 
    K^0\left( X_{\mu,K} \right)$, and $P_i^{-\ell(Q_i)}$ should be 
    understood as the expansion of the binomial
    $\left(1-(1-P_i^{-1}) \right)^{\ell(Q_i)}$.
    The function $\ell(a \, q^{-M} \, Q)$ is another $q$-logarithm, therefore there exists a $q$-constant function $f$ such that $\ell(a \, q^{-M} \, Q) = f(q) \ell(Q)$.
    Let us consider the decomposition of the solution $I_X\left( q^{-1}, a \, q^{-M} \, Q \right)$ in the basis of $K^0\left( X_{\mu,K} \right)$.
    In front of a vector of the form $(1-P_1^{-1})^\alpha (1-P_2^{-1})^\beta$, we will find a function of the form (up to a $q$-constant)
    \[
    	\sum_{u=0}^\alpha \sum_{v=0}^\beta
    	f_{u,v}(Q_1,Q_2) \ell(Q_1)^u \ell(Q_2)^v,
    \]
    with $f_{u,v}$ analytic at $(Q_1,Q_2)=0$.
\end{proof}

Our claim is that for any $B_j\in \mathbb{R}_{>0}$ ($1\leq j\leq n$), there
is a natural way to define the action $k(B_j,D_j) f(Q)$, where $D_j=m_{1j}
Q_1\partial_{Q_1}+m_{2j}Q_2\partial_{Q_2}$ and $f$ has the form
\eqref{f-reg}. Indeed, note that the difference operator $q^{-D_j} B_j^{-1}$
commutes with the differential operators 
$\ell(Q_i)-m_{ij} B_j\partial_{B_j}$ ($i=1,2$). Indeed, we have 
\ben
q^{-D_j} B_j^{-1} \, \left(\ell(Q_i)-m_{ij} B_j\partial_{B_j}\right)=
\ell(q^{-m_{ij}} Q_i) q^{-D_j} B_j^{-1} - m_{ij}(B_j\partial_{B_j} +1)
q^{-D_j}B_j^{-1}. 
\een 
Using that $\ell(q^{-m} x)=\ell(x)+m$ for all $m\in
\ZZ$, we get that the above expression coincides with 
$\left(\ell(Q_i)-m_{ij} B_j\partial_{B_j}\right)\, q^{-D_j} B_j^{-1} .$ 
Let us write $f$ in the form 
\ben
f(Q)=\sum_{a,b} f_{a,b}(Q_1,Q_2) \, 
\left( \ell(Q_1)-m_{1j} B_j\partial_{B_j}\right)^ a
\left( \ell(Q_2)-m_{2j} B_j\partial_{B_j}\right)^ b\, \cdot \, 1,
\een
that is, a differential operator in $B_j$ acting on 1. Since 
$k(B_j,D_j)=\theta(B_j^{-1})\, \theta(q^{-D_j}B_j^{-1})^{-1}$, using the
commutativity of $q^{-D_j} B_j^{-1}$ and $\ell(Q_i)-m_{ij} B_j$, we get
\ben
k(B_j,D_j) f(Q)=\sum_{a,b} 
\theta(B_j^{-1})
\left( \ell(Q_1)-m_{1j} B_j\partial_{B_j}\right)^ a
\left( \ell(Q_2)-m_{2j} B_j\partial_{B_j}\right)^ b\, \cdot \, 
\theta(q^{-D_j} B_j^{-1})^{-1} f_{a,b}(Q_1,Q_2).
\een 
On the other hand,
\ben
\theta(q^{-D_j} B_j^{-1}) ^{-1} Q_1^{d_1} Q_2^{d_2} = 
\theta(q^{-\alpha_j(d)} B_j^{-1})^{-1} Q_1^{d_1} Q_2^{d_2} = 
\theta(B_j^{-1})^{-1}\, 
q^{-\alpha_j(d) (\alpha_j(d)+1)/2} \, 
B_j^{-\alpha_j(d)} Q_1^{d_1} Q_2^{d_2} ,
\een
where $\alpha_j(d):=d_1m_{1j}+d_2 m_{2j}$.
Finally,
\ben
\theta(B_j^{-1}) \, B_j\partial_{B_j} \, \theta(B_j^{-1})^{-1} = B_j\partial_{B_j} -\ell(B_j^{-1})+\frac{1}{2}.
\een
Our definition takes the form 
\begin{align*}
&
k(B_j,D_j) f(Q) :=  \sum_{a,b=0}^L \sum_{d_1,d_2=0}^\infty 
f_{a,b,d_1,d_2}\, 
q^{-\alpha_j(d) (\alpha_j(d)+1)/2} \, 
Q_1^{d_1} Q_2^{d_2} \\
&
\left( \ell(Q_1)-m_{1j} (B_j\partial_{B_j}-\ell(B_j^{-1})+1/2 )\right)^ a
\left( \ell(Q_2)-m_{2j}( B_j\partial_{B_j}-\ell(B_j^{-1})+1/2) \right)^ b\, \cdot \, 
B_j^{-\alpha_j(d)} .
\end{align*}
The action of the composition $k(B,Q\partial_Q)=k(B_1,D_1)\cdots
k(B_n,D_n)$ is defined by 
\begin{align*}
&
k(B,Q\partial_Q) f(Q) :=  \sum_{l_1,l_2=0}^L \sum_{d_1,d_2=0}^\infty 
f_{l_1,l_2,d_1,d_2}\, 
q^{-\frac{1}{2}\sum_{j=1}^n \alpha_j(d) (\alpha_j(d)+1)} \, 
Q_1^{d_1} Q_2^{d_2} \times \\
&
\times
\prod_{i=1}^2\left( 
\ell(Q_i)-
\sum_{j=1}^n m_{ij} \Big(B_j\partial_{B_j}-\ell(B_j^{-1})+1/2 \Big)\right)^ {l_i}
\, \cdot \, \prod_{j=1}^n B_j^{-\alpha_j(d)} ,
\end{align*}
where $\cdot$ on the 2nd line denotes the action of a
differential operator on a function.
Note that under the Fano condition the number $\sum_{j=1}^n
\alpha_j(d)\to \infty$ as $d_1\to \infty$ or $d_2\to \infty$. It
follows that $k(B,Q\partial_Q)f(Q)$ has the form \eqref{f-reg} and
that the
coefficient in front of $\ell(Q_1)^a\ell(Q_2)^b$ is a convergent
power series in $Q_1$ and $Q_2$ whose radius of convergence is
$\infty$, that is, the coefficients are holomorphic for all
$(Q_1,Q_2)\in \CC^2$. Furthermore, if we expand 
$k(B,Q\partial_Q)f(Q)$ as a Laurent series in $q^{-1}$, then the
coefficients will be polynomials in $Q_i$ ($i=1,2$) and $\ell(Q_i)$
($i=1,2$).

\subsubsection{Small $I$-function in quantum $K$-theory}

\begin{definition}[$K$-theoretic $I$-function; \cite{Givental:perm_toric_q_hypergeometric}, Theorem p.8]\label{def:gamma_k_th_I_function}
	Let $X_{\mu,K}$ be a symplectic toric manifold, and denote by
        $m_{ij} \in \mathbb{Z}$ 
($1\leq i\leq r$, $1\leq j \leq n$)
the entries of the matrix $\operatorname{Mat}(\mu)$, and by $P_i$ the
ring generators of $K^0\left( X_{\mu,K} \right)$ 
as in Proposition
\ref{prop:gamma_cohomology_k_theory_of_toric_manifolds}. 
	For a multi-index $d=(d_1,\dots,d_r)$, we also use the notations $\alpha_j(d):=m_{1j}d_1 + \cdots + m_{rj}d_r$, $Q^d := Q_1^{d_1} \cdots Q_r^{d_r}$ and $U_j(P)=\prod_{i=1}^r P_i^{m_{ij}}$.
	The \textit{small $K$-theoretic $I$-function} $I_{X_{\mu,K}}^{K-\textnormal{th}}$ of the toric manifold $X_{\mu,K}$ is given by the $K$-theoretic formal series
	\[
		I_{X_{\mu,K}}^{K-\textnormal{th}}(q,Q)
		=
		P^{-\ell(Q)}
		\sum_{d=(d_1,d_2)}
		Q^d
		\prod_{j=1}^n
		\frac{
			\prod_{r=-\infty}^0
			\left(
				1 - U_j(P)q^r
			\right)
		}
		{
			\prod_{r=-\infty}^{\alpha_j(d)}
			\left(
				1 - U_j(P)q^r
			\right)
		}
		\in
		K^0\left( X_{\mu,K} \right)(q)[\![Q]\!],
	\]
	where $\ell(Q)$ is the $q$-logarithm \eqref{q-log},
        $P^{-\ell(Q)} := \prod_{i=1}^2 P_i^{-\ell(Q_i)} \in 
        K^0\left( X_{\mu,K} \right)$, and $P_i^{-\ell(Q_i)}$ should be 
        understood as the expansion of the binomial
        $\left(1-(1-P_i^{-1}) \right)^{\ell(Q_i)}$. 
\end{definition}
Using the ratio test we get that the $K$-theoretic $I$-function has the same analytic
properties as the $q$-oscillatory integral, that is, if $X_{\mu,K}$ is
a Fano toric manifold and $|q|>1$, then $I_{X_{\mu,K}}^{K-\textnormal{th}}(q,Q)$
can be expanded into a convergent power series in $q^{-1}$, whose
coefficients are polynomials in $Q_i$ ($1\leq i\leq 2$) and $\ell(Q_i)$ ($1\leq i\leq 2$).

\begin{proposition}[see also \cite{Givental_Yan:QK_of_grassmanian_localization} p.21; \cite{IMT}, Proposition 2.12]\label{prop:qosc_mirrors_have_same_qde}
	The $K$-theoretic oscillatory integral
        $\mathcal{I}^{K-\textnormal{th}}$ of Definition
        \ref{def:qosc_q_oscillatory_integral_real} and the small
        $I$-function $I_{X_{\mu,K}}^{K-\textnormal{th}}$ satisfy the
        same set of $q$-difference equations below (indexed by $i \in
        \{1, \dots, r\}$): 
	\begin{align*}
		\left[
			\prod_{j : m_{ij} > 0} \prod_{r=0}^{m_{ij}-1}
			\left(
				1 - q^{-r+\alpha_j(Q\partial_Q)}
			\right)
			-
			Q_i
			\prod_{j : m_{ij} < 0} \prod_{r=0}^{-m_{ij}-1}
			\left(
				1 - q^{-r+\alpha_j(Q\partial_Q)}
			\right)
		\right]
		f_q(Q)
		=
		0,
	\end{align*}
where $\alpha_j(Q\partial_Q)=m_{1j} Q_1\partial_{Q_1}+m_{2j} Q_2\partial_{Q_2}$.
\end{proposition}
\begin{proof}
Denote by $(m_{ij})$ the coefficients of the matrix of the moment map $\mu$.
Note that for all $j \in \{1, \dots, n=N+1+k\}$ the following identity holds:
\begin{align*}
q^{\alpha_j(Q\partial_Q)}
		\mathcal{I}^{K-\textnormal{th}}(q,Q_1,Q_2)  = 
		\left[
		\int_{\Gamma_\mathbb{R} }
		\right]_q
		q^{x_j \partial_{x_j}}\,
E_{q^{-1}}\left( \frac{ x_1 }{1-q} \right)\cdots 
E_{q^{-1}}\left( \frac{ x_n }{1-q} \right)\,
\omega_{\pi^{-1}(Q),q}.
\end{align*}
Indeed, the action of the difference operator $q^{x_j\partial_{x_j}}$ inside the
integrand amounts to rescaling $x_j\mapsto qx_j$. On the other hand,
the integration is by definition an infinite sum over all $x_a\in
q^{\mathbb{Z}}$ ($a\neq N,N+1$) satisfying the relations
$x_1^{m_{i1}}\cdots x_n^{m_{in}}=Q_i$ ($i=1,2$). Changing the integration
variables via $y_a= x_a$ for $a\neq j$ and $y_j= q x_j$, we get that
the sum defining the Jackson integral is a sum over all $y_a\in q^\ZZ$
($a\neq N,N+1$), satisfying the relations $y_1^{m_{i1}}\cdots
y_n^{m_{in}}= q^{m_{ij}}Q_i=q^{\alpha_j(Q\partial_Q)}(Q_i)$, while the integrand takes the form 
$E_{q^{-1}}(y_1/(1-q))\cdots E_{q^{-1}}(y_n/(1-q))$. Clearly the
resulting Jackson integral coincides with the LHS of the identity that
we wanted to prove. 

Next, we use that $q^{x \partial_x} E_{q^{-1}}(x/(1-q))=(1-x)
E_{q^{-1}}( x/(1-q))$ to obtain that for all $r \geq 0$,
\[
\prod_{r=0}^{m-1}(1-q^{-r}q^{x \partial_x}) \cdot 
E_{q^{-1}}\left(\frac{x}{1-q}\right)=
x^{m} E_{q^{-1}}\left(\frac{x }{1-q}\right).
\]
Combining these two results we get
\begin{align}
\label{k-qde}
		&\prod_{j : m_{ij} > 0} \prod_{r=0}^{m_{ij}-1}
		\left(
			1 - q^{-r} q^{\alpha_j(Q\partial_Q)}
		\right)
		\mathcal{I}^{K-\textnormal{th}}(q,Q_1,Q_2)=
		\\
\nonumber
		&
		\left[
		\int_{\Gamma_\mathbb{R} }
		\right]_q
		\left(
			\prod_{j' : m_{ij'} \geq 0} x_{j'}^{m_{ij'}}
       		E_{q^{-1}}\left(\frac{x_{j'} }{1-q} \right) \right)
\left(
			\prod_{j'' : m_{ij''} < 0} 
		E_{q^{-1}}\left(\frac{x_{j''} }{1-q} \right) \right)
		\omega_{\pi^{-1}(Q),q}.
\end{align}
Using the Batyrev relation we replace $\prod_{j' : m_{ij'} > 0}
x_{j'}^{m_{ij'}} $ with $ Q_i\prod_{j'' : m_{ij''} < 0} x_{j''}^{-m_{ij''}}
$. The RHS of \eqref{k-qde} transforms into 
\begin{align}
\label{k-qde2}
Q_i \left[
		\int_{\Gamma_\mathbb{R} }
		\right]_q
		\left(
		\prod_{j' : m_{ij'} \geq 0} 	
		E_{q^{-1}}\left(\frac{x_{j'} }{1-q} \right) \right) 
\left(
\prod_{j'' : m_{ij''} < 0} 
x_{j''}^{-m_{ij''}}
                 E_{q^{-1}}\left(\frac{x_{j''}}{1-q} \right) \right)
		\omega_{\pi^{-1}(Q),q}.
\end{align}
On the other hand, note that 
\begin{align}
\label{k-qde3}	
&
\prod_{j : m_{ij} < 0} \prod_{r=0}^{-m_{ij}-1}
		\left(
			1 - q^{-r} q^{\alpha_j(Q\partial_Q)}
		\right)
		\mathcal{I}^{K-\textnormal{th}}(q,Q_1,Q_2)
=		\left[
		\int_{\Gamma_\mathbb{R} }
		\right]_q
\left(
\prod_{j':m_{ij'}\geq 0}
E_{q^{-1}}\left( \frac{x_{j'}}{1-q}\right)\right)\times	\\
\nonumber
&
\times \left(
\prod_{j'' : m_{ij''} < 0} 
x_{j''}^{-m_{ij''}} 
E_{q^{-1}}\left( \frac{x_{j''}}{1-q}\right)\right)
		\omega_{\pi^{-1}(Q),q}.
\end{align}
We get that up to a factor of $Q_i$ the RHS of formula \eqref{k-qde3}
coincides with \eqref{k-qde2}. This completes the proof
of the fact that the oscillatory integral is a solution to the
$q$-difference system. For the $I$-function
$I_{X_{\mu,K}}^{K-\textnormal{th}}$, we refer to Theorem p.8 of
\cite{Givental:perm_toric_q_hypergeometric}. 
\end{proof}

\begin{remark}
If our conjecture that the integral \eqref{qoi:eq} has an expansion
of the form \eqref{f-reg} is true, then the $q$-oscillatory integral
\eqref{qoi:non-discr} makes sense and it has the following properties:
\begin{itemize}
\item[(a)]
The integral
\eqref{qoi:non-discr} is a solution to the system of $q$-difference equations in Proposition 
\ref{prop:qosc_mirrors_have_same_qde} for any choice of the positive
real numbers $A_1,\dots,A_n$. 
\item[(b)] The integral
\eqref{qoi:non-discr} is a $q$-constant with respect to $A_i$ for all
$1\leq i\leq n$. 
\item[(c)] If $Q_1,Q_2\in q^\ZZ$, then the two definitions of the
  $q$-oscillatory integral, that is, Definition 
\ref{def:qosc_q_oscillatory_integral_real} and \eqref{qoi:non-discr}
agree. 
\end{itemize}
\end{remark}
\begin{remark}
After rescaling the variables $Q_i$ by $(1-q)^{\textnormal{deg}(Q_i)}Q_i$, one can notice that the $q$-difference equation satisfied by our functions in the proposition above has a formal limit when $q \to 1$, using the formal limit
	\[
		\lim_{q \to 1}
		\frac{1-q^{Q_i \partial_{Q_i}}}{1-q}
		=
		Q_i \partial_{Q_i}.
	\]
	Moreover, this formal limit corresponds to the differential equation of Proposition \ref{prop:oscillatory_i_function_differential_equation}, evaluated at $z=1$.
	Confluence of the $q$-oscillatory integral will be investigated in the Subsection \ref{sec:confl_comp_thm}.
\end{remark}

\subsection{Comparison theorem}\label{sec:comp_thm}

Let us recall the $q$-gamma function, defined by
\[
    \Gamma_{q^{-1}}(t)
    :=
    (1-q^{-1})^{1-t}
    \frac{(q^{-1};q^{-1})_\infty}{(q^{-t};q^{-1})_\infty},
\]
where $q>1$. This function satisfies $\lim_{q \to 1}
\Gamma_{q^{-1}}(t) = \Gamma(t)$, see e.g. Equation (1.10.3) and its
proof p.21 in \cite{Gasper_Rahmann:q_hypergeometric_series}. We will
make use of a multiplicative characteristic class defined via the
following modification of the $q$-gamma function:
\[
\gamma_{q}(t):=
(1-q^{-1})^{t-1} \Gamma_{q^{-1}}(t)=
\frac{\left(q^{-1} ;q^{-1}\right)_\infty}{\left(q^{-t} ;q^{-1}\right)_\infty}.
\]
\begin{definition}[$q$-gamma class]\label{def:qosc_q_gamma_class}
	Suppose that $q > 1$ is a real number and that $E \to
        X_{\mu,K}$ is a vector bundle.
	The \textit{$q$-gamma class} $\widehat{\gamma}_q(E) \in H^*(X_{\mu,K};\mathbb{Q})$ is defined by
	\[
		\widehat{\gamma}_q (E)
		:=
		\prod_{j=1}^m \delta_j \gamma_q(\delta_j)
		\in H^*(X_{\mu,K};\mathbb{Q}),
	\]
where $\delta_1, \dots, \delta_m $ are the Chern roots of $E$.
\end{definition}
Let us compute the $q$-gamma class of the tangent bundle of a toric
manifold $X_{\mu,K}$. It is well known that $TX_{\mu,k} = \sum_{j=1}^n
U_j(P^{-1}) -r \mathbf{1}$ in $K^0(X_{\mu,k})$, where $U_j(P^{-1}) =
P_1^{-m_{1j}}\cdots P_r^{-m_{rj}}$ and $\mathbf{1}$ is the trivial
rank-1 bundle. On the other hand, we have 
$\gamma_q(t+1)=(1-q^{-t})\gamma_q(t)$ and
\ben
\lim_{t\to 0} t \gamma_q(t) = \lim_{t\to 0} \frac{t}{1-q^{-t}} =
\frac{1}{\log q}. 
\een
In particular, if $\epsilon$ is the trivial bundle of rank $r$, then
$\gamma_q(\epsilon) = (\log q)^{-r}$ and the $q$-gamma class of a
toric manifold takes the form
\ben
\widehat{\gamma}_q(TX_{\mu,K}) = 
(\log q)^r
\prod_{j=1}^n \alpha_j(p) \gamma_q(\alpha_j(p)),
\een
where $\alpha_j(p):=\sum_{i=1}^r p_i m_{ij}$ is the same as in
Proposition \ref{prop:gamma_cohomology_k_theory_of_toric_manifolds}.
\begin{remark}
Note also that in Equation (2.13) of \cite{JMNT:BPS_modularity_perturbations_in_QK}, another $q$-gamma class is introduced for Grassmannians through a different motivation.
\end{remark}

\begin{definition}[$q$-Chern character]\label{def:gamma_q_chern}
	Let $E \to X$ be a vector bundle, and denote by
        $\delta_1, \dots, \delta_m $ its Chern roots. 
	The \textit{$q$-Chern character} $\textnormal{ch}_q(E) \in H^*(X;\mathbb{Q})$ is defined by
	\[
		\textnormal{ch}_q(E)
		:= (\log q)^{\operatorname{deg}} \circ
                \operatorname{ch} (E)=
		\sum_{j=1}^m
		q^{\delta_j}
		\in
		H^*(X;\mathbb{C}).
	\]
\end{definition}

\begin{theorem}\label{thm:gamma_comparison_theorem_in_quantum_k_theory}
	Let $X=X_{\mu,\left( \mathbb{R}_{>0} \right)^2}$ be a symplectic toric Fano manifold of Picard rank 2.
	Consider the associated $q$-oscillatory integral
        $\mathcal{I}^{K-\textnormal{th}}$ of Definition
        \ref{def:qosc_q_oscillatory_integral_real} and the
        $K$-theoretic $I$-function $I_{X_{\mu,K}}^{K-\textnormal{th}}$
        of Definition \ref{def:gamma_k_th_I_function}. If $q > 1$ and $Q_1, Q_2\in q^\ZZ$, then the two
        functions are related by the following identity: 
	\[
		\mathcal{I}^{K-\textnormal{th}}(q,Q_1,Q_2)
		=
		\int_{
			\left[
				X
			\right]
		}
		\widehat{\gamma}_q (TX) \cup \textnormal{ch}_q\left(
                  I_{X}^{K-\textnormal{th}}(q,Q_1,Q_2) \right), 
	\]
	where $\int_{[X]}$ denotes the cap product with the fundamental class $[X] \in H_*(X_{\mu,K};\mathbb{C})$, $\widehat{\gamma}_q(TX_{\mu,K})$ is the $q$-gamma class of Definition \ref{def:qosc_q_gamma_class} and $\textnormal{ch}_q$ is the $q$-Chern character of Definition \ref{def:gamma_q_chern}.
\end{theorem}

Our strategy to prove this identity goes as follows: we use the $q$-Mellin transform and its inversion formula to write the $q$-oscillatory integral $\mathcal{I}^{K-\textnormal{th}}$ as a Jeffrey--Kirwan residue.
Then, we identify this Jeffrey--Kirwan residue as an intersection product using the results contained in Section 2 of \cite{Szenes_Vergne:jeffrey_kirwan_residues}, which will match with the right hand side of the identity we are trying to prove.
Let us begin by defining the $q$-Mellin transform and its inverse.
Then, we will state two computational lemmas, then give a proof of the theorem.

\begin{definition}[$q$-Mellin transform]
	The $q$\textit{-Mellin transform of a function} $f$ is the formal Jackson integral
	\[
		\mathcal{M}_q(f)(p)
		:=
		\left[
			\int_0^\infty
		\right]_q
		f(z)z^{p-1} d_q z
		=
		\sum_{n=-\infty}^\infty
		f(q^n)q^{np}.
	\]
\end{definition}
Notice that, at the level of functional operators, we have
\begin{align*}
	\mathcal{M}_q(q^{z \partial_z}) &= q^{-p} \cdot,
	\\
	\mathcal{M}_q(z \cdot) &= \tau_1,
\end{align*}
where $\tau_1 g(p)=g(p+1)$ is the difference operator.
Just like the Mellin transform changes a differential equation into a difference equation, the $q$-Mellin transform changes a $q$-difference equation into a difference equation.
In general, the classical Mellin transform of a $q$-difference equation is not a difference equation.

\begin{proposition}[\cite{Fitouhi_Bettaibi_Brahim:q_mellin}, Proposition 3]
	Let $f$ be a function defined over the $q$-spiral
        $q^\mathbb{Z}$, and assume there exist real numbers $u>v \in
        \mathbb{R}$, such that, 
	\begin{align*}
		f(x) = O_{x \to 0^+}(x^u)
		\quad \mbox{and} \quad
		f(x) = O_{x \to +\infty}(x^v).
	\end{align*}
	Then, the $q$-Mellin transform $\mathcal{M}_q f(p)$ is well
        defined in the complex strip $\{ p \in \mathbb{C} : -u <
        \textnormal{Re}(p) < -v \}$. 
\end{proposition}

\begin{theorem}[$q$-Mellin inversion formula; \cite{Fitouhi_Bettaibi_Brahim:q_mellin}, Theorem 2 p.315]\label{thm:qosc_q_mellin_inversion_formula}
	Consider a function $f$ defined on the $q$-spiral $q^\mathbb{Z}$, and assume its Mellin transform is well defined on a complex strip $\{ p \in \mathbb{C} : u < \textnormal{Re}(p) < v \}$, for some $u,v \in \mathbb{R}, u<v$.
	Let $\varepsilon \in \{ p \in \mathbb{C} : u < \textnormal{Re}(p) < v \}$.
	Then, for any $x \in q^\mathbb{Z}$,
	\[
		f(x)
		=
		\frac{\log(q)}{2\pi \ii}
		\int_{\varepsilon - \ii \pi / \log(q)}^{\varepsilon + \ii \pi / \log(q)}
		\mathcal{M}_q f(p) x^{-p} dp.
	\]
\end{theorem}



We now state two lemmas needed in the proof of Theorem \ref{thm:gamma_comparison_theorem_in_quantum_k_theory}.

\begin{lemma}\label{lemma:qosc_mellin_transform_computation}
	Let $X_{\mu,K}$ be a Fano symplectic toric manifold of Picard rank 2 as in Proposition \ref{prop:gamma_model_toric_picard_rank_two} and let $q > 1$.
	Denote by $\mathcal{I}^{K-\textnormal{th}}$ the $q$-oscillatory integral along the real Lefschetz thimble of Definition \ref{def:qosc_q_oscillatory_integral_real}.
	The $q$-Mellin transform of the $q$-oscillatory integral is given by
	\[
		\mathcal{M}_{q} \mathcal{I}^{K-\textnormal{th}}(q,p_1,p_2)
		=
		\gamma_{q}(p_1)^N \prod_{j=0}^k \gamma_{q}(p_2 - a_j p_1),
	\]
	where $\gamma_q(p) = \frac{\left(q^{-1};q^{-1}\right)_\infty}{\left(q^{-p};q^{-1}\right)_\infty}$.
\end{lemma}

The proof of this lemma will be done in Section \ref{sssection:proof_lemma_qosc_1}.

\begin{lemma}\label{lemma:qosc_inverse_q_mellin_as_residues}
	Let $g(p_1,p_2)$ be the $q$-Mellin transform of the
        $q$-oscillatory integral, computed in the previous Lemma
        \ref{lemma:qosc_mellin_transform_computation}, that is, 
	\[
		g(p_1,p_2)
		:=
		\mathcal{M}_{q} \mathcal{I}^{K-\textnormal{th}}(q,p_1,p_2)
		=
		\gamma_{q}(p_1)^N \prod_{j=0}^k \gamma_{q}(p_2 - a_j p_1).
	\]
	Then, the $q$-inverse Mellin transform of $g$ evaluated at $Q_1, Q_2 \in q^\mathbb{Z}$ can be computed by the following iterated residues:
	\[
		\mathcal{M}_q^{-1}g (Q_1,Q_2)
		=
		(\log q)^2
 \textnormal{Res}_{x_2 = 0} \textnormal{Res}_{x_1 = 0}
		\omega^{K-\textnormal{th}}(x_1,x_2)
		dx_1dx_2,
	\]
where
	\begin{align*}
		&\omega^{K-\textnormal{th}}(x_1,x_2)
		=
		\\
		&\sum_{d_1,d_2 \geq 0} Q_1^{-x_1+d_1} Q_2^{-x_2+d_2}
		\left(
			\gamma_q(x_1)
			\frac{
				\prod_{r=-\infty}^0 \left(1-q^{-x_1+r}\right)
			}
			{
				\prod_{r=-\infty}^{d_1} \left(1-q^{-x_1+r}\right)
			}
		\right)^N
		\prod_{j=0}^k
		\gamma_q(x_2 - a_j x_1)
		\frac{
		\prod_{r=-\infty}^0
		\left(
			1-q^{- x_2 + a_j x_1 + r}
		\right)
		}
		{
		\prod_{r=-\infty}^{d_2-a_j d_1}
		\left(
			1-q^{- x_2 + a_j x_1 + r}
		\right)
		}.
	\end{align*}
\end{lemma}
The proof of this lemma will be done in Section \ref{sssection:proof_lemma_qosc_2}.

\begin{proof}[Proof of Theorem \ref{thm:gamma_comparison_theorem_in_quantum_k_theory}]
	Using Lemma \ref{lemma:qosc_mellin_transform_computation}, we obtain that
	\[
		\mathcal{M}_{q} \mathcal{I}^{K-\textnormal{th}}(q,p_1,p_2)
		=
		\gamma_q(p_1)^N \prod_{j=0}^k \gamma_q(p_2 - a_j p_1).
	\]
	Using the $q$-Mellin inversion formula of Theorem \ref{thm:qosc_q_mellin_inversion_formula} and the computation of the inverse $q$-Mellin transform of Lemma \ref{lemma:qosc_inverse_q_mellin_as_residues}, we obtain that
	\[
		\mathcal{I}^{K-\textnormal{th}}(q,Q_1,Q_2)
		=
		(\log q)^2
\textnormal{Res}_{x_2 = 0} \textnormal{Res}_{x_1 = 0}
		\omega ^{K-\textnormal{th}} (x_1,x_2)
		dx_1dx_2,
	\]
where $\omega ^{K-\textnormal{th}} (x_1,x_2)$ is the form defined in
Lemma \ref{lemma:qosc_inverse_q_mellin_as_residues}. 
	Using Theorem \ref{thm:gamma_computation_of_intersection_products}, we identify the iterated residues above with the intersection product
	\[
		\mathcal{I}^{K-\textnormal{th}}(q,Q_1,Q_2)
		=
		(\log q)^2 \,
		\int_{[X_{\mu,K}]} x_1^N \prod_{j=0}^k (x_2 - a_j x_1) \omega^{K-\textnormal{th}}(x_1,x_2),
	\]
where $x_i=c_1(P_i^{-1})=-c_1(P_i)$. 	Finally,
	\begin{align*}
		(\log q)^2 \,
		x_1^N &\prod_{j=0}^k (x_2 - a_j x_1) \omega^{K-\textnormal{th}}(x_1,x_2)
		=
		(\log q)^2 \,
		\left[
			x_1^N \gamma_q(x_1)^N \prod_{j=0}^k (x_2 - a_j x_1) \gamma_q(x_2 - a_j x_1)
		\right]
		\times
		\\
		&\times
		\left[
			Q_1^{-x_1} Q_2^{-x_2}
			\sum_{d_1,d_2 \geq 0}
			Q_1^{d_1} Q_2^{d_2}
			\left(
				\frac{
					\prod_{r=-\infty}^0 \left(1-q^{-x_1+r}\right)
				}
				{
					\prod_{r=-\infty}^{d_1} \left(1-q^{-x_1+r}\right)
				}
			\right)^N
			\prod_{j=0}^k
			\frac{
				\prod_{r=-\infty}^0
				\left(
					1-q^{- x_2 + a_j x_1 + r}
				\right)
			}
			{
				\prod_{r=-\infty}^{d_2-a_j d_1}
				\left(
					1-q^{- x_2 + a_j x_1 + r}
				\right)
			}
		\right].
	\end{align*}
	The factor on the first line on the RHS of the above identity is the
        $q$-gamma class $\widehat{\gamma}_q\left(TX \right)$. Note
        that if $Q_i\in q^\ZZ$, then
        $\ell(Q_i)=-\tfrac{\log Q_i}{\log q}$. Therefore, the $q$-Chern
        character of $P^{-\ell(Q)}$ is $Q_1^{-x_1} Q_2^{-x_2}$ and the
        factor on the second line on the RHS of the above identity is
        the $q$-Chern character of the $K$-theoretic $I$-function
        $\textnormal{ch}_q\left( I_{X}^{K-\textnormal{th}}(q,Q_1,Q_2)
        \right)$. 
\end{proof}

\subsubsection{Proof of Lemma \ref{lemma:qosc_mellin_transform_computation}}\label{sssection:proof_lemma_qosc_1}

This lemma is a corollary of the following $q$-integral representation
of the $q$-gamma function due to
Koelink--Koornwinder \cite{KK:integral_representation_of_qgamma} (see
also Theorem 3.2 in \cite{DeSole_Kac:integral_representation_of_qgamma}) .
\begin{proposition}\label{prop:qosc_qgamma_integral_representation}
	Suppose that $q >1$.
	The $q$-gamma function admits the following $q$-integral representation:
	\[
		\Gamma_{q^{-1}}(p)
		=\left(1-q^{-1}\right)^{1-p}
		\left[
			\int_0^{\infty}
		\right]_q
		E_{q^{-1}}\left(\frac{x}{1-q}\right)
		x^p
		\frac{d_qx}{x}.
	\]
\end{proposition}

\begin{proof}[Proof of Lemma \ref{lemma:qosc_mellin_transform_computation}]
By definition, the $q$-Mellin transform $\mathcal{M}_{q}
\mathcal{I}^{K-\textnormal{th}}(q,p_1,p_2)$ is given by the expression 
	\[
		\left[
		\int_{(\mathbb{R}_{>0})^2}
		\right]_{q}
		\left(
		\left[
		\int_{\Gamma_\mathbb{R} }
		\right]_q
E_{q^{-1}}\left(\frac{x_1}{1-q}\right)\cdots E_{q^{-1}}\left(\frac{x_n}{1-q}\right)
		\omega_{\pi^{-1}(Q),q}
		\right)
		Q_1^{p_1} Q_2^{p_2}
		\frac{d_q Q_1}{Q_1}
		\frac{d_q Q_2}{Q_2}.
	\]
By definition, the above integral is a sum over all $x_j\in q^\ZZ$
($j\neq N,N+1$) and all $Q_1,Q_2\in q^\ZZ$ satisfying the relations
$Q_i=\prod_{j=1}^n x_j^{m_{ij}}$. These relations determine $x_N$ and
$x_{N+1}$ in terms of $x_j$ ($j\neq N,N+1$) and $Q_1,Q_2$. Note that
our sum can be viewed equivalently as a sum over all $x_j\in q^\ZZ$
($1\leq j\leq n$), where now we use the relations to solve for $Q_1$
and $Q_2$ in terms of $x_j$ ($1\leq j\leq n$). Clearly, the above
integral splits into a product of 1-dimensional integrals
\[
\mathcal{M}_{q} \mathcal{I}^{K-\textnormal{th}}(q,p_1,p_2) =
\prod_{j=1}^n 
\left[ \int_0^\infty \right]_{q} 
E_{q^{-1}}\left(\frac{x_j}{1-q}\right) x_j^{\alpha_j(p)} \frac{d_q x_j}{x_j},
\]
where $\alpha_j(p)=p_1 m_{1j}+p_2 m_{2j}$. Recalling the explicit
formulas for the moment matrix, we get 
	\[
		\mathcal{M}_{q} \mathcal{I}^{K-\textnormal{th}}(q,p_1,p_2)
		=
		\left(
			\left[
			\int_0^\infty
			\right]_{q} 
E_{q^{-1}}\left(\frac{x}{1-q}\right) \, x^{p_1}\, 
			\frac{d_q x}{x}
		\right)^N
		\prod_{j=0}^k
\left(
		\left[
			\int_0^\infty
		\right]_{q} 
E_{q^{-1}}\left( \frac{x}{1-q} \right) \, x^{p_2-a_j p_1}\, 
		\frac{d_q x}{x}\right).
	\]
Finally, we use Proposition \ref{prop:qosc_qgamma_integral_representation} to get
	\begin{align*}
		\left[\int_0^\infty\right]_q
E_{q^{-1}}\left(\frac{x}{1-q}\right) \, 
		x^{p}\, 
		\frac{d_q x}{x}
		=
		\left(
			1-q^{-1}
		\right)^{p-1}
		\Gamma_{q^{-1}}(p) = \gamma_q(p).
	\end{align*}
\end{proof}

\subsubsection{Proof of Lemma \ref{lemma:qosc_inverse_q_mellin_as_residues}}\label{sssection:proof_lemma_qosc_2}

The proof of this lemma relies on proving a contour deformation result to compute the integral in the $q$-Mellin inversion formula of Theorem \ref{thm:qosc_q_mellin_inversion_formula} using the residue theorem.
We will need a $q$-analogue of the Stirling formula.

\begin{proposition}[\cite{Moak:q_stirling}, Equation 2.13]\label{prop:gamma_moak_qstirling}
	Let $q < 1$ in this proposition only.
	For the usual $q$-gamma function $\Gamma_q$, the following $q$-analogue of the Stirling's formula holds for any $z$ such that $Re(z) > 0$:
	\[
		\log \Gamma_q(z)
		\sim
		\left(
			z - \frac{1}{2}
		\right)
		\log\left(
			\frac{1-q^z}{1-q}
		\right)
		+
		\frac{1}{\log(q)}
		\textnormal{Li}_2(1-q^z)
		+
		C_q
		+
		\sum_{k=1}^\infty
		\frac{B_{2k}}{(2k)!}
		\left(
			\frac{\log(q)}{q^z-1}
		\right)^{2k-1}
		q^z
		p_{2k-3}(q^z),
	\]
	where $C_q \in \mathbb{C}$ is some constant depending on $q$, and $p_k$ is a degree $k$ polynomial satisfying the recursion relation
	\begin{align*}
	p_0(z)=1,\quad 
		\textnormal{and }
		\forall k>0, \,
		p_k(z)
		=
		(z-z^2)p'_{k-1}(z)+(kz+1)p_{k-1}(z).
	\end{align*}
\end{proposition}

\begin{proof}[Proof of Lemma \ref{lemma:qosc_inverse_q_mellin_as_residues}]
	To apply Theorem \ref{thm:qosc_q_mellin_inversion_formula}, we choose real numbers $\varepsilon_1, \varepsilon_2 > 0$ such that $0 < \textnormal{max}(1,a_j) \varepsilon_1 < \varepsilon_2$.
	We also choose to integrate first with respect to the coordinate $p_1$, then with respect to $p_2$.
	Thus, we will compute
	\[
		\mathcal{M}_q^{-1} g (Q_1,Q_2)
		:=
		\left(
			\frac{\log(q)}{2\pi \ii}
		\right)^2
		\int_{p_2=\varepsilon_2 - \ii \pi / \log(q)}^{\varepsilon_2 + \ii \pi / \log(q)}
		\left(
		\int_{p_1=\varepsilon_1 - \ii \pi / \log(q)}^{\varepsilon_1 + \ii \pi / \log(q)}
			g(p_1,p_2) Q_1^{-p_1} Q_2^{-p_2} dp_1
		\right) dp_2,
	\]
where
	\[
		g(p_1,p_2) = 
		\gamma_q(p_1)^N \prod_{j=0}^k \gamma_q(p_2 - a_j p_1).
	\]
	We consider the following closed curve $\mathcal{C}(M)$, with the usual orientation :
	
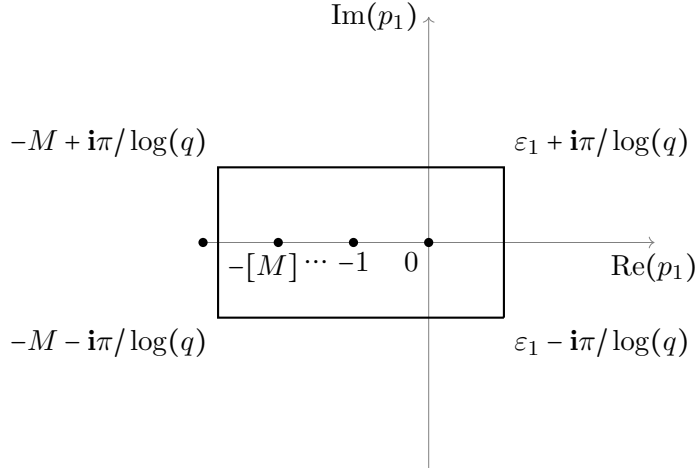
\begin{center}
		\begin{tikzpicture}
			\draw[help lines,->] (-3,0) -- (3,0) coordinate (xaxis);
			\draw[help lines,->] (0,-3) -- (0,3) coordinate (yaxis);

			\draw
				(0,0)  node[scale=3](O){.}
				(-1,0)  node[scale=3](O){.}
				(-2,0)  node[scale=3](O){.}
				(-3,0)  node[scale=3](O){.}
			;
			\path[draw,line width=0.8pt,postaction=decorate] (1,-1) node[below right] {$\varepsilon_1 - \ii \pi / \log(q)$} -- (1,1) node[above right] {$\varepsilon_1 + \ii \pi / \log(q)$} -- (-2.8,1) node[above left] {$-M + \ii \pi / \log(q)$} -- (-2.8,-1) node[below left] {$-M - \ii \pi / \log(q)$} -- (1,-1);

			\node[below] at (xaxis) {$\textnormal{Re}(p_1)$};
			\node[left] at (yaxis) {$\textnormal{Im}(p_1)$};
			\node[below left] {$0$};
			\node[below] at (-1,0) {$-1$};
			\node at (-1.5,-0.3) {$\cdots$};
			\node[below] at (-2.2,0) {$-[M]$};
		\end{tikzpicture}
	\end{center}
where we picked some large number $M > \! > 0$ such that this curve does not encounter any of the poles of the function $g$.
	Note also that $\varepsilon_1$ should be small enough so that the poles of the form $p_1 = \frac{k+p_2}{a_j}, k \in \mathbb{Z}_{\geq 0}$ sit outside of the curve.
	Let us prove that for the inverse $q$-Mellin transform, when $M \to +\infty$, integrating along the closed curve $\mathcal{C}(M)$ is the same as integrating along the cycle in the formula of Theorem \ref{thm:qosc_q_mellin_inversion_formula}, $\int_{p_1=\varepsilon_1 - \ii \pi / \log(q)}^{\varepsilon_1 + \ii \pi / \log(q)}$.

	For the horizontal lines of $\mathcal{C}(M)$, let $t \in (-M,\varepsilon_1) $ and $Q_1=q^{k_1}, k_1 \in \mathbb{Z}$.
	Then, notice that the integrand satisfies
	\[
		g(t + \ii \pi / \log(q),p_2) Q_1^{-(t + \ii \pi / \log(q))}
		=
		g(t - \ii \pi / \log(q),p_2) Q_1^{-(t - \ii \pi / \log(q))}
	\]
	Therefore, when computing
	$
		\int_{p_1 \in \mathcal{C}(M)}^{\varepsilon_1 + \ii \pi / \log(q)}
		g(p_1,p_2) Q_1^{-p_1} Q_2^{-p_2} dp_1
	$,
	the integrals along the horizontal lines cancel each other.

	For the vertical line $(-M + \ii \pi / \log(q),-M - \ii \pi / \log(q))$ of $\mathcal{C}(M)$, write $p_1=a+ib, a < 0$.
	For such choice (recall $q > 1$),
	\[
		|
			\gamma_q(p_1)
		|
		=
		\left|
			\left( 1-q^{-1} \right)
		\frac{
			\left(q^{-1};q^{-1}\right)_\infty
		}{
			\left(q^{-p_1};q^{-1}\right)_\infty
		}
		\right|
		\leq
		|1-q| \left(q^{-1};q^{-1}\right)_\infty q^a.
	\]
Therefore $|\gamma_q|$ has exponential decay as $a \to -\infty$.
Let us explain that the function given by
$|\gamma_q(p_1)\gamma_q(-p_1)|$ will also have exponential decay when
$\textnormal{Re}(p_1) \to \pm \infty$. 
	Using Moak's $q$-analogue of the Stirling formula (see
        Proposition \ref{prop:gamma_moak_qstirling}, recall
        $\gamma_q(p)=(1-q^{-1})^{p-1} \Gamma_{q^{-1}}(p)$), we have for
        $\textnormal{Re}(p)>0$, 
	\[
		\log \left|\Gamma_{q^{-1}}(p)\right| \sim_{p \to \infty} -\textnormal{Re}(p) \log(1-q^{-1}),
	\]
while, for $\textnormal{Re}(p)<0, p \notin \mathbb{Z}_{\leq 0}$, we have
	\[
		\left|
			\left(
				q^{-p};q^{-1}
			\right)
		\right|
		\sim
		\left|
			\prod_{k : -\textnormal{Re}(p)-k \geq 0} (1-q^{p-k})
		\right|
		\sim
		q^{-\textnormal{Re}(p)(-\textnormal{Re}(p)+1)/2}.
	\]
	Thus, the function given by $|\gamma_q(p_1)\gamma_q(-p_1)|$
        has exponential decay when $\textnormal{Re}(p_1) \to \pm
        \infty$. 
	Finally, we can write the integrand as
	\[
		\left(
			\frac{\gamma_q(p_1)^N}{\prod_{j=0}^k \gamma_q(a_j p_1 - p_2)}
		\right)
		\left(
			\prod_{j=0}^k \gamma_q(p_2 - a_j p_1)\gamma_q(-p_2 + a_j p_1)
		\right).
	\]
Our previous observation gives that the second factor in the big parentheses
has exponential decay.
	Using the Fano condition $N - \sum a_j > 0$, the first factor $\left| \frac{\gamma_q(p_1)^N}{\prod_{j=0}^k \gamma_q(a_j p_1 - p_2)} \right|$ also has exponential decay.
	Therefore, we have proved for $|Q_1| < 1$,
	\[
		\lim_{M \to +\infty}
		\int_{p_1 \in \mathcal{C}(M)}
		g(p_1,p_2) Q_1^{-p_1} Q_2^{-p_2} dp_1
		=
		\int_{p_1=\varepsilon_1 - \ii \pi / \log(q)}^{\varepsilon_1 + \ii \pi / \log(q)}
		g(p_1,p_2) Q_1^{-p_1} Q_2^{-p_2} dp_1.
	\]
Now, we apply the residue theorem to the left hand side, using that the poles of the integrand inside the contour are exactly at $p_1 \in \mathbb{Z}_{\leq 0}$. We obtain
	\[
		\frac{1}{2\pi \ii} \int_{p_1=\varepsilon_1 - \ii \pi / \log(q)}^{\varepsilon_1 + \ii \pi / \log(q)}
		g(p_1,p_2) Q_1^{-p_1} Q_2^{-p_2} dp_1
		=
		\sum_{d_1 \geq 0}
		\textnormal{Res}_{p_1=-d_1} g(p_1,p_2) Q_1^{-p_1} Q_2^{-p_2} dp_1.
	\]
	We will do a similar contour deformation for the coordinate $p_2$, using the same contour as for the previous coordinate $p_1$.
	Using the same reasoning, we obtain that
	\begin{align*}
		&\lim_{M \to +\infty}
		\int_{p_2 \in \mathcal{C}(M)}
		\sum_{d_1 \geq 0}
		\textnormal{Res}_{p_1=-d_1} g(p_1,p_2) Q_1^{-p_1} Q_2^{-p_2} dp_1
		dp_2
		\\
		&=
		\int_{p_2=\varepsilon_2 - \ii \pi / \log(q)}^{\varepsilon_2 + \ii \pi / \log(q)}
		\sum_{d_1 \geq 0}
		\textnormal{Res}_{p_1=-d_1} g(p_1,p_2) Q_1^{-p_1} Q_2^{-p_2} dp_1
		dp_2.
	\end{align*}
	Recall that $
		g(p_1,p_2)
		:=
		\gamma_{q}(p_1)^N \prod_{j=0}^k \gamma_q(p_2 - a_j p_1)$ and $a_0:=0
	$, 
	therefore when applying the residue theorem, the poles that are inside the contour are exactly given by $p_2 \in \mathbb{Z}_{\leq 0}$.
	Next, we do a change of variable $p_1 = x_1 -d_1, p_2 = x_2 - d_2$ and use Fubini theorem to permute the sums and residues.
	To obtain the identity announced in the statement of the
        lemma, it remains to use the difference equation for the
        function $\gamma_q$, that is,  
	\[
		\gamma_q(p-1)
		=
		\frac{(q^{-1};q^{-1})_\infty}{(q^{-p+1};q^{-1})_\infty}
		=
		\frac{1}{1-q^{-p+1}}\gamma_q(p).
	\]
\end{proof}


\subsection{Confluence of the comparison theorem}\label{sec:confl_comp_thm}

We would like to use Theorems \ref{t1} and
\ref{thm:gamma_comparison_theorem_in_quantum_k_theory} to give a 2nd
proof of Theorem \ref{thm:gamma_cohomology_oscillatory_i_function} in the case of Fano toric
manifolds of Picard rank 2.  If we compare the proofs of Theorem 
\ref{thm:gamma_comparison_theorem_in_quantum_k_theory} and Theorem 
\ref{thm:gamma_cohomology_oscillatory_i_function} for Picard rank 2, then we see that the
argument in the $K$-theoretic case is somewhat easier. Therefore, it
looks promising that our proof of Theorem
\ref{thm:gamma_comparison_theorem_in_quantum_k_theory} generalizes to
all weak Fano toric manifolds and hence by using the confluence result
from Theorem \ref{t1}, we would be able to obtain a new proof of
Iritani's theorem. 


Suppose now that $X=X_{\mu,K}$ is a Fano toric manifold of Picard rank
2. The main difficulty in our argument comes from the fact that the oscillatory integral
$\mathcal{I}^{K-\textnormal{th}}(q,Q_1,Q_2)$ is defined only for
$Q_1,Q_2\in q^\ZZ$. Therefore, rescaling $Q_i$ by $(q-1)^{m_i}$ and
passing to the limit $q\to 1$ should be done more carefully. Let us
define a sequence of integers $a_k>0$ ($k\geq 1$) as follows. The
equation $q^{k+1}-q^k-1=0$ has a unique solution $q_k>1$. Put
$a_k=q_k-1$. Note that $q_k-1= q_k^{-k}\in q_k^{\ZZ}$. 
\begin{proposition}\label{prop:ak}
a) We have $0<a_k<\frac{1}{\sqrt{k}}$ for all sufficiently large $k$. 

b) If $Q>0$ is a real number, then there exists a sequence of integers
$b_k$, such that, $q_k^{b_k}\to Q$ when $k\to \infty$. 
\end{proposition}
\proof
a) Put $g(q)=q^{k+1}-q^k-1$. We have to prove that
$q_k=1+a_k<1+1/\sqrt{k}$ for $k\gg 0$. Since $g'(q) = q^{k-1} ((k+1)
q- k)>0$ for $q>1$. The inequality is equivalent to
$g(1+1/\sqrt{k})>g(q_k)=0$. On the other hand, 
\ben
g(1+1/\sqrt{k}) = \left(1+\frac{1}{\sqrt{k}}\right)^k\,
\frac{1}{\sqrt{k}} -1 > \frac{2^{\sqrt{k}}}{\sqrt{k}} -1
\een
and the above expression tends to $+\infty$ when $k\to +\infty$. In
particular, the inequality that we need holds for $k\gg 0$. 

b) It is sufficient to consider the case when $Q>1$, that is, $c=\log
Q>0$. We claim that the interval $[c/a_k, (c+1/\sqrt{k})/a_k]$ has
length $>1$ for $k\gg 0$ and hence it contains at least one integer
$b_k$. Indeed,  
\ben
\frac{1}{a_k}\,\left( c+ \frac{1}{\sqrt{k}} \right) -
\frac{c}{a_k} = \frac{1}{\sqrt{k} a_k} >1,
\een
where the last inequality holds for $k\gg 0$ thanks to part a). Let
$b_k$ be an integer in the above interval, then
\ben
0< a_k b_k -c <\frac{1}{\sqrt{k}}.
\een
In particular, $a_k b_k \to c$ when $k\to \infty$. Finally, 
\ben
q_k^{b_k} = (1+a_k)^{b_k} = \left( (1+a_k)^{1/a_k} \right)^{a_k b_k}
\to e^{c} = Q,
\een
where we used that according to part a), $a_k\to 0$, so
$(1+a_k)^{1/a_k}\to e$. 

\begin{proof}[Proof of Theorem \ref{thm:gamma_cohomology_oscillatory_i_function}]

We are going to apply the identity in Theorem  
\ref{thm:gamma_comparison_theorem_in_quantum_k_theory}
in the following settings. Let $Q_i^\circ$ ($i=1,2$) be positive real
numbers. Let us choose integer sequences $b^i_k$ ($k\geq 1$), such that,
$q_k^{b^i_k}\to Q_i^\circ$ when $k\to \infty$ (see Proposition
\ref{prop:ak}, b)). Let us fix an integer $k>0$ and set $q=q_k$, $Q_i=
(q_k-1)^{m_i} q_k^{b^i_k}$, where $m_i=\sum_{j=1}^n m_{ij}$. Let us
rewrite the $q$-oscillatory integral  
\beq\label{oi_confl}
(1-q^{-1})^{N-1+k} \left[\int_{\Gamma_{\mathbb{R}}}\right]_q 
\prod_{j=1}^n E_{q^{-1}}\left( \frac{x_j}{1-q}\right) \, 
\omega_{\pi^{-1}(Q),q}
\eeq
in two different ways. First, by definition the integral is a sum over
all $x_j\in q^\ZZ$ satisfying the relations $\prod_{j=1}^n x_j^{m_ij}
= Q_i$. Let us change the variables by $x_j= (q-1) y_j$. Then, since
$q-1=q_k-1=q_k^{-k}\in q^\ZZ$, the integral becomes a sum over all
$y_j\in q^\ZZ$ satisfying the relations $\prod_j y_j^{m_{ij}} =
(q-1)^{-m_i} Q_i $. Therefore, the integral turns into 
\ben
(1-q^{-1})^{N-1+k} 
\left[\int_{\Gamma_{\mathbb{R}}}\right]_q 
\prod_{j=1}^n E_{q^{-1}}\left( -y_j\right) \, 
\omega_{\pi^{-1}((q-1)^{-m_1}Q_1, (q-1)^{-m_2} Q_2),q}.
\een
In the limit when $k\to \infty$, we have $q=q_k\to 1$, the Jackson
integral $(1-q^{-1}) ^{N-1+k}
\left[\int_{\Gamma_{\mathbb{R}}}\right]_q$ tends to the Riemannian
integral $\int_{\Gamma_{\mathbb{R}}}$ while the integrand tends to 
$\prod_j e^{-y_j} \omega_{\pi^{-1}(Q^\circ)}$, where
$Q^\circ=(Q_1^\circ,Q_2^\circ)$ and we used that $(q-1)^{-m_i} Q_i =
q_k^{b^i_k}\to Q_i^\circ$. In other words, in the limit $k\to \infty$,
the Jackson integral \eqref{oi_confl} tends to
$\mathcal{I}^{\textnormal{coh}}(1,Q^\circ)$. 

The integral \eqref{oi_confl} coincides with 
$(1-q^{-1}) ^{N-1+k}
\mathcal{I}^{K-\textnormal{th}}(q,Q_1,Q_2)$. Using Theorem 
\ref{thm:gamma_comparison_theorem_in_quantum_k_theory}, we rewrite
\eqref{oi_confl} as
\beq\label{cc_confl}
(1-q^{-1}) ^{N-1+k} \int_{[X]} 
\widehat{\gamma}_q(TX) 
\operatorname{ch}_q\left(
I_X^{K-\textnormal{th}}(q,Q) \right),
\eeq
where $X=X_{\mu,K}$. Under the Fano condition the $I$-function
essentially coincides with the $J$-function (see 
\cite{Givental:perm_toric_q_hypergeometric}, Theorem  p.8), that is, 
\beq
\label{gamma_mirror_map_in_qk}
I_X^{K-\textnormal{th}}(q,Q) = 
(1-q)^{-1} \, P^{-\ell(Q)}\,
J_X^{K-\textnormal{th}}(q,Q).
\eeq
Since $Q_i\in q^\ZZ$, we have 
$P^{-\ell(Q)} = \prod_i P_i^{\frac{\log Q_i}{\log q}}$. The $q$-Chern
  character of the $I$-function takes the form 
\begin{align*}
\operatorname{ch}_q\left(
I_X^{K-\textnormal{th}}(q,Q) \right) & = 
(\log q)^{\operatorname{deg}} 
\operatorname{ch}\left(
P_1^{\frac{\log Q_1}{\log q}}P_2^{\frac{\log Q_2}{\log q} }
(1-q)^{-1} J_X^{K-\textnormal{th}}(q,Q)\right) = \\
& = 
-\frac{\log q}{q-1}\, e^{-\sum_i p_i \log Q_i}  \, 
(\log q)^{\operatorname{deg}-1} 
\operatorname{ch}\left( 
J_X^{K-\textnormal{th}}(q,Q)
\right).
\end{align*}
Let us introduce also the $\Gamma_{q^{-1}}$-class of a vector bundle
$E$ by 
\ben
\widehat{\Gamma}_{q^{-1}}(E):= \prod_{\delta: \textnormal{ Chern roots of } E}
\delta \, \Gamma_{q^{-1}} (\delta).
\een
Note that 
\ben
\widehat{\gamma}_q(TX) = 
\left( \frac{\log q}{1-q^{-1}} \right)^2 
(1-q^{-1})^{-N+1-k} (1-q^{-1})^\rho \Gamma_{q^{-1}}(TX)
\een
and that, since $\rho= m_1 p_1 + m_2 p_2$,
\ben
(1-q^{-1})^\rho  \, e^{-\sum_i p_i \log Q_i} = 
q^{-\rho} \, e^{-\sum_i p_i \log \left((q-1)^{-m_i} Q_i\right)}.
\een
We get that \eqref{cc_confl} can be written as 
\ben
-\left( \frac{\log q}{1-q^{-1}} \right)^3
\int_{[X]} \Gamma_{q^{-1}}(TX) 
q^{-\rho-1} \, 
e^{-\sum_i p_i \log \left((q-1)^{-m_i} Q_i\right)}\, 
(\log q)^{\operatorname{deg}-1} 
\operatorname{ch}
\left( 
J_X^{K-\textnormal{th}}(q,Q)\right).
\een
Let us compute the limit of the above expression when $k\to
\infty$. Since $Q_i= (q_k-1)^{m_i} q_k^{b^i_k}$ and $q_k^{b^i_k}\to
Q_i^\circ$, Theorem \ref{t1} implies that 
$(\log q)^{\operatorname{deg}-1} 
\operatorname{ch}
\left( 
J_X^{K-\textnormal{th}}(q,Q)\right)\to
J_X^{\textnormal{coh}}(1,Q^\circ)$. Therefore, the limit of
\eqref{cc_confl} is 
\ben
-\int_{[X]} \widehat{\Gamma}(TX) \, e^{-\sum_i p_i \log Q_i^\circ}\, 
J_X^{\textnormal{coh}}(1,Q^\circ).
\een
Since $J_X^{\textnormal{coh}}(1,Q^\circ) =
-I_X^{\textnormal{coh}}(1,Q^\circ)$, we get precisely the identity in
Theorem \ref{thm:gamma_cohomology_oscillatory_i_function} for the case
when $z=1$. The general case follows from the homogeneity property 
\eqref{eqn:intro_cohom_J_recover_z}.  
\end{proof}

\section*{Acknowledgement}

We are thankful to the anonymous referee for spotting out several gaps in our proof of Theorem \ref{t1}.
This work is supported by World Premier International Research Center Initiative (WPI Initiative), Ministry of Education, Culture, Sports, Science and Technology, Japan.
The work of the first author is partially supported by Japan Society for the Promotion of Science KAKENHI [JP19F19802] and Grant-in-aid (Kiban C) [17K05193] and [22K03265].
The second author conducted their research while being a JSPS International Research Fellow (Standard program at Kavli IPMU, University of Tokyo). They are also supported by JSPS KAKENHI [JP19F19802]. 
\appendix
\section{Proof of the Givental--Tonita recursion}\label{sec:GT_recursion}

We would like to outline the proof of formula \eqref{GT-recursion}. The idea is to express the integral in \eqref{zeta_stratum} as an integral over the fiber product in \eqref{stem_iso}. The main difficulty is to find the image of the inertia tangent $T_{IX_{0,1,d}}$ and the inertia normal $N_{IX_{0,1,d}}$ bundles on $I_\zeta X_{0,1,d}$ via the isomorphism \eqref{stem_iso}. Strictly speaking, we have to solve this problem for any $B$-point of $I_\zeta X_{0,1,d}$, where $B$ is an arbitrary scheme $B$. However, we will do this only in the case when $B=\operatorname{Spec} \CC$.

\subsection{Virtual tangent space}\label{sec:vts}

Suppose that $\C=(C,s_1,\dots,s_n, f)$ is a point in the moduli space
$X_{g,n,d}$. The restriction of the virtual tangent bundle
$T_{g,n,d}$ on $X_{g,n,d}$ to $\C$ is the virtual vector space
$T(\C)= - T^0(\C)+T^1(\C)-T^2(\C)$, where  
\begin{align}
\nonumber
T^0(\C) & :=H^0(C,\T_C(-s_1-\cdots - s_n)), \\
\nonumber
T^1(\C) & :=H^1(C,\T_C(-s_1-\cdots - s_n)) + H^0(C,f^*\T_X) + 
\bigoplus_{z\in \operatorname{Sing}(C)} T'_z\otimes T''_z, \\
\nonumber
T^2(\C) & :=H^1(C,f^*\T_X).
\end{align}
Here the notation is as follows. $\T_C$ is the sheaf of holomorphic
vector fields on $C$ and $\T_C(-s_1-\cdots - s_n)$ is the sheaf of
holomorphic vector fields vanishing at the marked points
$s_1,\dots,s_n$. $T'_z$ and $T''_z$ are the tangent spaces at $z$ to
the two irreducible components of $C$ that meet at $z$. The groups $T^i(\C)$ come from the deformation theory of the stable map $\C=(C,s_1,\dots,s_n,f)$. Namely, $T^0(\C)$ is 
the Lie algebra of the group of automorphisms $\operatorname{Aut}(C,s_1,\dots,s_n)$, $T^1(\C)$ is the vector space of infinitesimal deformations, and $T^2(\C)$ is the obstruction space. We refer to \cite{Pa} for more details on deformation theory. 

Suppose now that $(\C,g)=(C,s_1,\dots,s_n,f,g)$ is a point in the inertia moduli space $IX_{0,n,d}$. The automorphism $g$ acts on the virtual tangent space $T(\C)$. Let us denote  by  
\ben
T(\C,g)_\lambda:=\{ v\in T(\C)\ |\ g v=\lambda v\}
\een
the eigensubspace with eigenvalue $\lambda$ and let $I_\lambda T_{0,n,d}$ be the virtual vector bundle on $IX_{0,n,d}$ whose fiber at $(\C,g)$ is $T(\C,g)_\lambda$. Note that the inertia tangent bundle is $T_{IX_{0,n,d}}:= I_1T_{0,n,d}$ and the inertia normal bundle is $N_{IX_{0,n,d}} =\oplus_{\lambda: \lambda\neq 1} I_\lambda T_{0,n,d}$. It is convenient to introduce the notation 
\ben
\widetilde{\operatorname{td}}(T_{0,n,d}) = 
\prod_{\lambda\in \CC^*} 
\operatorname{td}_{\lambda^{-1}} (I_\lambda T_{0,n,d}),
\een
where $\operatorname{td}_1$ is the usual Todd class and $\operatorname{td}_\lambda$, $\lambda\neq 1$, is the {\em moving Todd class} 
\ben
\operatorname{td}_\lambda (E) = 
\prod_{\mbox{ Chern roots $x$ of $E$}}   \frac{1}{1-\lambda e^{-x}}.
\een
Then the integral in formula \eqref{zeta_stratum} can be written as 
\beq\label{zeta_stratum_int}
\tau_{\zeta,d,a}(q):=
\int_{[I_\zeta X_{0,1,d}]} 
\widetilde{\operatorname{td}}(T_{0,1,d}) \, 
\frac{\operatorname{ev}_1^* \operatorname{ch}(\Phi^a)}{
1-q\operatorname{ch} (\operatorname{Tr} L_1)}.
\eeq
Let us recall the notation from Section \ref{sec:DM-isom}.
Suppose that $(\C,g)=(C,s_1,f,g)$ is a point in
$I_\zeta X_{0,1,d}$  obtained via the isomorphism \eqref{stem_iso} from an
orbifold stable map $\C_0:=(C_0,s^0_1,\dots, s^0_{k+2},f_0)\in
[X/\mu_m]^{\eta,1,\dots,1,\eta^{-1}}_{0,k+2,d_0} $ and a
collection of $k+1$ stable maps $(\C_i,g_i)=(C_i,s^i_1,f_i,g_i)\in
I_{\eta_i}X_{0,1,d_i}$ ($1\leq i\leq k+1$). In other words, the curve
$C$ decomposes as in \eqref{sl-structure} and the automorphism $g$ has
the form \eqref{g0}--\eqref{gi}. Let
$\widetilde{\C}_0:=
(\widetilde{C}_0,
s_1,
s_{i,a}(1\leq i\leq k,1\leq a\leq m),
s_{k+2},
\widetilde{f}_0)$
be the stable map, such that, $\C_0=[\widetilde{\C}_0/\mu_m]$. 
Our first goal, is to express the eigensubspaces $T(\C,g)_\lambda$ in
terms of the eigensubspaces $T(\widetilde{\C}_0,g_0)_{\lambda_0}$ and
$T(\C_i,g_i)_{\lambda_i}$ ($1\leq i\leq k$). 
\begin{lemma}\label{le:splitting}
Let $\C=(C,s_1,\dots,s_n,f)$ be a stable map obtained from gluing two
stable maps
\ben
\C'=(C', s_1,\dots,s_{n_1}, z',f_1) \quad \mbox{and}\quad 
\C''=(C'',s_{n_1+1},\dots, s_{n_1+n_2}, z'',f_2),
\een
that is, we have
$n_1+n_2=n$,  $f_1(z')=f_2(z'')$, and the points $z'$ and $z''$ are
identified yielding a node $z$ of $C$.
Then the virtual tangent space decomposes as follows:
\ben
T(\C)=T(\C')+T(\C'') + T_{z'} C'\otimes T_{z''} C''-  T_{f(z)} X.
\een
\end{lemma}
\proof
Locally near the node $z$ we have
$\O_{C,z}=\CC\{x,y\}/(xy)$, where $\O_C$ is the structure sheaf of $C$
and $x$ (resp. $y$) is a holomorphic coordinate on $C'$ (resp. $C''$)
near $z'$ (resp. $z''$). By definition, the stalk $\T_{C,z} $ of the tangent sheaf
at $z$ is given by the $\O_{C,z}$-module of derivations
$\operatorname{Der}(\O_{C,z},\O_{C,z})$. If $u:\O_{C,z}\to \O_{C,z}$
is a derivation, then put $u_1:=u(x)$, $u_2=u(y)$. We have
$0=u(xy)=yu_1+xu_2 $ in $\O_{C,z}$ $\Rightarrow$
$
u_1\in x\, \O_{C,z} = x\CC\{x\}
$
and
$
u_2\in y\, \O_{C,z} = y \CC\{y\}
$
$\Rightarrow$ $u\in x\CC\{x\} \partial_x + y\CC\{y\}
\partial_y$. Therefore,
\ben
\T_{C,z}= x\CC\{x\} \, \partial_x + y\CC\{y\} \, \partial_y.
\een
Let $\iota':C'\to C$ and $\iota'':C''\to C$ be the natural
inclusions. Then the above formula shows that
\beq\label{def:cs}
\T_{C}(-s_1-\cdots-s_n) =
\iota'_*\T_{C'} (-s_1-\cdots -s_{n_1}-z') + 
\iota''_*\T_{C''} (-s_{n_1+1}-\cdots -s_{n_1+n_2}-z'') 
\eeq
By comparing stalks, we can prove that the following short exact
sequence of sheaves on $C$ is exact
\beq\label{def:map}
\xymatrix{0\ar[r] &
  f^*\T_X \ar[r] &
  \iota'_* (f_1^*\T_X) \bigoplus
  \iota''_* (f_2^*\T_X) \ar[r] &
  \O_z \otimes T_{f(z)}X\ar[r] &
  0},
\eeq
where $\O_{z}$ is the structure sheaf of the point $z$. Using the long
exact cohomology sequence of \eqref{def:map} and \eqref{def:cs} we get
the formula stated in the lemma.

\subsection{Splitting into stems, legs, and tails}
Let us consider first the case when $d_{k+1}= 0$, that is, the butt
$s_{k+2}$ of $\widetilde{C}_0$ is a regular point. 
Using Lemma \ref{le:splitting} we get the following formula for the
virtual tangent space:
\beq\label{tangent-sl}
T(\C)=T(\widetilde{\C}_0)  - T_{s_{k+2}} \widetilde{C}_0+
\bigoplus_{i=1}^k \bigoplus_{a=1}^m T(\C_i\times\{a\}) + 
\bigoplus_{i=1}^{k} \bigoplus_{a=1}^m \Big(
T_{s_{i,a}}\widetilde{C}_0 \otimes T_{s^i_1} C_i - T_{\widetilde{f}_0(s_{i,a})}
X\Big),
\eeq
where the 2nd term on the RHS corresponds to forgetting that the butt $s_{k+2}$ is a
marked point. The RHS of \eqref{tangent-sl} splits naturall into five
types of subspaces. Each type corresponds to the fiber of a certain
virtual vector bundle on the fiber product in \eqref{stem_iso}. We
would like to work out the contributions of each of these five type of
virtual vector bundles to the total Todd class
$\widetilde{\operatorname{td}} (\T_{0,1,d})$. Our computation splits
naturally into five cases. 

{\em Case 1:} Contribution from $T(\widetilde{\C}_0)$. Note that
$\widetilde{C}_0$ is a $g$-invariant curve and that the restriction of
$g$ to $\widetilde{C}_0$ is $g_0$. Let us  denote by
$T(\widetilde{\C}_0)_\lambda$ the eigensubspace of $g_0$ with
eigenvalue $\lambda$. Note that $\lambda =\zeta^{-k}$ ($0\leq k\leq m-1$)
must be a $m$th root of 1.  

Let $Y=[X/\mu_m]$ and let us consider the commutative diagram
involving the universal curve
\ben
\xymatrix{
C_0\ar[r]^-{J} \ar[d]_\pi& 
Y^{\eta,1,\dots,1,\eta^{-1},1}_{0,k+3,d_0} 
\ar[r]^-{\operatorname{ev}} \ar[d]^\pi & Y \\
\operatorname{Spec}(\CC) \ar[r]^j & 
Y^{\eta,1,\dots,1,\eta^{-1}}_{0,k+2,d_0} & }
\een
where $j$ is the inclusion corresponding to the orbifold stable map
$\C_0$. Let us denote by 
\ben
T^0_{0,k+2,d_0}:= 
\pi_* \operatorname{ev}^* (T_X) -\pi_* L^{-1} - (\pi_* \iota_* \O_Z)^\vee
\een
the virtual tangent bundle  on
$Y^{\eta,1,\dots,1,\eta^{-1}}_{0,k+2,d_0}$, where $\pi_*=R^0\pi_*
-R^1\pi_*$ is the K-theoretic pushforward. Slightly abusing the
notation we view $T^0_{0,k+2,d_0}$ as an element in the topological K-ring
of $Y^{\eta,1,\dots,1,\eta^{-1}}_{0,k+2,d_0}$. One can check that
the eigensubspace $T(\widetilde{\C}_0)_1$  coincides with 
$j^*T^0_{0,k+2,d_0}$ and that the eigensubspace $T(\widetilde{\C}_0)_{\zeta^{-k}}$
coincides with 
\ben
j^*\left( 
\pi_*\operatorname{ev}^* (\operatorname{T_X}\otimes \CC_{\zeta^k} ) + 
\pi_*\left( 
(1-L^{-1})\otimes \operatorname{ev}^*
  \CC_{\zeta^k} \right)-
\left( 
\pi_* \left( 
\iota_* \O_{Z_1} \otimes \operatorname{ev}^*\CC_{\zeta^k} 
\right)\right)^\vee \right).
\een
Comparing with formulas \eqref{thetaA}--\eqref{thetaB} we get that the
contribution to $\widetilde{\operatorname{td}}(T_{0,1,d})$ is given
by the cohomology class $\Theta^{ABC}\in
H^*(Y_{0,k+2,d_0}^{\eta,1,\dots,1,\eta^{-1}},\CC)$ used in the
definition of the stem invariants (see Section \ref{sec:stem_inv}). 

{\em Case 2:} Contribution from $T_{s_{k+2}}\widetilde{C}_0$. Note
that this is a one dimensional vector space on which $g$ acts with
eigenvalue $\zeta$. The cotangent line $T^*_{s^0_{k+2}}C_0$ is by
definition the fiber of the orbifold line bundle $L_{k+2}$ on
$Y_{0,k+2,d_0}^{\eta,1,\dots,1,\eta^{-1}}$. Therefore, we can identify
$T_{s_{k+2}}\widetilde{C}_0$ with the fiber of an orbifold line bundle
$L_{k+2}^{-1/m}$, that is, an $m$-th root of $L_{k+2}^{-1}$. The
contribution to $\widetilde{\operatorname{td}}(T_{0,1,d})$ is given
by 
\beq\label{contr_2}
\operatorname{td}_{\zeta^{-1}} (- L_{k+2}^{-1/m} ) =
1-\zeta^{-1} e^{\psi_{k+2}/m}\quad \in \quad 
H^*(Y_{0,k+2,d_0}^{\eta,1,\dots,1,\eta^{-1}},\CC).
\eeq
where $\psi_{k+2}=c_1(L_{k+2})$. 

{\em Case 3:} Contribution from 
$\bigoplus_{a=1}^m T(\C_i\times \{a\} )$, $1\leq i\leq k$. To begin
with, note that we have the following identification of
eigensubspaces: 
\ben
\Big( 
\bigoplus_{a=1}^m T(\C_i\times \{a\} ), g\Big)_\lambda\cong 
T(\C_i,g_i)_{\lambda^m}. 
\een
Indeed, suppose that $(v_1,\dots, v_m)$ is an element of the vector
space on the LHS of the above isomorphism. Recalling the definition of
$g_i$, we get  
\ben
\lambda (v_1,\dots, v_m) =g (v_1,\dots, v_m) = (g_i(v_m),v_1,\dots, v_{m-1}).
\een
Comparing the components, we get 
\ben
v_{m-i}= \lambda^i v_m \quad (1\leq i\leq m-1) 
\een
and $g_i(v_m)=\lambda^m v_m$. Therefore, the isomorphism is given by
$(v_1,\dots, v_m)\mapsto v_m$. 

The vector space $T(\C_i,g_i)_\xi$ is the fiber of the virtual vector
bundle $I_\xi T_{0,1,d_i}$ on $I_{\eta_i} X_{0,1,d_i}$ introduced in
Section \ref{sec:vts}.  The contribution to
$\widetilde{\operatorname{td}}(T_{0,1,d})$ is given by 
\ben
\prod_{
\mbox{$x$:  Chern roots of $I_1 T_{0,1,d_i}$} } 
\frac{x}{1-e^{-x} }\quad 
\prod_{\substack{\lambda\in \CC^* \\ \lambda\neq 1}}
\prod_{
\mbox{$x$:  Chern roots of $I_{\lambda^m} T_{0,1,d_i}$} } 
\frac{1}{1-\lambda^{-1} e^{-x} }.
\een
Let us rewrite the product over $\lambda$ in the following way. Let us
take the terms for which $\lambda^m=1$, that is, $\lambda=\zeta^{-k}$
for $1\leq k\leq m-1$ and combine them with the terms of the first
product. Then the terms in the first product will change to
$\frac{x}{1-e^{-mx}}$. The remaining part of the product over
$\lambda$ can be parametrized by $\lambda=\xi^{1/m} \zeta^{-k}$, where
$\xi\neq 1$ and $1\leq k\leq m$. The above formula takes the form 
\ben
\prod_{
\mbox{$x$:  Chern roots of $I_1 T_{0,1,d_i}$} } 
\frac{x}{1-e^{-m x} }\quad 
\prod_{\substack{\xi\in \CC^* \\ \xi\neq 1}}
\prod_{
\mbox{$x$:  Chern roots of $I_{\xi} T_{0,1,d_i}$} } 
\frac{1}{1-\xi^{-1} e^{-mx} }.
\een 
The above formula coincides with 
\beq\label{contr_3}
m^{-\operatorname{rk}(I_1T_{0,1,d_i} )}  
\widetilde{\operatorname{td}}(\Psi^m T_{0,1,d_i})\quad \in \quad 
H^*(I_{\eta_i} X_{0,1,d_i}, \CC), 
\eeq
where $\Psi^m$ is the $m$th Adam's operation. 
 
{\em Case 4:} Contribution from 
$\bigoplus_{a=1}^m 
T_{s_{i,a}}\widetilde{C}_0 \otimes T_{s^i_1}C_i$,
 $1\leq i\leq k$. To begin 
with, note that we have the following identification of
eigensubspaces: 
\ben
\Big(
\bigoplus_{a=1}^m 
T_{s_{i,a}}\widetilde{C}_0 \otimes T_{s^i_1}C_i, g\Big)_\lambda \cong 
T_{s_{i+1}} C_0\otimes T_{s^i_1} C_i.
\een
Indeed, suppose that $(v_1\otimes w_1,\dots, v_m\otimes w_m)$ is an
eigenvector, i.e., an element of the LHS. Recalling again the
definition of $g_i$ we get 
\ben
\lambda (v_1\otimes w_1,\dots, v_m\otimes w_m) = 
g (v_1\otimes w_1,\dots, v_m\otimes w_m) = 
(v_m\otimes g_i(w_m),v_1\otimes w_1,\dots, v_{m-1}\otimes w_{m-1}). 
\een
By definition, $g_i(w_m)= \eta_i^{-1} w_m$. Comparing the components
in the above equality, we get
\ben
v_{i}\otimes w_{i} = \lambda^{m-i} v_m\otimes w_m\quad (1\leq i\leq m-1)
\een
and $\lambda^m = \eta_i^{-1}$. The desired isomorphism is given by 
$(v_1\otimes w_1,\dots, v_m\otimes w_m)\mapsto v_m\otimes
w_m$. Moreover, we proved that all possible values for the eigenvalue
are given by  $\lambda=\eta_i^{-1/m} \zeta^k$ ($1\leq k\leq m$). 

The vector spaces $T_{s_{i+1}} C_0$ and $T_{s^i_1} C_i$ are
respectively the fibers of the orbifold vector bundles $L_{i+1}^{(0)}:=L_{i+1}$ on
$[X/\mu_m]^{\eta,1,\dots,1,\eta^{-1}}_{0,k+2,d_0}$ and $L_1^{(i)}:= L_1$ on
$I_{\eta_i} X_{0,1,d_i}$.  The contribution to
$\widetilde{\operatorname{td}}(T_{0,1,d})$ is given by 
\beq\label{contr_4}
\prod_{k=1}^m 
\frac{1}{1-\eta_i^{1/m}\zeta^{-k} 
e^{\psi_{i+1}^{(0)} + \psi_1^{(i)} } }=
\frac{1}{1-\eta_i e^{m (\psi_{i+1}^{(0)} + \psi_1^{(i)}) } }= 
\Psi^m\Big(
\frac{1}{1-\eta_i \operatorname{ch}(L_{i+1}^{(0)}\otimes L_1^{(i)}) }
\Big) ,
\eeq
where 
$\psi_{i+1}^{(0)}=c_1(L_{i+1}^{(0)})$ and 
$\psi_1^{(i)}= c_1(L_1^{(i)}).$ Here, using the Chern character map,
we extend the Adam's operations to cohomology, that is, if $\phi$ is a
homogeneous cohomology class in $H^{2k}(M ;\CC)$ for some topological
space $M$, then $\Psi^m(\phi) :=m^k \phi$.  

{\em Case 5:} Contributions from 
$\bigoplus_{a=1}^m T_{\widetilde{f}_0(s_{i,a})}X$
 ($1\leq i\leq k$). 
Note that $g$ acts on the  direct sum by cyclically permuting the
summands. Therefore, the eigensubspaces 
\ben
\Big( 
\bigoplus_{a=1}^m T_{\widetilde{f}_0(s_{i,a})} X , g\Big)_\lambda\cong
T_{f_0(s_{i+1}^0 )} X = T_{f_i(s_1^i)} X
\een 
and all possible values for $\lambda$ are $\lambda=\zeta^k$ ($1\leq
k\leq m$). The vector space $T_{f_i(s_1^i)} X$ is the fiber of the
vector bundle $(\operatorname{ev}_1^i)^* T_X$, where
$\operatorname{ev}_1^i: I_{\eta_i}X_{0,1,d_i}\to X$ is the evaluation
map. The contribution to $\widetilde{\operatorname{td}}(T_{0,1,d})$
is given by  
\beq\label{contr_5}
\prod_x \frac{1-e^{-x}} {x} \ 
\prod_{k=1}^{m-1} 
\prod_x (1-\zeta^{-k} e^{-x})= 
\prod_x \frac{1-e^{-m x}}{x} = 
m^{\operatorname{dim}(X)} \, 
(\operatorname{ev}_1^i)^* \, \Psi^m\, \frac{1}{\operatorname{td}(X)},
\eeq
where in the products over $x$, the variable $x$ varies over the set
of all Chern roots of $(\operatorname{ev}_1^i)^*(T_X)$. 

Combining the results from the above 5 cases, that is formulas
\eqref{contr_2}--\eqref{contr_5}, we get the
following formula: 
\begin{align}
\label{total_td}
\widetilde{\operatorname{td}}(T_{0,1,d}) = 
&
\iota^*\left(
m^{k
\operatorname{dim}(X)-\sum_{i=1}^k 
\operatorname{rk}(I_1T_{0,1,d_i}) }
\Big(\Theta^{ABC}_{0,k+2,d_0} \cup
\Big(1-\zeta^{-1} e^{\psi^{(0)}_{k+2} /m}\Big)
\boxtimes 1^{\boxtimes k}\Big)\cup 
\phantom{
\frac{1}{
\Big( 
L_{i+1}^{(0)}\Big)} }
\right.\\
\nonumber
&
\quad
\Big( 1\boxtimes 
\Psi^m \Big(
\widetilde{\operatorname{td}}(T_{0,1,d_1}) 
(\operatorname{ev}_1^1)^*
\tfrac{1}{\operatorname{td}(X)} \Big)
\boxtimes \cdots \boxtimes
\Psi^m \Big(
\widetilde{\operatorname{td}}(T_{0,1,d_k}) 
(\operatorname{ev}_1^k)^*
\tfrac{1}{\operatorname{td}(X)} \Big)
\Big) \cup \\
\nonumber 
&
\quad
\left.
\bigcup_{i=1}^k \frac{1}{
1-\eta_i \operatorname{ch}\circ \Psi^m\Big( 
L_{i+1}^{(0)}\otimes L_1^{(i)} \Big)}\right),
\end{align}
where $\iota$ is the natural inclusion of the fiber product into the
direct product. More precisely, we have the following pullback diagram
\beq\label{cd}
\xymatrix{
    [X/\mu_m]^{\eta,1,\dots,1,\eta^{-1}}_{0,k+2,d_0}\times_{X^k}
\Big(
    I_{\eta_1}X_{0,1,d_1}\times
    \cdots \times
    I_{\eta_k}X_{0,1,d_k} 
\Big) \ar[d]^-\iota \ar[r]^-{\operatorname{ev}} & X^k\ar[d] \\
[X/\mu_m]^{\eta,1,\dots,1,\eta^{-1}}_{0,k+2,d_0}\times \Big(
    I_{\eta_1}X_{0,1,d_1}\times
    \cdots \times
    I_{\eta_k}X_{0,1,d_k} \Big) \ar[r]^-{\operatorname{ev}} & X^k\times X^k},
\eeq
where the right vertical arrow is the diagonal embedding, the upper
horizontal arrow is the evaluation map at $s_1^i$ ($1\leq i\leq k$),
and the lower horizontal map is the evaluation map at $s^0_{i+1}$
($1\leq i\leq k$) and $s_1^i$ ($1\leq i\leq k$).

If $d_{k+1}\neq 0$, that is, the butt $s_{k+2}$ of $\widetilde{C}_0$
is a node, then we have to replace the term
$-T_{s_{k+2}}\widetilde{C}_0$ with
\ben
T(\mathcal{C}_{k+1})+T_{s_{k+2}}\widetilde{C}_0\otimes T_{s^{k+1}_1}
C_{k+1} - T_{\widetilde{f}_0(s_{k+2})} X.
\een
The contributions to the Todd class
$\widetilde{\operatorname{td}}(T_{0,1,d})$ 
of the above three terms can be worked out in the same way as in Cases
3, 4, and 5. We get 
\ben
\widetilde{\operatorname{td}}(T_{0,1,d_{k+1}}) \cup 
\frac{1}{1-\zeta^{-1} e^{\psi^{(0)}_{k+2}/m}\, \eta_{k+1}
  e^{\psi^{(k+1)}_1} }\cup (\operatorname{ev}_1^{k+1})^*\Big(
\tfrac{1}{\operatorname{td}(X)} \Big).
\een
Therefore, if $d_{k+1}\neq 0$, then the formula for the total Todd class
takes the form
\begin{align}
\label{total_td_tail}
\widetilde{\operatorname{td}}(T_{0,1,d}) = 
&
\iota^*\left(
m^{k
\operatorname{dim}(X)-\sum_{i=1}^k 
\operatorname{rk}(I_1T_{0,1,d_i}) }
\Big(\Theta^{ABC}_{0,k+2,d_0} \cup
\frac{1}{1-\zeta^{-1} e^{\psi^{(0)}_{k+2}/m}\, \eta_{k+1}
  e^{\psi^{(k+1)}_1} }
\boxtimes 1^{\boxtimes k}\Big)\cup 
\right.\\
\nonumber
&
\quad
\Big( 1\boxtimes 
\Psi^m \Big(
\widetilde{\operatorname{td}}(T_{0,1,d_1}) 
(\operatorname{ev}_1^1)^*
\tfrac{1}{\operatorname{td}(X)} \Big)
\boxtimes \cdots \boxtimes
\Psi^m \Big(
\widetilde{\operatorname{td}}(T_{0,1,d_k}) 
(\operatorname{ev}_1^k)^*
            \tfrac{1}{\operatorname{td}(X)} \Big)  
\Big) \cup \\
\nonumber 
&
\quad
\left.
\bigcup_{i=1}^k \frac{1}{
1-\eta_i \operatorname{ch}\circ \Psi^m\Big( 
             L_{i+1}^{(0)}\otimes L_1^{(i)} \Big)}
\cup
\widetilde{\operatorname{td}}(T_{0,1,d_{k+1}}) 
(\operatorname{ev}_1^{k+1})^*
            \tfrac{1}{\operatorname{td}(X)}    
             \right),
\end{align}

\subsection{Integration over the Kawasaki strata}
Using formulas \eqref{total_td} and \eqref{total_td_tail}, 
let us compute the integral \eqref{zeta_stratum_int}. Recalling the
isomorphism \eqref{stem_iso}, let us split our computation into two
parts depending on whether $d_{k+1}=0$ or $d_{k+1}\neq 0$. Suppose
first that $d_{k+1}=0$. The integral
over the fiber product can be rewritten as an integral over the
corresponding  direct product by using the Thom isomorphism, i.e., if
$\iota:X\to Y$ is a regular embedding, then  
\beq\label{Thom_iso}
\int_{[X]} \iota^* \Phi = \int_{ [Y]} \Phi\cup \theta(N_\iota), 
\eeq
where $\theta(N_{\iota})$ is the Thom class of the normal bundle
$N_\iota$ to $X$ in $Y$.  
In our case, having in mind the pullback diagram \eqref{cd}, the
normal bundle $N_\iota$ is the pullback of the normal 
bundle to the diagonal embedding $X^k\to X^k\times X^k$. Therefore,
$\theta(N_\iota)$ is the pullback of the Thom class of the normal
bundle to the diagonal, which is well known to be the Poincare dual of
the diagonal, that is, 
\beq\label{Thom_class}
\theta(N_\iota) = 
\sum_{a_1,\dots,a_k=1}^N\Big( 
(\operatorname{ev}^0_2)^*\phi_{a_1} \cup \cdots \cup 
(\operatorname{ev}^0_{k+1} )^*\phi_{a_k} \Big) \boxtimes 
(\operatorname{ev}^1_1)^*\phi^{a_1}  \boxtimes\cdots\boxtimes
(\operatorname{ev}^1_k)^*\phi^{a_k},
\eeq
where $\{\phi_a\}$ and $\{\phi^a\}$ are dual bases of $H^*(X;\CC)$
with respect to the Poincare pairing. For simplicity we will choose
$\phi_a$ to be homogeneous and denote its degree by $2|\phi_a|$. Note
that the corresponding dual vector $\phi^a$ must be homogeneous too
and if we denote its degree by $2|\phi^a|$, then
$|\phi_a|+|\phi^a|=\operatorname{dim}(X)$. 
The integral \eqref{zeta_stratum_int} takes the form 
\ben
\sum_{\eta_1,\dots,\eta_k}
\sum_{(d_0+d_1+\cdots + d_k)m=d} 
\int_{[X/\mu_m]_{0,k+2,d_0}^{\eta,1,\dots,1,\eta^{-1}} } 
\int_{I_{\eta_1} X_{0,1,d_1} } \cdots 
\int_{I_{\eta_k} X_{0,1,d_k} }  
\theta(N_\iota)\,
\widetilde{\operatorname{td}}(T_{0,1,d}) \, 
\frac{
(\operatorname{ev}_1^{(0)})^* \operatorname{ch}(\Phi^a)}{
1-q \zeta e^{\psi_1^{(0)}/m} },
\een
where the 1st sum is over all primitive roots of unity
$\eta_1,\dots, \eta_k$, such that, $\eta_i\neq 1$ for all $i$, and the 2nd sum is
over all degree classes $d_0,\dots,d_k$, such that,
$m(d_0+\cdots+d_k)=d$. Here we used that under the isomorphism
\eqref{stem_iso} the line bundle $L_1$ on $I_\zeta X_{0,1,d}$ is a
$m$th root of $L_1^{(0)}$ -- the line bundle $L_1$ on
$[X/\mu_m]_{0,k+2,d_0}^{\eta,1,\dots,1,\eta^{-1}}$. Let us substitute
formulas \eqref{total_td} and \eqref{Thom_class} in the above formula
and single out the terms in the integrand that involve cohomology
classes on $I_{\eta_i}X_{0,1,d_i}$ ($1\leq i\leq k$)
\beq\label{leg_i}
m^{-\operatorname{rk}(I_1 T_{0,1,d_i})-|\phi^{a_i}| }\, 
\Psi^m \Big(
\widetilde{\operatorname{td}}(T_{0,1,d_i}) 
(\operatorname{ev}_1^i)^*
\tfrac{\phi^{a_i} }{\operatorname{td}(X)} \Big)
\frac{1}{
1-\eta_i \operatorname{ch}\circ \Psi^m\Big( 
L_{i+1}^{(0)}\otimes L_1^{(i)} \Big)}.
\eeq 
The remaining terms of the integrand are given by 
\beq\label{stem}
m^{k \operatorname{dim}(X) }\, 
\Theta^{ABC}_{0,k+2,d_0} 
\frac{
(\operatorname{ev}_1^{(0)})^* \operatorname{ch}(\Phi^a)}{
1-q \zeta e^{\psi_1^{(0)}/m} }
(\operatorname{ev}^0_2)^*\phi_{a_1} \cup \cdots \cup 
(\operatorname{ev}^0_{k+1} )^*\phi_{a_k} 
\cup
\Big(1-\zeta^{-1} e^{\psi^{(0)}_{k+2} /m}\Big)
\eeq
Let us integrate \eqref{leg_i} along the virtual fundamental cycle of
$I_{\eta_i}X_{0,1,d_i}$. First, note that the cohomology classes
$\tfrac{\phi^{a}}{\operatorname{td}(X)} =
\operatorname{ch}(\Phi^{a})$ form a basis dual to
$\phi_a=\operatorname{ch}(\Phi_a)$ with respect to the K-theoretic
Euler pairing. Furthermore, note that the rank of the inertia tangent
bundle $I_1T_{0,1,d_i}$ 
coincides with the dimension of the virtual fundamental cycle of
$I_{\eta_i}X_{0,1,d_i}$. On the other hand, the Adam's operation $\Psi^m$ in
\eqref{leg_i} rescales each cohomology class by a power of $m$ equal
to its degree. Since, the cohomology classes on $
I_{\eta_i}X_{0,1,d_i}$  that contribute
non-trivially to the integral should have total degree matching the
degree of the virtual fundamental cycle, we get that the Adam's operation cancels out
with $m^{-\operatorname{rk}(I_1T_{0,1,d_i})}$, except for the part of
the Adam's operation acting on cohomology classes not supported on
$I_{\eta_i}X_{0,1,d_i}$, such as,
$\operatorname{ch}(L^{(0)}_{i+1})$. After these remarks, it is clear
that the integral of \eqref{leg_i} along the virtual fundamental cycle
of $I_{\eta_i}X_{0,1,d_i}$ is precisely 
$
m^{-|\phi^{a_i}|}\, 
\tau_{\eta_i, d_i,  a_i}(e^{m\psi^{(0)}_{i+1}})
$, that is, we get an integral of the type  \eqref{zeta_stratum_int}.  
Therefore, in order to compute the integral \eqref{zeta_stratum_int}
we have to multiply $\prod_{i=1}^k m^{-|\phi^{a_i}|}\, 
\tau_{\eta_i, d_i,  a_i}(e^{m\psi^{(0)}_{i+1}})$ with \eqref{stem} and
integrate over the virtual fundamental cycle of
$[X/\mu_m]_{0,k+2,d_0}^{\eta,1,\dots,1,\eta^{-1}}$. Using also that 
$m^{\operatorname{dim}(X)-|\phi^a|} \phi_a = m^{|\phi_a|} \phi_a=
\Psi^m \phi_a$, we get 
\ben
\tau_{\zeta,d,a}(q) = \sum
\left[
\frac{\operatorname{ch}(\Phi^a)}{
1-q \zeta e^{\psi_1/m} }, 
\tau_{\eta_1, d_1,  a_1}(e^{m\psi_2})
\Psi^m \phi_{a_1}, \cdots, 
 \tau_{\eta_k, d_k,  a_k}(e^{m\psi_{k+1}})
\Psi^m \phi_{a_k},
1-\zeta^{-1} e^{\psi_{k+2} /m} 
\right]_{0,k+2,d_0},
\een
where the sum is over all non-trivial primitive roots of unity
$\eta_1,\dots,\eta_k$, all effective degree classes
$d_0,d_1,\dots,d_k$, such that, $m(d_0+\cdots+d_k)=d$ (just like above),  and
over all $1\leq a_1,\dots,a_k\leq N$.

In the second case, when $d_{k+1}\neq 0$, the differences are the
following. Formula \eqref{Thom_class} takes the form
\beq\label{Thom_class_tail}
\theta(N_\iota) = 
\sum_{a_1,\dots,a_{k+1}=1}^N\Big( 
(\operatorname{ev}^0_2)^*\phi_{a_1} \cup \cdots \cup 
(\operatorname{ev}^0_{k+2} )^*\phi_{a_{k+1}} \Big) \boxtimes 
(\operatorname{ev}^1_1)^*\phi^{a_1}  \boxtimes\cdots\boxtimes
(\operatorname{ev}^1_{k+1})^*\phi^{a_{k+1}},
\eeq
The contributions \eqref{leg_i} remain the same but in addition we
have  the following term involving cohomology classes on
$I_{\eta_{k+1}}X_{0,1,d_{k+1}}$:
\ben
\widetilde{\operatorname{td}}(T_{0,1,d_{k+1}}) \cup
\frac{
  (\operatorname{ev}_1^{k+1})^* \operatorname{ch}(\Phi^{a_{k+1}} )
}{
  1-\zeta^{-1} e^{\psi^{(0)}_{k+2}/m} \eta_{k+1} e^{\psi_1^{(k+1)}} }
\een
The integral of the above class over the virtual fundamental cycle of
$I_{\eta_{k+1}}X_{0,1,d_{k+1}}$ is precisely
$\tau_{\eta_{k+1},d_{k+1},a_{k+1}} (\zeta^{-1}
e^{\psi^{(0)}_{k+2}/m})$.
Therefore, the contribution \eqref{stem} should be modified to
\ben
m^{k \operatorname{dim}(X) }\, 
\Theta^{ABC}_{0,k+2,d_0} 
\frac{
(\operatorname{ev}_1^{(0)})^* \operatorname{ch}(\Phi^a)}{
1-q \zeta e^{\psi_1^{(0)}/m} }
(\operatorname{ev}^0_2)^*\phi_{a_1} \cup \cdots \cup 
(\operatorname{ev}^0_{k+1} )^*\phi_{a_k} 
\cup
(\operatorname{ev}^0_{k+2} )^*\phi_{a_{k+1}}
\tau_{\eta_{k+1},d_{k+1}, a_{k+1}} (\zeta^{-1}
e^{\psi^{(0)}_{k+2}/m}).
\een
Finally, in order to complete the proof of \eqref{GT-recursion}, it
remains only to note that 
\ben
\tau_\zeta(q,Q) = 
\sum_{a=1}^N 
\sum_{d\in \operatorname{Eff}(X)} 
\tau_{\zeta,d,a} (q) Q^d \Phi_a,
\een
\ben
\delta\tau_\zeta(q,Q)= 1-q +
\sum_{a_{k+1}=1}^N
\sum_{\eta_k: \eta_{k+1}\neq \zeta}
\sum_{d_{k+1}}
\tau_{\eta_{k+1},d_{k+1},a_{k+1}}(q) Q^{d_{k+1}} \phi_{a_{k+1}},
\een
and 
\ben
\sum_{\eta_i}\sum_{ d_i } \sum_{a_i=1}^N 
\tau_{\eta_i, d_i,  a_i}(e^{m\psi_{i+1}}) Q^{m d_i} \Psi^m \phi_{a_i} = 
\tau^{(m)}(\psi_{i+1}, Q),
\een
where $\tau^{(m)}(z,Q)$ is the same as in \eqref{leg_contr}.

\section{A proof of Theorem
  \ref{thm:gamma_cohomology_oscillatory_i_function} for Picard rank 2} 
\label{sec:6_IT}

The strategy of our proof is different from the one originally used in
\cite{Iritani:gamma_structure}. It relies on the inversion formula for
the Mellin transform and the formula for the Poincaré pairing for toric
cohomology in terms of Jeffrey--Kirwan residues (c.f. Theorem \ref{thm:gamma_computation_of_intersection_products}).
Similar use of the inversion formula for the Mellin transform appears
in the master's thesis
\cite{Xia:gamma_structure_blow_up} to study the Gamma integral
structure of the blowup of $\mathbb{P}^N$ at a point. 

\subsection{Strategy of the proof}
Let $X_{\mu,K}$ be a Fano symplectic toric manifold of Picard rank 2. According to Proposition
\ref{prop:gamma_model_toric_picard_rank_two}, we may assume that the
matrix of the moment map is 
	\[
\operatorname{Mat}(\mu)=
		\begin{pmatrix}
			1
			&
			\cdots
			&
			1
			&
			0
			&
			-a_1
			&
			\cdots
			&
			-a_k
			\\
			0
			&
			\cdots
			&
			0
			&
			1
			&
			1
			&
			\cdots
			&
			1
		\end{pmatrix}.
	\]
where $a_1,\dots,a_k \in \mathbb{Z}_{\geq 0}$. It is convenient to
define $a_0:=0$. We have $c_1(TX_{\mu,k})=m_1 p_1 + m_2 p_2$, where
$m_i=\sum_{j=1}^{N+k+1} m_{ij}$ is the sum of the entries in the $i$th
row of $\operatorname{Mat}(\mu)$, that is, $m_1=N-a_1-\cdots-a_k$ and
$m_2=k+1$.

We will make use of the inversion theorem for {\em Mellin transform}
of a smooth function $f:\mathbb{R}_{>0}^2\to \CC$. Recall that the Mellin
transform of $f$ is defined by 
\[
\mathcal{M}f (p_1,p_2):= \int_{\mathbb{R}_{>0}^2}
f(q_1,q_2) q_1^{p_1} q_2^{p_2} \frac{dq_1}{q_1} \wedge \frac{dq_2}{q_2}.
\]
The Mellin transform $\mathcal{M}f$ is a holomorphic function for all inputs $(p_1,p_2)$
for which the integral is absolutely convergent. 
Let $E= \left\{(x,y) \in \mathbb{C}^2 | x_{\textnormal{min}} < \textnormal{Re}(x) < x_{\textnormal{max}} \textnormal{ and } y_{\textnormal{min}} < \textnormal{Re}(y) < y_{\textnormal{max}} \right\}\subseteq \mathbb{C}^2$ be the product of two strips.
Assume that the Mellin transform $\mathcal{M}f$ is holomorphic on the strip $E$.
Let $F(p_1,p_2):=\mathcal{M}f(p_1,p_2)$, then the inverse Mellin transform of $F$ is given by the formula below (\cite{Srivastava_Panda:Multidimensional_integral_transforms}, Lemma 2 p.125; see also for the one dimensional case \cite{Misra_Lavoine:Transform_analysis}, Lemma 11.10.1 p.246):
\[
f(q_1,q_2) = \mathcal{M}^{-1} F(q_1,q_2):=
\lim_{T_1\to +\infty}\lim_{T_2\to +\infty} 
\frac{1}{(2\pi\mathbf{i})^2} 
\int_{\epsilon_1-\mathbf{i} T_1}^{\epsilon_1+\mathbf{i} T_1} dp_1
\int_{\epsilon_2-\mathbf{i} T_2}^{\epsilon_1+\mathbf{i} T_2} dp_2
F(p_1,p_2) q_1^{-p_1} q_2^{-p_2} ,
\] 
where $(\epsilon_1,\epsilon_2)$ is a real point of $E$.

We begin by announcing the results of two computations --  Lemmas
\ref{lemma:gamma_coh_mellin_transform_computation} and
\ref{lemma:gamma_coh_inverse_mellin_residue}, then give our proof of
the theorem. The proofs of both lemmas will follow after. 
\begin{lemma}\label{lemma:gamma_coh_mellin_transform_computation}	
	Let $X_{\mu,K}$ be a symplectic toric manifold of Picard rank 2 as in Proposition \ref{prop:gamma_model_toric_picard_rank_two}, let
	$
		\Gamma_\mathbb{R}
		=\{
		(x_1,\dots,x_n) \in \pi^{-1}(Q_1,\dots,Q_r)
		\, | \,
		x_j > 0
		\}
	$ be the real Lefschetz thimble and $
		\mathcal{I}^{\textnormal{coh}}(z,Q)
		=
		\int_{\Gamma_\mathbb{R}}
		e^{-W_Q/z}
		\omega_{\pi^{-1}(Q)}
	$ be the oscillatory integral.
	Then, the Mellin transform $ \mathcal{MI}^{\textnormal{coh}} $
        of the oscillatory integral exists and it is given by
	\[
		\mathcal{MI}^{\textnormal{coh}}(z,p_1,p_2)
		=
		z^{\rho}\Gamma(p_1)^N \prod_{j=0}^k \Gamma(p_2 - a_j p_1).
	\]
\end{lemma}
For a proof of this lemma, see Section \ref{sssection:proof_lemma_mellin}.
Note that the oscillatory integral has the following symmetry:
\begin{align}
\label{oi-sym}
\mathcal{I}(z,Q_1,Q_2) = \mathcal{I}(1, z^{-m_1}Q_1, z^{-m_2}Q_2).
\end{align}
Indeed, if we change the variables in the oscillatory integral via
$x_i= z y_i$, then the Batyrev constraints take the form
$z^{-m_i}Q_i=\prod_{j=1}^{N+1+k} y_j^{m_{ij}}$ and the above identity
follows.  Furthermore, the I-function has a similar symmetry 
\begin{align}
\label{I-sym}
I^{\textnormal{coh}}(z,Q_1,Q_2) = z^{-\operatorname{deg}}\,
z^{-\rho}\,  
I^{\textnormal{coh}}(1,z^{-m_1}Q_1,z^{-m_2}Q_2) 
\end{align}
which can be checked easily. Let us point out that in the Fano case 
$-z e^{\sum p_i\log Q_i/z} \, I^{\textnormal{coh}}(z,Q_1,Q_2) $
coincides with the $J$-function, so the above symmetry follows also
from  \eqref{eqn:intro_cohom_J_recover_z}. Using the symmetries
\eqref{oi-sym} and \eqref{I-sym} we get that it is sufficient to prove
Theorem \ref{thm:gamma_cohomology_oscillatory_i_function} for $z=1$.
\begin{lemma}\label{lemma:gamma_coh_inverse_mellin_residue}
	Consider the Mellin transform
	\[
g(p_1,p_2):=
		\mathcal{MI}^{\textnormal{coh}}(1,p_1,p_2)
		=
		\Gamma(p_1)^N \prod_{j=0}^k \Gamma(p_2 - a_j p_1).
	\]
	Then, its inverse Mellin transform $\mathcal{M}^{-1}g(Q_1,Q_2)$ can be computed by the following residue formula:
	\[
		\mathcal{M}^{-1} g(Q_1,Q_2)
		=
		\textnormal{Res}_{x_2=0} \textnormal{Res}_{x_1=0}
		\omega(x_1,x_2) dx_1 dx_2,
	\]
	where
	\[
		\omega(x_1,x_2)
		:=
		\sum_{d_1,d_2 \geq 0}
		Q_1^{-x_1+d_1} Q_2^{-x_2+d_2}
		\left(
		\Gamma(x_1)
		\frac{
			\prod_{r=-\infty}^0 (x_1-r)
		}
		{
			\prod_{r=-\infty}^{d_1} (x_1-r)
		}
		\right)^N
		\prod_{j=0}^k
		\Gamma \left(x_2-a_j x_1\right)
		\frac{
			\prod_{r=-\infty}^0 (x_2 - a_j x_1 - r)
		}{
			\prod_{r=-\infty}^{d_2-a_j d_1} (x_2-a_j x_1 - r)
		}
	\]
\end{lemma}
The proof of Lemma \ref{lemma:gamma_coh_inverse_mellin_residue} will
be given in Section \ref{sssection:proof_lemma_inverse_mellin}.


\begin{proof}[Proof of Theorem \ref{thm:gamma_cohomology_oscillatory_i_function}]
	Consider the oscillatory integral for the real Lefschetz thimble $\Gamma_\mathbb{R}$ evaluated at $z=1$, $\mathcal{I}^\textnormal{coh}(1,Q_1,Q_2)$.
	The Mellin transform of this oscillatory integral is computed in Lemma \ref{lemma:gamma_coh_mellin_transform_computation}, in which we obtain
	\[
		\mathcal{MI}^{\textnormal{coh}}(1,p_1,p_2)
		=
		\Gamma(p_1)^N \prod_{j=0}^k \Gamma(p_2 - a_j p_1)
	\]
	Using Lemma \ref{lemma:gamma_coh_inverse_mellin_residue}, we compute the inverse Mellin transform of this expression to be
	\[
		\mathcal{M}^{-1}\mathcal{MI}^{\textnormal{coh}}(1,Q_1,Q_2)
		=
		\textnormal{Res}_{x_2=0} \textnormal{Res}_{x_1=0}
		\omega(x_1,x_2)dx_1 dx_2,
	\]
	where
	\[
		\omega(x_1,x_2)
		= \!
		\sum_{d_1,d_2 \geq 0}
		Q_1^{-x_1+d_1} Q_2^{-x_2+d_2}
		\left(
		\Gamma(x_1)
		\frac{
			\prod_{r=-\infty}^0 (x_1-r)
		}
		{
			\prod_{r=-\infty}^{d_1} (x_1-r)
		}
		\right)^N
		\prod_{j=0}^k
		\Gamma \left(x_2-a_j x_1\right)
		\frac{
			\prod_{r=-\infty}^0 (x_2 - a_j x_1 - r)
		}{
			\prod_{r=-\infty}^{d_2-a_j d_1} (x_2-a_j x_1 - r)
		}.
	\]
	Using Theorem \ref{thm:gamma_computation_of_intersection_products}, we identify the iterated residues above with the intersection product
	\[
		\mathcal{M}^{-1}\mathcal{MI}^{\textnormal{coh}}(1,Q_1,Q_2)
		=
		\int_{[X_{\mu,K}]} x_1^N \prod_{j=0}^k (x_2 - a_j x_1) \omega(x_1,x_2)
	\]
	Finally,
	\begin{align*}
		x_1^N \prod_{j=0}^k (x_2 - a_j x_1) \omega(x_1,&x_2)
		=
		\left[
			x_1^N \Gamma(x_1)^N \prod_{j=0}^k (x_2 - a_j x_1) \Gamma(x_2 - a_j x_1)
		\right]
		\times
		\\
		&\times
		\left[
			Q_1^{-x_1} Q_2^{-x_2}
			\sum_{d_1,d_2 \geq 0}
			Q_1^{d_1} Q_2^{d_2}
			\left(
				\frac{
					\prod_{r=-\infty}^0 (x_1-r)
				}
				{
					\prod_{r=-\infty}^{d_1} (x_1-r)
				}
			\right)^N
			\prod_{j=0}^k
				\frac{
					\prod_{r=-\infty}^0 (x_2 - a_j x_1 - r)
				}{
					\prod_{r=-\infty}^{d_2-a_j d_1} (x_2-a_j x_1 - r)
				}
		\right].
	\end{align*}
	In the right hand side, the factor in the first line is the gamma class $\widehat{\Gamma}\left(TX_{\mu,K} \right)$, and the factor in the second line is the small $I$-function of the toric manifold $I^\textnormal{coh}(1,Q_1,Q_2)$.
	Using Equation \ref{eqn:intro_cohom_J_recover_z}, we recover the result for all $z$.
\end{proof}

\begin{remark}
	In the language of \cite{Szenes_Vergne:jeffrey_kirwan_residues}, consider the projective sequence $\mathcal{A} = (\alpha_1, \dots, \alpha_n)$. The choice of integrating first with respect to the input $p_1$ amounts to choosing in Theorem 2.6 a vector $\xi \in K$ regular with respect to $\Sigma \mathcal{A}$ that is located below the line $\textnormal{Vect}(c_1(TX_{\mu,K}))$.
	The other choice replaces the residue $\textnormal{Res}_{p_2=0} \textnormal{Res}_{p_1=0}$ with the sum $\sum_j \textnormal{Res}_{p_1=0} \textnormal{Res}_{p_2=a_j p_1}$.
\end{remark}

\subsection{Proof of Lemma
  \ref{lemma:gamma_coh_mellin_transform_computation}}\label{sssection:proof_lemma_mellin} 

The computation relies on applying the Fubini theorem to see the Mellin transform $\mathcal{MI}^{\textnormal{coh}}$ as an integral on the space $Y_\mathbb{R} := ( \mathbb{R}_{>0})^{N+k+1}$.


\begin{proof}
	We have
	\[
		\mathcal{MI}^{\textnormal{coh}}(z,p_1,p_2)
		=
		\int_{(\mathbb{R}_{>0})^2} \mathcal{I}^{\textnormal{coh}}(z,Q_1,Q_2)Q_1^{p_1} Q_2^{p_2} \frac{dQ_1}{Q_1}\frac{dQ_2}{Q_2}
	\]
	We recall the diagram of the Landau--Ginzburg model below.
	\[
		\begin{tikzcd}
			Y := \left( \mathbb{C}^* \right)^{n=N+k+1}
				\arrow[r,"W"]
				\arrow[d,"\pi"]
			&
			\mathbb{C}
			\\
			B := \left( \mathbb{C}^* \right)^{r=2}
		\end{tikzcd}
	\]
	Notice that in the oscillatory integral $\mathcal{I}^{\textnormal{coh}}$, we have a first integral along a fibre $\pi^{-1}(Q_1,Q_2)$, while the Mellin transform introduces an integral over the base, i.e. all $(Q_1,Q_2) \in B_\mathbb{R} := ( \mathbb{R}_{>0})^2$, for which  $\frac{dQ_1}{Q_1}\frac{dQ_2}{Q_2}$ is a volume form.
	Using Fubini theorem, we can write
	\[
		\mathcal{MI}^{\textnormal{coh}}(z,p_1,p_2)
		=
		\int_{Y_\mathbb{R}= ( \mathbb{R}_{>0})^n}
		e^{-W(x_1, \dots, x_{N+k+1})/z}
		Q_1^{p_1} Q_2^{p_2}
		\frac{dx_1}{x_1} \cdots \frac{dx_{N+k+1}}{x_{N+k+1}}
	\]
	Using the Batyrev constraints $Q_i = \prod_{j=1}^n x_i^{m_{ij}}$, we can write the Mellin transform as
	\[
		\mathcal{MI}^{\textnormal{coh}}(z,p_1,p_2)
		=
		\int_{( \mathbb{R}_{>0})^n}
		\prod_{j=1}^{N+k+1} e^{-x_j/z} x_j^{\alpha_j(p)} 
		\frac{dx_1}{x_1} \cdots \frac{dx_{N+k+1}}{x_{N+k+1}}
		=
		z^{c_1(TX_{\mu,K})(p)}
		\Gamma(p_1)^N \prod_{j=0}^k \Gamma(p_2 - a_j p_1)
	\]
\end{proof}

\subsection{Proof of Lemma
  \ref{lemma:gamma_coh_inverse_mellin_residue}}\label{sssection:proof_lemma_inverse_mellin} 

To prove this lemma, the goal will be to obtain a contour deformation result to express the integrals along $p_l \in \varepsilon_l + i\mathbb{R}$ (where $l=1,2$) in terms of integrals along closed curves, for which we can then apply the residue theorem.
The main ingredients to prove our contour deformation result will be Stirling's formula and the Fano condition $N - \sum_{j=0}^k a_j > 0$.

\begin{remark}
	In general, the formula for the inverse Mellin transform relies on the choice of a base point $\varepsilon=(\varepsilon_1,\varepsilon_2) \in \mathbb{R}^2$, as in the formula
	\[
		\mathcal{M}^{-1} g(Q_1,Q_2)
		:=
		\int_{p_2 \, \in \, \varepsilon_2 + i \mathbb{R}}
		\int_{p_1 \, \in \, \varepsilon_1 + i \mathbb{R}}
		\Gamma(p_1)^N \prod_{j=0}^k \Gamma(p_2 - a_j p_1)
		Q_1^{-p_1} Q_2^{-p_2} dp_1 dp_2
	\]
	This point $\varepsilon$ could be chosen such that the input inside every gamma function is positive, i.e. for all $j \in \{0, \dots, n\}, m_{1j}\varepsilon_1 + \cdots + m_{rj}\varepsilon_r > 0$.
	If $\sigma \subset \mathbb{R}^2$ denotes the cone spanned by the columns of the moment map of $X_{\mu,K}$, this condition is equivalent to choosing a point $\varepsilon$ in the interior of the dual cone $\left(\sigma^\vee\right)^\circ$.
	This interior is not empty as $X_{\mu,K}$ is compact, c.f. Remark \ref{rmk:gamma_toric_manifold_dimension_dual_cone}.
\end{remark}

\begin{proof}[Proof of Lemma \ref{lemma:gamma_coh_inverse_mellin_residue}]
	We consider the inverse Mellin transform
	\[
		\mathcal{M}^{-1} g(Q_1,Q_2)
		:=
		\int_{p_2 \, \in \, \varepsilon_2 + i \mathbb{R}}
		\int_{p_1 \, \in \, \varepsilon_1 + i \mathbb{R}}
		\Gamma(p_1)^N \prod_{j=0}^k \Gamma(p_2 - a_j p_1)
		Q_1^{-p_1} Q_2^{-p_2} dp_1 dp_2,
	\]
	where positive numbers $\varepsilon_1, \varepsilon_2$ are chosen such that $0 < \textnormal{max}_j\{1,\alpha_j\} \varepsilon_1 < \varepsilon_2 < 1$.

	We begin by showing that
	\[
		\mathcal{M}^{-1}g(Q_1,Q_2)
		=
		\sum_{d_2 \in \mathbb{Z}_{\geq 0}} \textnormal{Res}_{p_2=-d_2}
		\sum_{d_1 \in \mathbb{Z}_{\geq 0}} \textnormal{Res}_{p_1=-d_1}
		\Gamma(p_1)^N \prod_{j=0}^k \Gamma(p_2 - a_j p_1)
		Q_1^{-p_1} Q_2^{-p_2}
	\]
	Let us begin by treating the first integral, with respect to the coordinate $p_1$. We will deform the integration contour $\varepsilon_1 + i \mathbb{R}$ by the following contour: pick $R_M > \! > 0$ some large number and define the closed contour $\mathcal{C}(R_M)$ as the union of the following pieces: a vectical line segment $L_1(R_M)$ from $(\varepsilon_1,-\sqrt{R_M^2-1})$ to $(\varepsilon_1,+\sqrt{R_M^2-1})$, a horizontal line $L_2(R_M)$ from $(\varepsilon_1,+\sqrt{R_M^2-1})$ to $(-1,+\sqrt{R_M^2-1})$, a circular arc $C(R_M)$ from $(-1,+\sqrt{R_M^2-1})$ to $(-1,-\sqrt{R_M^2-1})$ along the circle of radius $R_M$ and origin 0 in the half space $\{ \textnormal{Re}(p_1) < 0 \}$, and a horizontal line $L_3(R_M)$ from $(-1,-\sqrt{R_M^2-1})$ to $(\varepsilon_1,-\sqrt{R_M^2-1})$.


	Our goal is to show that the contributions to the integral of all parts except the vectical line segment $L_1(R_M)$ vanish when $R_M \to \infty$.
	We recall that for $a,b \in \mathbb{R}$ we have
	\[
		|\Gamma(a+ib)|^2
		=
		|\Gamma(a)|^2 \prod_{k=0}^\infty
		\frac{1}{1+\frac{b^2}{(a+k)^2}}
	\]
	Therefore, for $a$ fixed, the function $b \mapsto |\Gamma(a+ib)|$ has exponential decay as $b \to \pm \infty$.
	Thus, for the integrals along the horizontal lines,
	\[
		\lim_{R_M \to \infty}
		\int_{L_2(R_M)}
		\Gamma(p_1)^N \prod_{j=0}^k \Gamma(p_2 - a_j p_1)
		Q_1^{-p_1} dp_1
		=
		\lim_{R_M \to \infty}
		\int_{L_3(R_M)}
		\Gamma(p_1)^N \prod_{j=0}^k \Gamma(p_2 - a_j p_1)
		Q_1^{-p_1} dp_1
		=
		0
	\]
	Next, on the circular arc $C(R_M)$, we have $\textnormal{Re}(p_1), \varepsilon_2 < 1$, thus
	\[
		\left|
			\Gamma(p_1)^N \prod_{j=0}^k \Gamma(p_2 - a_j p_1)
		\right|
		\leq
		\left|
			\Gamma(p_1)^N \prod_{j=0}^k \Gamma(1 - a_j p_1)
		\right|
		=
		\left|
			\frac{\Gamma(p_1)^N}{\prod_{j=0}^k \Gamma(a_j p_1)}
		\right|
		\prod_{j=0}^k
		\left|
			\frac{\pi}{\sin(\pi a_j z)}
		\right|
	\]
	For $p_1=R_M e^{i\theta}, \textnormal{Arg}(p_1) \neq \pi$, we can use Stirling's formula to obtain an asymptotic of the right hand side. We obtain
	\begin{align*}
		&\left|
			\frac{\Gamma(p_1)^N}{\prod_{j=0}^k \Gamma(a_j p_1)}
		\right|
		\\
		&\sim_{R_M \to \infty}
		(\textnormal{Constant})
		R_M^{\frac{1-N}{2}}
		e^{R_M \cos(\theta)(-N+\sum_j a_j-a_j\log(a_j))}
		e^{
		(N-\sum_j a_j)(R_M \log(R_M) \cos(\theta)-R_M \theta \sin(\theta))
		}
	\end{align*}
	The leading term in this expression is $e^{\cos(\theta) (N-\sum_j a_j)(R_M \log(R_M))}$, and the coefficient $\cos(\theta) (N-\sum_j a_j)$ is negative on the circular arc $C(R_M)$ due to the Fano condition $N-\sum_j a_j >0$.
	Combining with $|\sin (\pi a_j z)| \sim e^{\pi a_j R_M |\sin(\theta)|}/2$, we get that the function $\left| \Gamma(p_1)^N \prod_{j=0}^k \Gamma(p_2 - a_j p_1) \right|$ has exponential decay as $R_M \to \infty$ on the circular arc $C(R_M)$.
	Finally, we get
	\begin{align*}
		\lim_{R_M \to \infty} \int_{\mathcal{C}(R_M)}
		\Gamma(p_1)^N \prod_{j=0}^k \Gamma(p_2 - a_j p_1)
		Q_1^{-p_1} dp_1
		&=
		\lim_{R_M \to \infty} \int_{L_1(R_M)}
		\Gamma(p_1)^N \prod_{j=0}^k \Gamma(p_2 - a_j p_1)
		Q_1^{-p_1} dp_1
		\\
		&=
		\int_{\varepsilon_1 + i \mathbb{R}}
		\Gamma(p_1)^N \prod_{j=0}^k \Gamma(p_2 - a_j p_1)
		Q_1^{-p_1} dp_1
	\end{align*}

	It remains to apply the residue theorem to the closed contour $\mathcal{C}(R_M)$.
	As a function of $p_1$, the integrand $
		\Gamma(p_1)^N \prod_{j=0}^k \Gamma(p_2 - a_j p_1)
		Q_1^{-p_1}
	$ has poles at $p_1 \in \mathbb{Z}_{\leq 0}$ and for $p_2 - a_j p_1 \in \mathbb{Z}_{\leq 0}$.
	Since $p_2$ takes values on $\varepsilon_2 + i \mathbb{R}$, and since $0 < \textnormal{max}_j\{1,\alpha_j\} \varepsilon_1 < \varepsilon_2 < 1$, the poles obtained from the condition $p_2 - a_j p_1 \in \mathbb{Z}_{\leq 0}$ lay outside of the integration contour $\mathcal{C}(R_M)$.
	Therefore, we have computed the first integral with respect to the input $p_1$:
	\[
		g(p_1,p_2):=\mathcal{M}^{-1}\mathcal{MI}^{\textnormal{coh}}(1,p_1,p_2)
		=
		\int_{\varepsilon_2 + i \mathbb{R}}
		\sum_{d_1 \geq 0}
		\textnormal{Res}_{p_1=-d_1}
		\Gamma(p_1)^N \prod_{j=0}^k \Gamma(p_2 - a_j p_1)
		Q_1^{-Q_1} Q_2^{-Q_2}
		d{p_2}
	\]
	The second integral with respect to the input $p_2$ can be computed using a similar contour deformation, proven by using Stirling's formula once more (the Fano condition will not appear there).

	To obtain the identity in the statement of this lemma, we do the change of variables (for $l=1,2$) $x_l:=p_l-d_l$ and use Fubini's theorem to permute the discrete sums and the residues to obtain
	\begin{equation}\label{eqn:gamma_mid_computation_inverse_mellin_transform_cohomology}
		\mathcal{M}^{-1} g(Q_1,Q_2)
		=
		\textnormal{Res}_{x_2=0} \textnormal{Res}_{x_1=0}
		\sum_{d_1,d_2 \geq 0}
		\Gamma(x_1-d_1)
		\prod_{j=0}^k
		\Gamma \left(x_2-a_j x_1 - (d_2 - a_j d_1) \right)
		Q_1^{-x_1+d_1} Q_2^{x_2+d_2}
	\end{equation}
	Then, using the difference equation satisfied by the gamma function, we obtain
	\[
		\Gamma \left(x_2-a_j x_1 - (d_2 - a_j d_1) \right)
		=
		\Gamma (x_2-a_j x_1)
		\frac{
			\prod_{r=-\infty}^0 (x_2 - a_j x_1 - r)
		}{
			\prod_{r=-\infty}^{d_2-a_j d_1} (x_2-a_j x_1 - r)
		}
	\]
	Applying this formula to every factor in the Equation \ref{eqn:gamma_mid_computation_inverse_mellin_transform_cohomology} above gives the formula given in the statement of Lemma \ref{lemma:gamma_coh_inverse_mellin_residue}.
\end{proof}
\section{Continuous oscillatory integral in quantum $K$-theory}\label{sec:continuous_oscillatory_in_qk}

In this appendix we study another model of an oscillatory integral in quantum $K$-theory, by using the usual (continuous) integral instead of the Jackson integral.
This model was introduced first in \cite{Givental:perm_mirror}.
For such an oscillatory integral, we are only able to prove the analogue of our theorem for projective spaces using our strategy with Mellin transforms. When the toric manifold has Picard rank above 1, the contour deformation does not seem to work.
Combined with the fact that in general, the Mellin transform of a $q$-difference equation is not a difference equation, we decided to move these results to an appendix.

\subsection{Oscillatory integral in quantum $K$-theory}

In this subsection, we will always consider $|q|<1$.

\begin{definition}[Oscillatory integral in quantum $K$-theory; \cite{Givental:perm_mirror}, Theorem 2]
  Let $q \in (0,1)$.
  Consider the $K$-theoretic Laudau--Ginzburg model associated to a toric manifold $X_{\mu, K}$.
  Fix $(Q_1, \dots, Q_r) \in B$ and let $\omega_{\pi^{-1}(Q)} \in \Lambda^r \left(T^* \pi^{-1}(Q_1,\dots,Q_r) \right)$ be a volume form on $\pi^{-1}(Q_1,\dots,Q_r)$.
  The (continuous) oscillatory integral $\mathcal{I}_c^{K-\textnormal{th}}$ is the function defined by
  \[
    \mathcal{I}_c^{K-\textnormal{th}}(q,Q)
    :=
    \int_{\Gamma \subset \pi^{-1}(Q_1,\dots,Q_r)}
    e^{W_Q}
    \omega_{\pi^{-1}(Q)},
  \]
  where $W_Q := W_{|\pi^{-1}(Q_1,\dots,Q_r)}$ and $\Gamma \subset \pi^{-1}(Q_1,\dots,Q_r)$ is a Lefschetz thimble.
\end{definition}
Note that in this definition and unlike Subsection \ref{ssection:q_gamma_structure_in_qk}, we will understand the function $e^{W_Q}$ as its analytical continuation given by the infinite product ($|q| < 1$)
\[
  \prod_{j=1}^n \frac{1}{(x_j;q)_\infty}
\]
To choose which Lefchetz thimble we will consider, we have to compare our oscillatory integrals in $K$-theory and in cohomology.
Recall that in cohomology, we were considering (for $z,Q_i, x_j> 0$)
\[
  \mathcal{I}^{\textnormal{coh}}(z,Q)
  =
  \int_{\pi^{-1}(Q) \cap (\mathbb{R}_{> 0})^n}
  \exp(-x_1/z - \cdots - x_n/z) \omega_{\pi^{-1}(Q)}
\]
One can define a $q$-analogue of the exponential by the following formula
\[
  e_q(x)
  :=
  \sum_{d \geq 0}
  x^d
  \prod_{l=0}^d \frac{1-q}{1-q^l}
\]
This function inherits its name from the observation that 
\begin{equation}\label{eqn:gamma_q_exponential_limit}
  \lim_{q \to 1} e_q(x) = \exp(x)
\end{equation}
Furthermore, we will pick a Lefchetz thimble $\Gamma_\mathbb{R}$ so that $e^{+W(x_1, \dots, x_n)}$ is a $q$-analogue of $e^{-x_1/z - \cdots - x_n/z}$, in the sense of the limit in Equation \ref{eqn:gamma_q_exponential_limit}.
Therefore, we will need to take a real Lefschetz thimble $\Gamma_\mathbb{R}$ for which the coordinates $(x_j)$ are negative.

%

\begin{definition}
  Let $(Q_1, \dots, Q_r) \in (\mathbb{R}_{>0})^r$, the corresponding \textit{real Lefschetz thimble} $\Gamma_\mathbb{R} \subseteq \pi^{-1}(Q)$ is given by the negative points
  \[
    \Gamma_\mathbb{R}
    :=
    \{
      (x_1,\dots,x_n) \in \pi^{-1}(Q_1,\dots,Q_r)
      \, | \,
      x_j < 0
    \}
  \]
\end{definition}

To match the signs in the Batyrev relations, we will replace $Q_i$ by $(-1)^{\textnormal{deg}} Q_i := (-1)^{\sum_j m_{ij}} Q_i$, and finally consider the oscillatory integral (defined for $|q| < 1, Q_i > 0$):

\begin{equation}\label{eqn:gamma_k_theory_real_oscillatory_integral}
  \mathcal{I}_c^{K-\textnormal{th}}(q,Q)
  =
  \int_{\Gamma_\mathbb{R} \subset \pi^{-1}((-1)^{\textnormal{deg}}Q_1,\dots,(-1)^{\textnormal{deg}}Q_r)}
    e^{W_{(-1)^{\textnormal{deg}}Q}}
    \frac{d{Q_1}}{Q_1}
    \cdots
    \frac{d{Q_n}}{Q_n}
\end{equation}
Note that if we were to consider $q > 1$, we would have to replace the analytical continuation of $W$ (given by $\frac{1}{(-x;q)_\infty}$) in the integral by the expression $(q^{-1}x;q^{-1})_\infty$.
Then, the integral would be immediately divergent.

\begin{proposition}[\cite{Givental:perm_mirror}, Theorem 2; see also \cite{IMT}, Proposition 2.12]
  The $K$-theoretical oscillatory integral $\mathcal{I}_c^{K-\textnormal{th}}$ and the small $I$-function $I_{X_{\mu,K}}^{K-\textnormal{th}}$ satisfy the same set of $q$-difference equations below (indexed by $i \in \{1, \dots, r\}$):
  \begin{align*}
    \left[
      \prod_{j : m_{ij} > 0} \prod_{r=0}^{m_{ij}-1}
      \left(
        1 - q^{-r} q^{
          \sum_{i'} m_{i'j} Q_{i'} \partial_{Q_{i'}}
        }
      \right)
      -
      Q_i
      \prod_{j : m_{ij} < 0} \prod_{r=0}^{-m_{ij}-1}
      \left(
        1 - q^{-r} q^{
          \sum_{i'} m_{i'j} Q_{i'} \partial_{Q_{i'}}
        }
      \right)
    \right]
    f_q(Q)
    =
    0
  \end{align*}
\end{proposition}

\subsection{Corresponding $q$-gamma class and comparison theorem}

Since we changed the definition of the oscillatory integral, it turns out we will be using another $q$-analogue of the gamma function.

\begin{definition}[Continuous $q$-gamma class]\label{def:gamma_q_gamma_class}
  Let $E \to X_{\mu,K}$ be a vector bundle, and denote by $\delta_1, \dots, \delta_m \in H^2(X_{\mu,K};\mathbb{Q})$ its Chern roots.
  The \textit{continuous $q$-gamma class} $\widehat{\gamma^c_q}(E) \in H^*(X_{\mu,K};\mathbb{Q})$ is defined by
  \[
    \widehat{\gamma^c_q} (E)
    :=
    \prod_{j=1}^m \delta_j \gamma_q^c(\delta_j)
    \in H^*(X_{\mu,K};\mathbb{Q}),
  \]
  where
  \begin{equation}\label{eqn:gamma_q_gamma}
    \gamma_q^c(z)
    :=
    \int_0^\infty
    \frac{x^z}{(-x;q)_\infty}
    \frac{dx}{x}
  \end{equation}
\end{definition}

In the definition of $\widehat{\gamma^c_q} (E)$, the right hand side should be understood as its power series expansion, using the Ramanujan identity below.

\begin{proposition}[\cite{Askey:Ramanujan_extension}]\label{prop:Ramanujan_formula}
  For $a \in \mathbb{C}, z > 0$ and $q \in (0,1)$, the following Ramanujan formula holds:
  \[
    \int_0^\infty
    t^z
    \frac{
      (-at;q)_\infty
    }
    {
      (-t;q)_\infty
    }
    \frac{dt}{t}
    =
    \frac{(a;q)_\infty}{(q;q)_\infty}
    \frac{\pi}{\sin(\pi z)}
    \frac{\left( q^{1-z} ; q\right)_\infty}{\left( aq^{1-z} ; q\right)_\infty}
  \]
\end{proposition}

Setting $a=0$ in this proposition, we get

\begin{equation*}
  \gamma_q^c(z) = \frac{\pi}{\sin(\pi z)} \frac{(q^{1-z};q)_\infty}{(q;q)_\infty}
\end{equation*}
Note that this function satisfies the difference equation
\begin{equation*}
  \gamma_q^c(z+1)
  =
  \frac{1}{q^{-z}-1} \gamma_q^c(z)
\end{equation*}

%
%

\begin{remark}\label{rmk:gamma_comparison_between_q_gamma_functions}
  The function $\gamma_q$ defined in Equation (\ref{eqn:gamma_q_gamma}) is another $q$-analogue of the gamma function $\Gamma$, as we replacing the exponential $e^{-x}$ by the $q$-analogue $e_q(-x/(1-q))=(-x;q)_\infty^{-1}$.
  Recall that "the" $q$-gamma function $\Gamma_q$ introduced by Jackson (see e.g. Equation (1.10.1) in \cite{Gasper_Rahmann:q_hypergeometric_series}) is defined by
  \[
    \Gamma_q(z)
    :=
    (1-q)^{1-z}
    \frac{(q;q)_\infty}{(q^z;q)_\infty}
  \]
  Our two $q$-analogues are related by the formula
  \[
    \gamma_q^c(z)
    =
    \frac{\pi}{\sin(\pi z)}
    \frac{(1-q)^z}{\Gamma_q(1-z)}
  \]
  Using Euler's reflection formula for the classical gamma function, we obtain that also 
  \begin{equation}\label{eqn:gamma_confluence_q_gamma}
    \lim_{q \to 1} (1-q)^{-z} \gamma_q^c(z) = \Gamma(z)
  \end{equation}
  Note that we introduced the factor $(1-q)^{-z}$ in the left hand side as otherwise the difference equation satisfied by $\gamma_q^c(z)$ would have no formal limit when $q \to 1$.
\end{remark}

We are now ready to state our theorem comparing the $q$-oscillatory integral and the $I$-function.

\begin{theorem}\label{thm:gamma_comparison_theorem_for_continuous_integrals}
  Let $X=\mathbb{P}^N$ be a projective space.
  Then, the oscillatory integral $\mathcal{I}_c^{K-\textnormal{th}}$ defined in Equation (\ref{eqn:gamma_k_theory_real_oscillatory_integral}) and the $I$-function $I^{K-\textnormal{th}}$ of Definition \ref{def:gamma_k_th_I_function} are related by the identity
  \[
    \mathcal{I}_c^{K-\textnormal{th}}(q,Q)
    =
    \int_{\left[X\right]}
    \textnormal{ch}_q\left(
      I_X^{K-\textnormal{th}}(q,Q)
    \right)
    \cup
    \widehat{\gamma^c_q}(TX),
  \]
  where $\int_{\left[X_{\mu,K}\right]}$ denotes the intersection product by $[X_{\mu,K}] \in H_*(X_{\mu,K};\mathbb{C})$, $\widehat{\gamma^c}(TX_{\mu,K})$ is the continuous $q$-Gamma class of Definition \ref{def:gamma_q_gamma_class} and $\textnormal{ch}_q$ is the $q$-Chern character of Definition \ref{def:gamma_q_chern}.
\end{theorem}

\subsection{Proof of the comparison theorem for projective planes}

In this case, we will begin by writing our proof as if the target manifold was a symplectic toric manifold of Picard rank 2.
We are able to compute the Mellin transform of the oscillatory integral, however trouble will appear when considering its inverse.

\begin{lemma}\label{lemma:gamma_kth_mellin_transform_computation}
  Let $X_{\mu,K}$ be a Fano symplectic toric manifold of Picard rank 2, whose moment map is given by (cf Proposition \ref{prop:gamma_model_toric_picard_rank_two})
  \[
    \begin{pmatrix}
      1
      &
      \cdots
      &
      1
      &
      0
      &
      -a_1
      &
      \cdots
      &
      -a_k
      \\
      0
      &
      \cdots
      &
      0
      &
      1
      &
      1
      &
      \cdots
      &
      1
    \end{pmatrix},
  \]
  where $a_0:=0, a_1,\dots,a_k \in \mathbb{Z}_{\geq 0}$.
  Let $
    \Gamma_\mathbb{R}
    :=
    \{
      (x_1,\dots,x_n) \in \pi^{-1}(Q_1,\dots,Q_r)
      \, | \,
      x_j < 0
    \}
  $ be the real Lefschetz thimble and $
    \mathcal{I}_c^{K-\textnormal{th}}(q,Q)
    =
    \int_{\Gamma_\mathbb{R}}
    e^{W_{(-1)^{\textnormal{deg}}Q}}
    \frac{d{Q_1}}{Q_1}
    \cdots
    \frac{d{Q_n}}{Q_n}
  $ be the oscillatory integral.
  Then, the Mellin transform $ \mathcal{MI}_c^{K-\textnormal{th}}$ of the oscillatory integral exists and is given by
  \[
    \mathcal{MI}_c^{K-\textnormal{th}}(q,p_1,p_2)
    =
    \gamma_q^c(p_1)^N \prod_{j=0}^k \gamma_q^c(p_2 - a_j p_1),
  \]
  where we recall that $
  \gamma_q^c(z)
    :=
    \frac{\pi}{\sin(\pi z)}
    \frac{(q^{1-z};q)_\infty}{(q;q)_\infty}
  $.
\end{lemma}

The proof will be similar to its analogue in cohomology, see Lemma \ref{lemma:gamma_coh_mellin_transform_computation}: our goal is to use Fubini's theorem to compute the Mellin transform as an integral on $Y_\mathbb{R} = \left(\mathbb{R}_{> 0}\right)^n$.

\begin{proof}
  We recall the diagram of the $K$-theoretic mirror family of Definition \ref{def:gamma_k_theoretic_mirror_family}:
  \[
    \begin{tikzcd}
      Y := \left( \mathbb{C}^* \right)^n
        \arrow[r,"W"]
        \arrow[d,"\pi"]
      &
      \mathbb{C}
      \\
      B := \left( \mathbb{C}^* \right)^r
    \end{tikzcd}
  \]
  We also recall that in the expression of the oscillatory integral, $W_Q$ (for a point $Q \in B$) designates the restriction of the map $W$ along the fibre $\pi^{-1}(Q)$, and that $\omega_{\pi^{-1}(Q)}$ designates a volume form on the same fibre $\pi^{-1}(Q)$.
  Using Fubini's theorem and the Batyrev constraints $(-1)^{\textnormal{deg}(Q_i)} Q_i = \prod_j x_j^{m_{ij}}$, we get
  \begin{align*}
    &\mathcal{MI}_c^{K-\textnormal{th}}(q,p_1,p_2)
    =
    \int_{Q_1,Q_2 > 0} \int_{\Gamma_\mathbb{R}}
    \exp \left (
      W_{(-1)^{\textnormal{deg}}Q}(x_1,\dots,x_n)
    \right)
    Q_1^{p_1} Q_2^{p_2}
    \omega_{\pi^{-1}(Q)}
    \frac{dQ_1}{Q_1}\frac{dQ_2}{Q_2}
    \\
    &=
    \int_{\{x_1, \dots, x_n < 0\}}
    \prod_{j=1}^n
    \frac{1}{(x_j;q)_\infty}
    \left(
      (-1)^{\textnormal{deg}(Q_1)} x_1^{m_{11}} \cdots x_n^{m_{1n}}
    \right)^{p_1}
    \left(
      (-1)^{\textnormal{deg}(Q_2)} x_1^{m_{21}} \cdots x_n^{m_{2n}}
    \right)^{p_2}
    \frac{dx_1}{x_1} \cdots \frac{dx_n}{x_n}
  \end{align*}
  We now do a change of variable and set $x_j'=-x_j$.
  After this change of variable, we obtain
  \begin{align*}
    \mathcal{MI}_c^{K-\textnormal{th}}(q,p_1,p_2)
    =
    \prod_{j=1}^n
    \int_{x_j' > 0}
    \frac{1}{(-x'_j;q)_\infty}
    x_j'^{\alpha_j(d)}
    \frac{dx'_j}{x'_j}
  \end{align*}
  Using the Ramanujan formula of Proposition \ref{prop:Ramanujan_formula} and the description of Picard rank 2 symplectic toric manifolds, we obtain a product of functions $\gamma_q$ as stated in the lemma.
\end{proof}


Now, our goal is to compute the inverse Mellin transform using a contour deformation result to apply the residue theorem.

\begin{conjecture}\label{conj:gamma_contour_deformation_q_gamma}
  The inverse Mellin transform of $\mathcal{MI}_c^{K-\textnormal{th}}(q,p_1,p_2)$ can be computed by the following sum of iterated residues:
  \[
    \mathcal{M}^{-1} \mathcal{MI}_c^{K-\textnormal{th}}(q,Q_1,Q_2)
    =
    \sum_{d_2 > 0} \textnormal{Res}_{p_2=-d_2} \sum_{d_1 > 0} \textnormal{Res}_{p_1=-d_1}
    \gamma_q^c(p_1)^N \prod_{j=0}^k \gamma_q^c(p_2 - a_j p_1)
    Q_1^{-d_1} Q_2^{-d_2} dQ_1 dQ_2
  \]
\end{conjecture}

\begin{proof}[Proof when $X = \mathbb{P}^N$]
  We recall that the comparison between the function $\gamma_q$ and the usual $q$-gamma function is given by (cf. Remark \ref{rmk:gamma_comparison_between_q_gamma_functions})
  \[
    \gamma_q^c(p)
    =
    \frac{\pi}{\sin(\pi p)}
    \frac{(1-q)^p}{\gamma_q^c(1-p)}
  \]
  Write $1-p = z = \rho e^{i \theta}$. If $\textnormal{Re}(p) < 0$, then $\theta \in \left(-\frac{\pi}{2},\frac{\pi}{2} \right)$ and we can use Moak's asympototic for the function $\Gamma_q$.
  When $\rho \to \infty$, we have
  \begin{align*}
    | \sin(\pi z) |
    &=
    | \sin(\pi \rho  e^{i \theta}) |
    \sim
    \frac{1}{2}
    e^{\rho |\sin(\theta)|}
    \\
    |(1-q)^p|
    &= e^{\log(1-q)(1-\rho \cos(\theta))}
    \\
    \textnormal{Li}_2(1-q^{z})
    &\to
    \frac{\pi^2}{6}
    \\
    \left|
      \left(
        \frac{1-q^z}{1-q}
      \right)^{z-\frac{1}{2}}
    \right|
    &\sim
    \frac{1}{
      (1-q)^{\rho \cos(\theta)-\frac{1}{2}}
    }
  \end{align*}
  Therefore,
  \[
    \left|
      \gamma_q^c(p)
    \right|
    \sim
    \textnormal{(Constant)}
    e^{- \pi \rho |\sin(\theta)|}
  \]

  In the case of $X=\mathbb{P}^N$, we have
  \[
    \mathcal{MI}_c^{K-\textnormal{th},X=\mathbb{P}^N}(q,p_1)
    =
    \gamma_q^c(p_1)^{N+1}
  \]
  In that case, we can compute the inverse Mellin transform of this expression through a contour deformation identical to the one used in the proof of Lemma \ref{lemma:gamma_coh_inverse_mellin_residue}:
  
  We want to compute
  \[
    \int_{p_1 \in \varepsilon_1+i\mathbb{R}}
    \gamma_q^c(p_1)^{N+1} Q_1^{-p_1}
    dQ_1
  \]
  We deform this integral using the following contour:
  we pick $R_M > \!> 0$ some large number and define the closed contour $\Gamma_{R_M}$ as the union of the following pieces: a vectical line segment $L_1(R_M)$ from $(\varepsilon_1,-R_M)$ to $(\varepsilon_1,+R_M)$, a horizontal line $L_2(R_M)$ from $(\varepsilon_1,R_M)$ to $(0,R_M)$,
  a half circle $C(R_M)$ from $(0,R_M)$ to $(0,-R_M)$ of radius $R_M$ and origin 0 in the half place $\{ \textnormal{Re}(p_1) < 0 \}$ ,
  and a horizontal line $L_3(R_M)$ from $(0,-R_M)$ to $(\varepsilon_1,-R_M)$.

  We focus on the integral
  \[
    \int_{p_1 \in C(R_M)}
    \gamma_q^c(p_1)^{N+1} Q_1^{-p_1}
    dQ_1
  \]
  When $|Q_1| < 0$, this integrand has exponential decay for $\textnormal{Re}(p_1) < 0$ as $|p_1| \to \infty$, therefore the integral along the arc of circle $C(R_M)$ will vanish at when $R_M \to \infty$.
  Applying the residue theorem to the deformed contour and using continuity for the vanishing of the other integrals along $L_{1,2}(R_M)$ as $\varepsilon_1 \to 0$, we obtain the formula of Conjecture \ref{conj:gamma_contour_deformation_q_gamma}.
\end{proof}

To obtain a proof of the comparison theorem \ref{thm:gamma_comparison_theorem_for_continuous_integrals}, it remains once again to identify the residue computed in Conjecture \ref{conj:gamma_contour_deformation_q_gamma} with a Jeffrey--Kirwan residue using Theorem 2.6 of \cite{Szenes_Vergne:jeffrey_kirwan_residues} and apply Proposition 2.3 of \cite{Szenes_Vergne:jeffrey_kirwan_residues}.

Unfortunately, when the target space is not a projective space $\mathbb{P}^N$, we have not been able to express the inverse Mellin transform of the oscillatory integral as a sum of residues yet.
We have attempted to find a formal continuation (for the coordinate $q$) of the Mellin transform $\mathcal{MI}_c^{K-\textnormal{th}}$ to $\{ q > 1 \}$ that satisfies the same difference equation, however our contour deformation strategy does not work for that function either.

%
%

\bibliographystyle{amsalpha}
\bibliography{Bibliography}


\end{document}